\documentclass[a4paper, 10pt, twoside, notitlepage]{amsart}

\usepackage[utf8]{inputenc}
\usepackage[dvipsnames]{xcolor}
\usepackage{amsmath} 
\usepackage{amssymb} 
\usepackage{amsthm}
\usepackage{geometry}
\usepackage{graphicx}
\usepackage{mathtools}
\usepackage{esint}
\usepackage[shortlabels]{enumitem}
\usepackage{comment}
\usepackage[colorlinks=true,linkcolor=blue]{hyperref}

\usepackage{tikz}
\usetikzlibrary{decorations.pathreplacing}
\usetikzlibrary{hobby}
\usetikzlibrary{calc}
\usetikzlibrary{math}

\theoremstyle{plain}
\newtheorem{thm}{Theorem}[section]
\newtheorem{cor}[thm]{Corollary}
\newtheorem{prop}[thm]{Proposition}
\newtheorem{lem}[thm]{Lemma}

\newtheorem{defi}[thm]{Definition}
\newtheorem{rmk}[thm]{Remark}

\newcommand {\R} {\mathbb{R}} 
 \newcommand {\N} {\mathbb{N}}
 
\newcommand {\p} {\partial}

\newcommand {\D} {\Delta}

\newcommand {\supp} {\text{supp}}

\newcommand\restr[2]{{
  \left.\kern-\nulldelimiterspace
  #1
  \littletaller
  \right|_{#2}
  }}
\newcommand{\littletaller}{\mathchoice{\vphantom{\big|}}{}{}{}}

\numberwithin{equation}{section}
\DeclareMathOperator {\dist} {dist}

\DeclareMathOperator{\Id} {Id}

\title[Stability Nonlocal-Local Reduction]{Transfer of Stability from the Classical to the Fractional Anisotropic Calder\'on Problem}

\author[H. Baers]{Hendrik Baers}
\address{Institute for Applied Mathematics, University of Bonn, Endenicher Allee 60, 53115 Bonn, Germany}
\email{hendrik.baers@uni-bonn.de}

\author[A. R\"uland]{Angkana R\"uland}
\address{Institute for Applied Mathematics and Hausdorff Center for Mathematics, University of Bonn, Endenicher Allee 60, 53115 Bonn, Germany}
\email{rueland@uni-bonn.de}

\begin{document}

\begin{abstract}
We discuss two spectral fractional anisotropic Calder\'on problems with source-to-solution measurements and their quantitative relation to the classical Calder\'on problem. Firstly, we consider the anistropic fractional Calder\'on problem from \cite{FGKU21}. In this setting, we quantify the relation between the local and nonlocal Calder\'on problems which had been deduced in \cite{R23} and provide an associated stability estimate. As a consequence, any stability result which holds on the level of the local problem with source-to-solution data has a direct nonlocal analogue (up to a logarithmic loss). Secondly, we introduce and discuss the fractional Calder\'on problem with source-to-solution measurements for the spectral fractional Dirichlet Laplacian on open, bounded, connected, Lipschitz sets on $\R^n$. Also in this context, we provide a qualitative and quantitative transfer of uniqueness from the local to the nonlocal setting. As a consequence, we infer the first stability results for the principal part for a fractional Calder\'on type problem for which no reduction of Liouville type is known. Our arguments rely on quantitative unique continuation arguments. As a result of independent interest, we also prove a quantitative relation between source-to-solution and Dirichlet-to-Neumann measurements for the classical Calder\'on problem.
\end{abstract}

\maketitle

\section{Introduction}
\label{sec:intro}

We consider the quantitative transfer of uniqueness from the local to the nonlocal setting for two inverse problems of Calder\'on type: On the one hand, we study the fractional Calder\'on problem on a closed, connected, smooth Riemannian manifold as introduced and analyzed in the seminal works \cite{F21} and \cite{FGKU21}. On the other hand, we investigate the closely related Calder\'on type problem for the spectral fractional Dirichlet Laplacian on a bounded, Lipschitz domain. In both settings we consider source-to-solution data. It is our objective to prove stability properties for the reduction of the respective fractional problem to the associated local Calder\'on problem. This provides the quantitative analogue of the qualitative correspondence between the nonlocal and local problems from \cite{R23} as well as a similar qualitative and quantitative relation for the spectral fractional problem.

\subsection{Transfer of stability for the fractional Calder\'on problem on closed, connected, smooth Riemannian manifolds}

We begin by outlining the setting and results for the fractional Calderón problem on closed, connected, smooth Riemannian manifolds.

\subsubsection{Set-up}
Let us begin by recalling the set-up of the anisotropic fractional Calder\'on problem with source-to-solution data. Let $(M,g)$ be a closed, connected, smooth $n$-dimensional Riemannian manifold. Let $(-\D_g)$ denote the Laplace-Beltrami operator on $(M,g)$. Let $O \subset M$ be an open set, assume that $g|_O$ is known and let $s\in (0,1)$. As in \cite{FGKU21} we are given the source-to-solution data 
\begin{align}\label{eq:def_nonlocal_StoSOperator_manifold}
L_{s,O}: \widetilde{H}^{-s}_{\diamond}(O):=\{f \in \widetilde{H}^{-s}(O): \ (f,1)_{L^2(M)}=0\} \rightarrow H^s(O), \quad f \mapsto \restr{u^f}{O},
\end{align}
where $u^f$ is the unique solution with vanishing mean value of the equation
\begin{align*}
(-\D_g)^s u^f = f \mbox{ in } M,
\end{align*}
and $(\cdot, \cdot)_{L^2(M)}$ denotes the $L^2$ inner product on $M$ (with respect to the volume form $d V_g$).  With slight abuse of notation, we also often use this notation for the corresponding duality pairing obtained by extension.
Here, the equation is defined spectrally. More precisely, considering an eigenbasis of $L^2$-orthonormal eigenfunctions $(\phi_k)_{k \in \N_0}$ of the Laplace-Beltrami operator and $(\lambda_k)_{k \in \N_0}$ their associated ordered  eigenvalues, i.e., $(-\D_g) \phi_k = \lambda_k \phi_k$, with $0=\lambda_0 < \lambda_1 \leq \lambda_2 \leq \cdots $, we have that for $u \in C^{\infty}(M)$
\begin{align*}
(-\D_g)^s u := \sum\limits_{k=1}^{\infty} \lambda_k^{s} (u, \phi_k)_{L^2(M )} \phi_k.
\end{align*}
In \cite{FGKU21} it is proved that the source-to-solution data from \eqref{eq:def_nonlocal_StoSOperator_manifold} allow one to uniquely recover the manifold and metric $(M,g)$. Let us emphasize, that as proven in \cite{CR25}, the condition that $g|_{O}$ is assumed to be known, does not impose a serious restriction on either the results of  \cite{FGKU21} or of the present work. On the contrary, by virtue of the results from \cite{CR25}, the given data suffice to reconstruct the metric in $O$ in a Lipschitz stable way.

Connected to the uniqueness result from \cite{FGKU21}, in \cite{R23}, using a variable-coefficient extension interpretation, it is shown that the knowledge of the fractional source-to-solution map $L_{s,O}$ allows one to recover the local source-to-solution map $L_{1,O}$. This map is defined as the operator
\begin{align}\label{eq:def_local_StoSOperator_manifold}
L_{1,O}: \widetilde{H}^{-1}_{\diamond}(O):=\{f \in \widetilde{H}^{-1}(O): \ (f,1)_{L^2(M)}=0\} \rightarrow H^1(O), \ f \mapsto v^f|_{O},
\end{align}
where the function $v^f$ is the unique solution of
\begin{align*}
(-\D_g) v^f = f \mbox{ in } M
\end{align*}
with vanishing mean value. More precisely, in \cite{R23} it is shown that there exists a constant $c_s = \frac{2^{2s-1}\Gamma(s)}{\Gamma(1-s)}$ such that for $x \in O$
\begin{align}\label{eq:from_nonlocal_to_local_manifold}
L_{1,O}(f)(x) = c_s \int_0^\infty t^{1-2s} \tilde{u}^f(x,t)dt,
\end{align}
where $\tilde{u}^f$ is the unique $\dot{H}^1(M \times \R_+, x_{n+1}^{1-2s})$-regular solution with vanishing tangential mean value of the associated variable coefficient Caffarelli-Silvestre extension type problem. More precisely, building on the seminal observations from \cite{CS07} in the constant and from \cite{ST10} in the variable coefficient setting, we have that $\tilde{u}^f \in \dot{H}^1(M \times \R_+, x_{n+1}^{1-2s})$ is a solution of
\begin{equation}\label{eq:CS}
\begin{cases}
\begin{alignedat}{2}
( x_{n+1}^{1-2s} (-\Delta_g) - \p_{n+1}x_{n+1}^{1-2s} \p_{n+1} ) \tilde{u}^f & = 0 \quad &&\text{in } M \times \R_+,\\
- \bar{c}_s \lim_{x_{n+1}\to0} x_{n+1}^{1-2s} \partial_{n+1} \tilde{u}^f & = f \quad &&\text{on } M \times \{0\},
\end{alignedat}
\end{cases}
\end{equation}
such that $(\tilde{u}^f(\cdot, x_{n+1}), 1)_{L^2(M)}=0$ for all $x_{n+1}\geq 0$. Here, $\bar{c}_s = \frac{4^s \Gamma(s)}{2s \vert \Gamma(-s) \vert}>0$.

In this article, it is our objective to prove \emph{stability} in the recovery of the local source-to-solution map $L_{1,O}$ from the nonlocal source-to-solution map $L_{s,O}$ under suitable a priori bounds (on the metric and data).

\subsubsection{Main results: Quantitative transfer of uniqueness from the local to the nonlocal inverse problem}

In what follows we discuss our main results on the quantitative transfer of information from the local to the nonlocal context. In particular, this will allow us to deduce stability properties for the nonlocal problem whenever the associated local problem enjoys a stability estimate. While our stability transfer is accompanied with a logarithmic loss, it provides the first stability results for the principal parts in fractional Calder\'on problems for which no direct Liouville transform is known (contrary to the setting in, for instance, \cite{C20}).

We will always suppose that the manifold $M$ as an open set is fixed and known and only consider the metric $g$ as variable and to be determined. Additionally, we impose the following a-priori assumptions on the underlying Riemannian manifolds.\\

\textbf{Assumption (A1):} Firstly, we assume that the metric $g$ is smooth and satisfies the following uniform bounds: There exists a constant $\theta_1\in(0,1)$ such that 
\begin{align}\label{eq:A1}
0 < \theta_1 \leq g \leq \theta_1^{-1} < \infty.
\end{align}

This is a mild condition ensuring uniform bounds for the metrics in the class of Riemannian manifolds under consideration. As a consequence of this assumption, there exists a uniform constant $\bar{\lambda} \in (0,1)$ which only depends on $\theta_1$ and the manifold $M$ such that the first non-zero eigenvalue $\lambda_1((M,g))>0$ of the Laplace-Beltrami operator on $(M,g)$ is bounded from above and below by
\begin{align}\label{eq:uniform_bounds_first_eigenvalue}
0 < \bar{\lambda} < \lambda_1((M,g)) < \bar{\lambda}^{-1} < \infty.
\end{align}
This can be seen by comparison with a fixed Laplace-Beltrami operator associated with a fixed metric $\bar{g}$.
\\

\textbf{Assumption (A2):} Secondly, we impose a uniform constraint on the gradient of the metric $g$. More precisely, we assume that there exists $\theta_2 > 0$ such that
\begin{align}\label{eq:A2}
\vert \nabla g \vert \leq \theta_2.
\end{align}

This condition, together with assumption (A1), guarantees the equivalence of the $H^s(O)$-norms for $O \subset M$, $s\in[-2,2]$, if taken with respect to different metrics $g_1$ and $g_2$, i.e. there exist constants $c,C>0$ (depending on $\theta_1$ and $\theta_2$, but otherwise independent of the metrics $g_1$ and $g_2$) such that
\begin{align*}
c \Vert u \Vert_{H^s(O;dV_{g_1})} \leq \Vert u \Vert_{H^s(O;dV_{g_2})} \leq C \Vert u \Vert_{H^s(O;dV_{g_1})}.
\end{align*}
Here $dV_g$ denotes the Riemannian volume element with respect to the metric $g$. For more details on this, see Section \ref{ssec:prel_equivalence_of_norms}. Due to this equivalence result, in the following, we will opt out from specifying the metric with respect to which we consider the norm.

Finally, we impose a third more technical assumption.\\

\textbf{Assumption (A3):} Thirdly, we impose a flatness condition on the measurement domain $O$. We assume that $O$ is a known, flat subset of $M$ (i.e. $R_{ijkl} = 0$ where $R_{ijkl}$ denotes the Riemannian curvature tensor). In particular, we assume that in local coordinates the Laplace-Beltrami operator turns into the Euclidean constant-coefficient Laplacian $\Delta$ on $O$.

This condition is of technical nature and we expect that it is possible to remove it. It primarily allows us to invoke known unique continuation results for the Caffarelli-Silvestre extension with constant coefficients. While we anticipate that these estimates remain valid for sufficiently regular, non-constant coefficients, for clarity of exposition, we do not deal with these more technical aspects in the present article, but postpone these to future work. In Assumption (A3), without loss of generality, we may always assume that $O$ is covered by a single coordinate patch and is a smooth domain. Indeed, if this is not the case, we can always decrease the size of the domain $O$ (at the price of reducing the available measurement data).\\

Under the outlined assumptions, we will prove the following main results regarding the single-measurement and the infinitely-many-measurements settings.

\begin{thm}[Transfer of single measurement stability]
\label{thm:stability_nonlocal_local_single_meas}
Let $(M,g_1)$ and $(M,g_2)$ be two closed, connected, smooth $n$-dimensional Riemannian manifolds satisfying the conditions (A1) and (A2).
Let $O', O \subset M$ be smooth, open, non-empty sets such that $O$ satisfies assumption (A3) and such that $O' \Subset O$.
Let $L_{s,O}$ and $L_{1,O}$ be the source-to-solution maps from \eqref{eq:def_nonlocal_StoSOperator_manifold} and \eqref{eq:def_local_StoSOperator_manifold}, respectively. Let $f \in \widetilde{H}^{-s}_{\diamond}(O)$. 

Then there exist constants $\beta>0$ and $C>0$, such that if for some $\varepsilon \in (0,\frac{1}{2})$
\begin{align*}
\|L_{s,O}^1(f) - L_{s,O}^2(f)\|_{H^{s}(O)} \leq \varepsilon \|f\|_{H^{-s}(M)},
\end{align*}
then it holds
\begin{align*}
\Vert L_{1,O'}^1(f) - L_{1,O'}^2(f) \Vert_{H^1(O')} \leq C \vert \log(\varepsilon) \vert^{-\beta} \Vert f \Vert_{H^{-1}(M)}.
\end{align*}
Here, the constants $\beta$ and $C$ only depend on $s$, $n$, $M$, $O'$, $O$, $\theta_1$ and $\theta_2$ but they are independent of the measurement data $f$.
\end{thm}
 
The above single measurement stability result immediately also entails an infinite-measurement stability result.

\begin{thm}[Transfer of infinite measurement stability]
\label{thm:stability_nonlocal_local_infinite_meas}
Let $(M,g_1)$ and $(M,g_2)$ be two closed, connected, smooth $n$-dimensional Riemannian manifolds satisfying the conditions (A1) and (A2).
Let $O',O \subset M$ be smooth, open, non-empty sets such that $O$ satisfies assumption (A3) and such that $O' \Subset O$.
Let $L_{s,O}$ and $L_{1,O}$ be the source-to-solution maps from \eqref{eq:def_nonlocal_StoSOperator_manifold} and \eqref{eq:def_local_StoSOperator_manifold}, respectively.

Then there exist constants $\beta>0$ and $C>0$, such that if for some $\varepsilon \in (0,\frac{1}{2})$
\begin{align*}
\|L_{s,O}^1 - L_{s,O}^2 \|_{\widetilde{H}^{-s}_{\diamond}(O) \to H^{s}(O)} \leq \varepsilon,
\end{align*}
then it holds
\begin{align*}
\Vert L_{1,O'}^1 - L_{1,O'}^2 \Vert_{\widetilde{H}^{-1}_{\diamond}(O') \to H^1(O')} \leq C \vert \log(\varepsilon) \vert^{-\beta},
\end{align*}
The constants $\beta$ and $C$ only depend on $s$, $n$, $M$, $O'$, $O$, $\theta_1$ and $\theta_2$.
\end{thm}

As a consequence, whenever stability properties are known for the local Calder\'on problem with source-to-solution measurements, this then also implies stability for the analogous nonlocal problem.  We state this as the following conditional stability result.

\begin{cor}[Conditional transfer of stability: from local estimates to nonlocal ones]\label{cor:transfer_of_stability_manifold_setting}
Let $(M,g_1)$ and $(M,g_2)$ be two closed, connected, smooth $n$-dimensional Riemannian manifolds satisfying the conditions (A1) and (A2).
Let $O',O \subset M$ be smooth, open, non-empty sets such that $O$ satisfies assumption (A3) and such that $O' \Subset O$.
Let $L_{s,O}^j$ and $L_{1,O}^j$, $j\in\{1,2\}$, be the source-to-solution maps from \eqref{eq:def_nonlocal_StoSOperator_manifold} and \eqref{eq:def_local_StoSOperator_manifold}, respectively.
Assume that there exists some modulus of continuity $\omega: (0,\varepsilon_0) \rightarrow (0,\infty)$ such that it holds
\begin{align}
\label{eq:assump_cor}
\|g_1 - g_2 \|_{X} \leq \omega(\Vert L_{1,O'}^1 - L_{1,O'}^2 \Vert_{\widetilde{H}^{-1}_{\diamond}(O') \to H^1(O')}) 
\end{align}
for some suitable normed space $X$.
Then, there exist constants $C>0$ and $\beta>0$ such that
\begin{align*}
\|g_1 - g_2 \|_{X} \leq \omega \left( C \vert \log( \Vert L_{s,O}^1 - L_{s,O}^2 \Vert_{\widetilde{H}^{-s}_{\diamond}(O) \to H^s(O)}) \vert^{-\beta} \right).
\end{align*}
\end{cor}

We remark that we are not aware of an estimate of the generality and strength of \eqref{eq:assump_cor} which is valid for general metrics and manifolds. Similar transfer results are however also applicable for more restricted classes of metrics (e.g., for suitable conformal factors etc). In the spectral setting outlined in the next section, we will illustrate an application of a result of this type in the isotropic setting (see Corollary \ref{cor:spectral_stability}).

\subsection{Transfer of uniqueness and stability for the fractional Calder\'on problem for the spectral fractional Laplacian} \label{sec:Intro_Spectral_StoSCalderon}

The arguments from the previous section can be generalized to variations of the above set-up. We discuss the Calder\'on problem for the fractional Dirichlet spectral Laplacian as one such example. We refer to \cite{BSV15,CS16,G16,AD16} and the references therein for properties of the spectral fractional Laplacian.

\subsubsection{Set-up}
Let us begin by outlining the setting of the Calder\'on problem for the Dirichlet spectral fractional Laplacian. The connection between an extension perspective and the spectral definition can be found in \cite{CS16}. Let us recall the set-up for this problem.

Consider $\Omega \subset \R^n$ open, bounded, Lipschitz and connected and let $a \in L^\infty(\Omega, \R^{n\times n})$ be symmetric, uniformly elliptic and bounded. Let $L_a:= (-\nabla \cdot a \nabla ): H^{1}_0(\Omega) \rightarrow H^{-1}(\Omega)$ denote a variable coefficient Laplacian on $\Omega$ complemented with vanishing Dirichlet data. Then, this operator can be diagonalized with a sequence of diverging eigenvalues $(\lambda_{k})_{k\in \N}$ with $0<\lambda_1 \leq \lambda_2 \leq \cdots$ and associated eigenfunctions $(\phi_k)_{k \in \N}$, i.e.,
\begin{equation*}
\begin{cases}
\begin{alignedat}{2}
L_a \phi_k &= \lambda_k \phi_k \quad &&\text{in } \Omega,\\
\phi_k & = 0 \quad &&\text{on } \partial \Omega.
\end{alignedat}
\end{cases}
\end{equation*}
Associated with this operator, we define the Dirichlet spectral fractional Laplacian by setting for $s \in (0,1)$ and $u \in C_c^{\infty}(\Omega)$
\begin{align}
\label{eq:frac_Lapl_spec}
L_a^s u:= \sum\limits_{j=1}^{\infty} \lambda_j^{s} (u,\phi_k)_{L^2(\Omega)} \phi_k,
\end{align}
which converges in $L^2(\Omega)$.

As above, for this operator, we consider an inverse problem with source-to-solution data:
Let $A \Subset \Omega$ be open, bounded, Lipschitz such that $\Omega \setminus \overline{A}$ is connected. The goal is to determine the metric $a$ in $A$. On $\Omega \setminus A$ we assume that the metric is known. Let $O \Subset \Omega \setminus \overline{A}$ open, Lipschitz, be our measurement set (cf. Figure \ref{fig:setting_spectral_inverse_problem}). We define the source-to-solution operator
\begin{align}\label{eq:def_nonlocal_StoSOperator_spectral}
\tilde{L}_{s,O}: \widetilde{\mathcal{H}}^{-s}(O) \to \mathcal{H}^{s}(O), \ f \mapsto \restr{u^f}{O},
\end{align}
where $u^f \in \mathcal{H}^{s}(\Omega)$ is the unique solution of 
\begin{equation}\label{eq:nonlocal_spectral}
\begin{cases}
\begin{alignedat}{2}
L_a^s u^f &= f \quad &&\text{in } \Omega,\\
u^f & = 0 \quad &&\text{on } \partial\Omega,
\end{alignedat}
\end{cases}
\end{equation}
in the associated energy space (we refer to Section \ref{sec:spaces_frac_Dirichlet} for details on the function spaces, the formulation of the PDE and the meaning of the imposed boundary conditions).
This problem is well-posed (see for instance \cite{CS16}). In what follows below, we study the associated inverse problem, i.e., the question whether the knowledge of $\tilde{L}_{s,O}$ allows us to recover the unknown metric $a$ in $A$.\\

In parallel to the nonlocal source-to-solution problem, we also study the local source-to-solution problem: Let $O \Subset \Omega \setminus \overline{A}$ be as before and define
\begin{align}\label{eq:def_local_StoSOperator_spectral}
\tilde{L}_{1,O}: \widetilde{H}^{-1}(O) \to H^{1}(O), \ f \mapsto \restr{v^f}{O},
\end{align}
where $v^f \in H_0^{1}(\Omega)$ is the unique solution of
\begin{equation}\label{eq:local_spectral}
\begin{cases}
\begin{alignedat}{2}
L_a v^f &= f \quad &&\text{in } \Omega,\\
v^f & = 0 \quad &&\text{on }\partial \Omega.
\end{alignedat}
\end{cases}
\end{equation}

As shown in \cite{CS16}, it is possible to realize the operator $L_a^{s}$ by means of an extension perspective: For $u \in C_c^{\infty}(\Omega)$ it holds for $\bar{c}_s = \frac{4^s\Gamma(s)}{2s\vert\Gamma(-s)\vert}$ that
\begin{align*}
L_a^s u (x) = -\bar{c}_s \lim\limits_{x_{n+1} \rightarrow 0} x_{n+1}^{1-2s} \p_{n+1} \tilde{u}(x,x_{n+1}),
\end{align*}
where $\tilde{u}(x,x_{n+1})$ is the unique solution of
\begin{equation}\label{eq:extension_spectral}
\begin{cases}
\begin{alignedat}{2}
(x_{n+1}^{1-2s} L_a - \p_{n+1} x_{n+1}^{1-2s} \p_{n+1})  \tilde{u} &= 0 \quad &&\text{in } \Omega \times \R_+,\\
\tilde{u} & = 0 \quad &&\text{on } \partial \Omega \times \R_+,\\
\tilde{u} & = u \quad &&\text{on } \Omega \times \{0\},
\end{alignedat}
\end{cases}
\end{equation}
in the energy space $H^{1}(\Omega\times \R_+, x_{n+1}^{1-2s})$ (see Section \ref{sec:preliminaries} and, in particular, Section \ref{sec:spaces_frac_Dirichlet} for the definition of this function space).
We now turn to our main results connecting the nonlocal and the local spectral inverse problems.

\begin{figure}
    	\begin{tikzpicture}[closed hobby, scale=0.5]
    	\draw[thick] plot coordinates {(1,1.5) (0.5,3) (3,6) (9,7.2) (14,6) (16,3) (13,0.2) (10,0.7)};
    	\draw[thick] plot coordinates {(3.3,3.2) (4.5,5.5) (7,6) (8.3,4.5) (6.5,2.5)};
    	\draw[thick] plot coordinates {(11,3) (12.2,4.2) (14.3, 3.6) (13.3,1.7)};
    	\node at (14,7) {\huge{$\Omega$}};
    	\node at (5.8,5) {\huge{$A$}};
    	\node at (12.8,3.4) {\huge{$O$}};
    	\node[blue] at (5.8,3.7) {\large{$a = \ ?$}};
    	\node[blue] at (11,5.5) {\large{$a = \Id_{n \times n}$}};
    	\end{tikzpicture}
\caption{Schematic set-up for our formulation of the (fractional) Calderón problem for the spectral (fractional) Laplacian.}
\label{fig:setting_spectral_inverse_problem}
\end{figure}
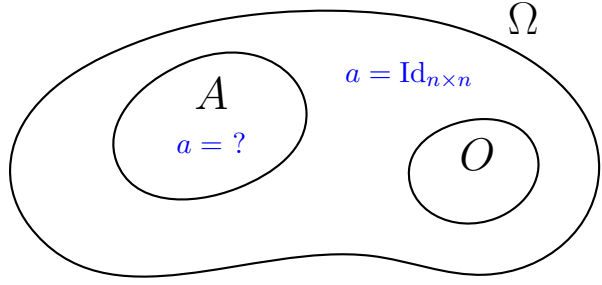

\subsubsection{Main results: Qualitative and quantitative transfer of uniqueness from the local to the nonlocal inverse problem}

Let $\Omega, A, O \subset \R^n$ be known open, bounded and connected Lipschitz sets such that $A \Subset \Omega$ and $O \Subset \Omega \setminus \overline{A}$ with $\Omega \setminus \overline{A}$ being connected (see Figure \ref{fig:setting_spectral_inverse_problem}).

Complementary to the assumptions in the manifold setting, in this subsection, we impose the following a-priori conditions:\\

\textbf{Assumption (A1').}
There exists $\theta_1 \in (0,1)$ such that for all admissible metrics $a \in C^1(\Omega, \R^{n \times n}_{sym})$ it holds 
\begin{align*}
\theta_1 |\xi|^2 \leq \xi \cdot a \xi \leq \theta_1^{-1} |\xi|^2, \mbox{ for all } \xi \in \R^n \setminus \{0\}.
\end{align*}

As a consequence, by comparison with the Laplacian, the assumption (A1') immediately entails estimates on the first eigenvalue. Moreover, as we also always assume that $a$ is symmetric, (A1') provides uniform $L^{\infty}$ bounds on the metrics.
\\

\textbf{Assumption (A2').}
There exists $\theta_2>0$ such that for all admissible metrics $a \in C^1(\Omega, \R^{n \times n}_{sym})$ it holds 
\begin{align*}
\vert \nabla a \vert \leq \theta_2.
\end{align*}

As before, this assumption together with (A1') guarantees the equivalence of the $\mathcal{H}^s(O)$-norms for $O \subset \Omega$, $s\in[-2,2]$, if taken with respect to different metrics $a_1$ and $a_2$ (for more details on this see Section \ref{ssec:prel_equivalence_of_norms}).
\\

\textbf{Assumption (A3').}
The metric $a \in C^1(\Omega,\R^{n \times n})$ satisfies $\restr{a}{\Omega \setminus A} = \Id_{n \times n}$.

As in the setting from the previous section, also in the bounded domain case, we expect that this assumption is a technical condition. We anticipate that with additional technical effort (in the quantitative unique continuation results) it should be possible to remove this assumption.
\\

With the extension perspective, \eqref{eq:extension_spectral}, in hand, analogously as in \cite{R23}, it follows that the nonlocal source-to-solution operator uniquely determines its local counterpart.

\begin{thm}[Density]
\label{thm:qualitative_spectral}
Let $n\geq 2$.
Let $\Omega, A, O \subset \R^n$ be as above and let $a \in C^1(\Omega, \R^{n \times n}_{sym})$ be a metric satisfying the assumptions (A1'), (A2') and (A3'). Let $\tilde{L}_{s,O}$, $\tilde{L}_{1,O}$ be the source-to-solution operators as in \eqref{eq:def_nonlocal_StoSOperator_spectral} and \eqref{eq:def_local_StoSOperator_spectral}, respectively.
Then the operator $\tilde{L}_{s,O}$ determines the operator $\tilde{L}_{1,O}$. More precisely, for $c_s = \frac{2^{2s-1}\Gamma(s)}{\Gamma(1-s)}$ the set
\begin{align*}
\widetilde{\mathcal{C}}_1 &:= \left\{ \Big(f, \ \restr{c_s\int\limits_{0}^{\infty} t^{1-2s} \tilde{u}^f(\cdot,t) dt}{O} \Big): \ \text{$\tilde{u}^f$ is the extension of $u^f$ as in \eqref{eq:extension_spectral}}\right\} \\
& \subset C_c^{\infty}(O) \times \mathcal{H}^{1}(O),
\end{align*}
where $u^f$ is the unique solution to \eqref{eq:nonlocal_spectral} with data $f \in C_c^{\infty}(O)$, is a dense subset of 
\begin{align*}
\mathcal{C}_1 := \left\{ (f, \restr{v^f}{O}): \ v^f \text{ is a solution to \eqref{eq:local_spectral} with data $f \in C_c^{\infty}(O) $}  \right\} \subset \widetilde{\mathcal{H}}^{-1}(O) \times \mathcal{H}^{1}(O),
\end{align*}
with respect to the $\widetilde{\mathcal{H}}^{-1}(O) \times \mathcal{H}^{1}(O)$ topology.
\end{thm}

\begin{rmk}
We remark that as in \cite{R23}, in the density result, it would suffice to assume that $a$ is symmetric and bounded, not necessarily differentiable. For consistency with the above assumptions, we have opted to assume the higher regularity condition also already at this point.
\end{rmk}

In particular, Theorem \ref{thm:qualitative_spectral} implies that any uniqueness result on the local level transfers to a uniqueness result on the nonlocal level. As an example we consider the isotropic setting. More precisely,
for coefficients $a = \gamma \Id_{n \times n}$ with $\gamma \in C^{1}(\overline{\Omega})$ and $0<c\leq \gamma \leq C<\infty$, we deduce a qualitative determination of the coefficients in the nonlocal spectral source-to-solution problem under suitable geometric assumptions on $\Omega$, $A$ and $O$.

\begin{cor}[Uniqueness]
\label{cor:spectral_uniqueness}
Let $n\geq 3$. Let $\Omega, A, O \subset \R^n$ be as above and let $a_j = \gamma_j\Id_{n \times n}$, $j\in\{1,2\}$, with $\gamma_j \in C^{1}(\overline{\Omega})$ be metrics satisfying the assumptions (A1'), (A2') and (A3'). Let $\tilde{L}_{s,O}^j$, $\tilde{L}_{1,O}^j$ be the corresponding source-to-solution operators as in \eqref{eq:def_nonlocal_StoSOperator_spectral} and \eqref{eq:def_local_StoSOperator_spectral}, respectively.
Then, 
\begin{align*}
\tilde{L}_{s,O}^1 = \tilde{L}_{s,O}^2,
\end{align*}
implies that $\gamma_1 = \gamma_2$.
\end{cor}

As in the previous sections, as a main result of this article, we observe that this uniqueness result can be upgraded to a \emph{quantitative} transfer of information from the local to the nonlocal context.

\begin{thm}[Transfer of stability in the single measurement setting]
\label{thm:quantitative_spectral}
Let $n\geq 2$.
Let $\Omega, A, O \subset \R^n$ be given as above and let $a_j \in C^1(\Omega, \R^{n \times n}_{sym})$, $j\in\{1,2\}$, be metrics satisfying the assumptions (A1'), (A2') and (A3'). Let $O' \Subset O$ be Lipschitz. Let $\tilde{L}_{s,O}^j$, $\tilde{L}_{1,O}^j$ be the corresponding source-to-solution operators as in \eqref{eq:def_nonlocal_StoSOperator_spectral} and \eqref{eq:def_local_StoSOperator_spectral}, respectively, and let $f \in \widetilde{\mathcal{H}}^{-s}(O)$.

There exist constants $\beta > 0$ and $C>0$, such that, if for some $\varepsilon \in (0,\frac{1}{2})$
\begin{align*}
\|\tilde{L}_{s,O}^1 (f) - \tilde{L}_{s,O}^2 (f) \|_{\mathcal{H}^{s}(O)} \leq \varepsilon \Vert f \Vert_{\mathcal{H}^{-s}(\Omega)},
\end{align*}
then it holds that
\begin{align*}
\Vert \tilde{L}_{1,O}^1(f) - \tilde{L}_{1,O}^2(f) \Vert_{H^{1}(O')} \leq C \vert \log(\varepsilon) \vert^{-\beta} \Vert f \Vert_{H^{-1}(\Omega)}.
\end{align*}
The constants $\beta$ and $C$ depend on $s$, $n$, $\Omega$, $A$, $O'$, $O$, $\theta_1$ and $\theta_2$, both are however independent of $f$.
\end{thm}

This result, hence, allows one to transfer any quantitative uniqueness result from the local setting to a quantitative uniqueness result in the nonlocal setting.

In order to illustrate its applicability, as a corollary of Theorem \ref{thm:quantitative_spectral}, for isotropic coefficients, we have the following stability transfer. 

\begin{cor}[Stability for isotropic metrics]
\label{cor:spectral_stability}
Let $\Omega, A, O \subset \R^n$, $n \geq 3$, be given as above and $a_j = \gamma_j\Id_{n \times n}$, $j\in\{1,2\}$, with $\gamma_j \in H^{t+2}(\Omega)\cap C^1(\Omega)$, $t>\frac{n}{2}$, be metrics satisfying the assumptions (A1'), (A2') and (A3'). Let $\tilde{L}_{s,O}^j$ be the corresponding source-to-solution operators as in \eqref{eq:def_nonlocal_StoSOperator_spectral}. Assume that
\begin{align*}
\Vert \gamma_j \Vert_{H^{t+2}(\Omega)} \leq M < \infty.
\end{align*}
Then there exists  $\omega: (0,\frac{1}{2}) \rightarrow [0,\infty)$ with $\omega(t)\leq C|\log(|\log(t)|)|^{-\mu}$ for some constants $C,\mu>0$ such that
\begin{align*}
\|\gamma_1 - \gamma_2\|_{L^{\infty}(\Omega)} \leq \omega(\|\tilde{L}_{s,O}^1 - \tilde{L}_{s,O}^2\|_{\widetilde{\mathcal{H}}^{-s}(O) \rightarrow \mathcal{H}^{s}(O)} ), 
\end{align*}
if $\|\tilde{L}_{s,O}^1 - \tilde{L}_{s,O}^2\|_{\widetilde{\mathcal{H}}^{-s}(O) \rightarrow \mathcal{H}^{s}(O)} \in [0,\frac{1}{2})$.
\end{cor}

We remark that the double logarithm is a consequence of the logarithmic stability estimate for local source-to-solution Calder\'on problems, see Theorem \ref{thm:stability_local_StoSCalderon} below, combined with the logarithmic loss in Theorem \ref{thm:quantitative_spectral}. This is the first instance of a stability result for the principal part of the operator in fractional Calder\'on type problems in settings in which no Liouville transform is known for the nonlocal problem.

\subsection{On a quantitative relation between the Dirichlet-to-Neumann and source-to-solution operators for the classical Calder\'on problem}
Let us emphasize that in proving Corollary \ref{cor:spectral_stability}, we discuss a quantitative relation between the classical Calder\'on problem with source-to-solution and with Dirichlet-to-Neumann measurements which we view as of independent interest. To this end, we also introduce the slightly extended set $A_+$, a bounded Lipschitz set with $A \Subset A_+ \Subset \Omega$ such that $\Omega\setminus\overline{A_+}$ is connected and such that $O \Subset \Omega \setminus \overline{A_+}$.

\begin{prop}\label{prop:reduction_StoS_DtN}
Let $\Omega \subset \R^n$, $n\geq2$, $A \Subset A_+ \Subset \Omega$ and $O \Subset \Omega \setminus \overline{A_+}$ be given as above. Let $a_j \in L^\infty(\Omega , \R^{n \times n}_{sym})$, $j\in\{1,2\}$, with $a_1 = a_2$ in $\Omega \setminus A$ satisfy the assumption (A1'). 
Let $\tilde{L}_{1,O}^j$ and $\Lambda_{a_{j},A_+}$, $j\in\{1,2\}$, be the local source-to-solution operators or Dirichlet-to-Neummann maps as in \eqref{eq:def_StoS_operator_conductivity} and \eqref{eq:def_DN_map_conductivity}, respectively.
Then, on the one hand, there exists a constant $C>0$ such that
\begin{align*}
\Vert \tilde{L}_{1,O}^1 - \tilde{L}_{1,O}^2 \Vert_{\widetilde{H}^{-1}(O) \to H^1(O)} \leq C \Vert \Lambda_{a_1,A_+} - \Lambda_{a_2,A_+} \Vert_{H^{\frac{1}{2}}(\partial A_+) \to H^{-\frac{1}{2}}(\partial A_+)},
\end{align*}
and, on the other hand, there exist constants $C>0$ and $\gamma \in (0,1)$ such that
\begin{align*}
\Vert \Lambda_{a_1,A_+} - \Lambda_{a_2,A_+} \Vert_{H^{\frac{1}{2}}(\partial A_+) \to H^{-\frac{1}{2}}(\partial A_+)} \leq C \Vert \tilde{L}_{1,O}^1 - \tilde{L}_{1,O}^2 \Vert_{\widetilde{H}^{-1}(O) \to H^1(O)}^{\gamma}.
\end{align*}
\end{prop}

Proposition \ref{prop:reduction_StoS_DtN} hence essentially shows the quantitative equivalence of source-to-solution and Dirichlet-to-Neumann measurements (up to an algebraic loss).
While this result is well-known in its qualitative form, we did not find a quantitative version in the literature and hence provide an independent proof of it in Section \ref{sec:stability_local_StoSCalderon_Schrödinger}. It relies on a quantification of the corresponding qualitative result by deriving a suitable quantitative Runge approximation property.

\begin{rmk}
We remark that the result of Proposition \ref{prop:reduction_StoS_DtN}, in particular, even holds in the anisotropic case. Consequently, if either for the source-to-solution or the Dirichlet-to-Neumann Calderón problem a stability estimate for anisotropic metrics is known, then the same stability estimate (up to at most an algebraic loss) also holds for the respective other problem.
\end{rmk}

\subsection{Main ideas and outline of the argument}\label{sec:sketch_of_proof}

Let us outline the key ideas for the above results. We focus on the manifold setting, the setting of the spectral Dirichlet problem is analogous. We consider the variable coefficient Caffarelli-Silvestre type extension interpretation of the fractional Laplacian (\cite{CS07, ST10}). More precisely, the map $L_{s,O}$ can be realized in the following extension interpretation:
Consider $f \in C^{\infty}_c(O)$ with vanishing mean value and the following extension problem
\begin{equation}\label{eq:CS2}
\begin{cases}
\begin{alignedat}{2}
(x_{n+1}^{1-2s}(-\D_g) - \p_{n+1} x_{n+1}^{1-2s} \p_{n+1}) \tilde{u}^f & = 0 \quad &&\text{in } M \times \R_+,\\
-\bar{c}_s \lim\limits_{x_{n+1} \rightarrow 0} x_{n+1}^{1-2s} \p_{n+1} \tilde{u}^f & = f \quad &&\text{on } M \times \{0\}.
\end{alignedat}
\end{cases}
\end{equation}
Here $\tilde{u}^f \in \dot{H}^1(M \times \R_+, x_{n+1}^{1-2s})$ satisfies a vanishing tangential mean condition for every $x_{n+1}>0$.
Then, the source-to-solution map $L_{s,O}$ takes the form
\begin{align*}
L_{s,O}(f) = \tilde{u}^f(\cdot,0).
\end{align*}
In \cite{R23} it is shown that 
\begin{align*}
L_{1,O}(f)(x) =  c_s\int\limits_{0}^{\infty} t^{1-2s} \tilde{u}^f(x,t) dt =: v(x)
\end{align*}
with a constant $c_s = \frac{2^{2s-1}\Gamma(s)}{\Gamma(1-s)} > 0$.
Hence, the data $(f,L_{s,O}(f))$ can be viewed as the Neumann and Dirichlet data of \eqref{eq:CS2} and with $g|_O= \Id_{n \times n}$. By unique continuation and since $g|_O = \Id_{n \times n}$, it is possible to recover $\tilde{u}^f|_{O \times \R_+}$ in a unique and constructive way from the data $(f, L_{s,O}(f))$. Using the identity $L_{1,O}(f)(x) = c_s\int\limits_{0}^{\infty} t^{1-2s} \tilde{u}^f(x,t) dt$, it is then possible to uniquely recover $L_{1,O}$. 

It is the main objective of this article to quantify this, i.e., assuming that for the nonlocal data $f \in \widetilde{H}^{-s}_{\diamond}(O)$ we have
\begin{align*}
\Vert L_{s,O}^1(f) - L_{s,O}^2(f) \Vert_{H^s(O)} \leq \varepsilon \Vert f \Vert_{H^{-s}(M)}
\end{align*}
with $\varepsilon \in (0,\varepsilon_0)$, we seek to estimate with an appropriate modulus of continuity $\omega(t)$
\begin{align*}
\Vert L_{1,O}^1(f) - L_{1,O}^2(f) \Vert_{H^1(O)} = c_s\Vert \int_0^\infty t^{1-2s} (\tilde{u}_1^f - \tilde{u}_2^f)(\cdot,t) dt \Vert_{H^1(O)} \leq \omega(\varepsilon) \Vert f \Vert_{H^{-1}(M)},
\end{align*}
where $\tilde{u}_j^f$ are the solutions to \eqref{eq:CS2} with metrics $g_j$ and data $f$.

To this end, we will choose a sufficiently large cut-off height $L=L(\varepsilon)$ and use decay properties of the solutions to \eqref{eq:CS2} to deduce
\begin{align*}
\Vert \int_L^\infty t^{1-2s} \tilde{u}^f(\cdot,t) dt \Vert_{H^1(O)} \leq \omega(\varepsilon) \Vert f \Vert_{H^{-1}(M)}.
\end{align*}
It then suffices to estimate the finite height integral
\begin{align*}
\Vert \int_0^L t^{1-2s} (\tilde{u}_1^f - \tilde{u}_2^f)(\cdot,t) dt \Vert_{H^1(O)}.
\end{align*}
In order to do this, we note that the smallness condition for the nonlocal problem corresponds to smallness in the boundary data for the Caffarelli-Silvestre extension $\tilde{u}_1^f - \tilde{u}_2^f$. Using a quantitative boundary-bulk unique continuation result (see Proposition \ref{prop:bbucp}), we transfer the smallness in the boundary data into the bulk $O\times\R_+$. More precisely, we infer the smallness for half-balls placed at the boundary $O\times\{0\}$. Then, we propagate the smallness onto $O\times(0,L)$ by a typical three-balls-inequality argument (see Proposition \ref{prop:3ballsinequ}) and in that way we derive smallness of the full finite height integral.

Let us stress that the estimates for the upper integral and for the finite height integral are derived by two very different mechanisms. The first one heavily relies on the precise solution representation for the Caffarelli-Silvestre type extension equation, \eqref{eq:CS2}, whereas the second one just works on the level of the equation and uses rather robust propagation of smallness arguments.

\subsection{Relation to the literature}

Since the introduction of the fractional Calder\'on problem in \cite{GSU20}, the study of nonlocal Calder\'on type problems has developed into an extremely active field of research. 

Genuinely nonlocal effects which have been extracted include partial data uniqueness \cite{GSU20, RS20} and stability results \cite{RS20,RS18,BCR24}, single measurement uniqueness properties \cite{GRSU20, R21} and unique reconstruction of potentials in the presence of known anisotropic background metrics \cite{GLX17}. Moreover, monotonicity methods for inversion \cite{HL19, HL20}, nonlinear problems \cite{LL22}, conductivity formulations \cite{C20}, (complex) geometric optics solutions \cite{CdHS24, CR21} have been investigated. 
We further refer to \cite{CLR20, BGU21, RS20a, CdHS23,KRZ23} for related nonlocal problems.

In addition to the study of lower order contributions, building on the seminal works \cite{F21} and \cite{FGKU21}, a recent focus in the literature has been on the reconstruction of the principal part of the operator. In particular, as a major novelty, in \cite{FGKU21} it was proved that it is possible to uniquely identify smooth anisotropic metrics up to the natural gauges -- a problem which is widely open in the local context. This has been extended into various geometric contexts, see for instance \cite{QU22, C23, R23, FKU24, FL24}, and the whole space setting in \cite{FGKRSU25}. These interior uniqueness results have recently also been complemented with boundary reconstruction results \cite{CR25}. A second direction in the context of anisotropic problems which has attracted substantial research activity is the relation between local and nonlocal Calder\'on type problems. In \cite{GU21,CGRU23,R23} it was proved that it is possible to transfer uniqueness from a local Calder\'on problem to the nonlocal setting, both in the context of the source-to-solution data measurements and in the setting of exterior data measurements. This connection was also used in \cite{LNZ24}.

In the present article, we build on the qualitative connections between the local and nonlocal Calder\'on type problems and seek to illustrate that these connections are robust and allow one not only to transfer uniqueness but even stability. To this end, we build on the source-to-solution setting from \cite{R23}. In addition, with the spectral Dirichlet fractional Calder\'on problem, we introduce a new nonlocal inverse problem of Calder\'on type which we can address by the transfer of uniqueness ideas from this article.

\subsection{Outline of the article}

The remainder of the article is organized as follows. In Section \ref{sec:preliminaries} we introduce some preliminary notation and recall some relevant known results. Section \ref{sec:proof_manifold} contains the proofs of our main results in the closed manifold setting. In particular, we lay out in detail the propagation of smallness argument described in Section \ref{sec:sketch_of_proof}. The following section, Section \ref{sec:proof_spectral}, then deals with the proof of the main results for the fractional Calderón problem for the spectral fractional Laplacian on Euclidean domains. A major part of this section is about deriving a stability estimate for the local Calderón problem with source-to-solution data. In particular, we provide the proof of the quantitative relation between the source-to-solution and Dirichlet-to-Neumann measurements for the local Calderón problem, Proposition \ref{prop:reduction_StoS_DtN}. With this in hand we then infer the stability estimate of Corollary \ref{cor:spectral_stability}. In Section \ref{sec:instability} we comment on the optimality of our reduction argument. Here we discuss instability of the fractional (spectral) source-to-solution Calderón problem and find that the best modulus of continuity we can hope for is of logarithmic type. Combined with our positive results, we hence infer that the optimal stability of the fractional (spectral) Calderón problem is between logarithmic and double-logarithmic.

\section{Preliminaries}\label{sec:preliminaries}

We begin by introducing some relevant notation and auxiliary results for our analysis.

\subsection{Laplace-Beltrami operator}\label{ssec:prel_Laplace_Beltrami}

We start by recalling the representation of the Laplace-Beltrami operator in local coordinates. Let $(M,g)$ be a smooth $n$-dimensional Riemannian manifold. In local coordinates the Laplace-Beltrami operator $\Delta_g$ on $(M,g)$ then reads
\begin{align*}
\Delta_g = \frac{1}{\sqrt{\vert\det(g)\vert}} \nabla\cdot \left( \sqrt{\vert\det(g)\vert} g^{-1} \nabla \right).
\end{align*}

\subsection{Sobolev spaces}\label{ssec:prel_Sobolev_spaces}

At this point we now introduce the for our purposes most relevant function spaces. We consider two settings.
On the one hand, for $(M,g)$ being a closed, connected, smooth $n$-dimensional Riemannian manifold and for $s\in\R$ we define the fractional Sobolev space $H^s(M)$ by
\begin{align*}
H^s(M) := \{ u:M \to \R : \ \Vert u \Vert_{H^s(M)}^2 := \sum_{k=0}^\infty (1+\lambda_k)^{s} \vert (u,\phi_k)_{L^2(M)} \vert^2 < \infty \},
\end{align*}
where $(\phi_k)_{k\in\N}$ and $(\lambda_k)_{k\in\N}$ are an $L^2$-normalized eigenbasis and the corresponding eigenvalues of the Laplace-Beltrami operator $(-\Delta_g)$ on $(M,g)$. We will comment on the dependence of the $H^s(M)$-norm on $g$ in Section \ref{ssec:prel_equivalence_of_norms}. On the other hand, on $\R^n$ we define the fractional Sobolev space $H^s(\R^n)$ by
\begin{align*}
H^s(\R^n) := \{ u \in \mathcal{S}'(\R^n): \ \Vert u \Vert_{H^s(\R^n)} := \Vert (1+\vert\cdot\vert^2)^{s/2} \hat{u}(\cdot) \Vert_{L^2(\R^n)} < \infty \}.
\end{align*}

Now, let $(M,g)$ either be a closed, connected, smooth $n$-dimensional manifold or $(M,g) = (\R^n,\Id)$, and let $O \subset M$ be open, bounded, Lipschitz. We define the fractional Sobolev space on the bounded domain $O$ as the quotient space
\begin{align*}
H^s(O) := \{ \restr{u}{O}: \ u \in H^s(M) \},
\end{align*}
to which we associate the quotient norm $\Vert u \Vert_{H^s(O)} := \inf\{ \Vert U \Vert_{H^s(M)}: \ \restr{U}{O} = u \}$. Moreover, we define the spaces
\begin{align*}
\widetilde{H}^s(O) := \overline{C_c^\infty(O)}^{H^s(M)}, \quad H_{\overline{O}}^s := \{u \in H^s(M): \ \supp(u) \subseteq \overline{O}\},
\end{align*}
and we recall that $\widetilde{H}^s(O) = H_{\overline{O}}^s$ holds for Lipschitz sets $O$ and all $s\in \R$ (see \cite[Theorem 3.29]{McLean}).

Further, for $(M,g)$ a closed, connected, smooth $n$-dimensional manifold or $(M,g)=(\R^n,\Id)$, $s\in(0,1)$ and $\widetilde{\Omega} \subset M \times \R_+$ open, Lipschitz we define the weighted Lebesgue and Sobolev spaces

\begin{align*}
L^2(\widetilde{\Omega}, x_{n+1}^{1-2s}) &:= \{ \tilde{u}: \widetilde{\Omega} \to \R: \ \Vert x_{n+1}^{\frac{1-2s}{2}} \tilde{u} \Vert_{L^2(\widetilde{\Omega})} < \infty \},\\
H^1(\widetilde{\Omega}, x_{n+1}^{1-2s}) &:= \mbox{closure of } C_c^{\infty}(\overline{\widetilde{\Omega}}) \mbox{ with respect to the norm } \\
& \qquad \|\tilde{u}\|_{H^1(\widetilde{\Omega}, x_{n+1}^{1-2s})}:= \Vert x_{n+1}^{\frac{1-2s}{2}} \tilde{u} \Vert_{L^2(\widetilde{\Omega})} + \ \Vert x_{n+1}^{\frac{1-2s}{2}} \nabla_{\tilde{g}} \tilde{u} \Vert_{L^2(\widetilde{\Omega})} ,\\
\dot{H}^1(\widetilde{\Omega}, x_{n+1}^{1-2s}) &:= \mbox{closure of } C_c^{\infty}(\overline{\widetilde{\Omega}}) \mbox{ with respect to the semi-norm } \\
& \qquad \|\tilde{u}\|_{\dot{H}^1(\widetilde{\Omega}, x_{n+1}^{1-2s})}:=  \Vert x_{n+1}^{\frac{1-2s}{2}} \nabla_{\tilde{g}} \tilde{u} \Vert_{L^2(\widetilde{\Omega})} .
\end{align*}
Here $C_c^{\infty}(\overline{\widetilde{\Omega}}):=\{ u \in C^{\infty}(\overline{\widetilde{\Omega}}) \mbox{ with } \supp(u) \mbox{ compact in } \overline{\widetilde{\Omega}} \}$.
In particular, these functions may have non-zero trace on $\widetilde{\Omega} \cap \{x_{n+1}=0\}$.
Here, we consider the gradient $\nabla_{g}$ on a manifold $(M,g)$ in terms of the local coordinates as $\nabla_g = g^{-1} \nabla$ and the extension $\tilde{g}$ of $g$ is given by $\tilde{g} = \begin{pmatrix} g & 0\\ 0 & 1 \end{pmatrix}$. The gradient $\nabla_{\tilde{g}}$ is defined analogously. Note that in the Euclidean setting we have $\nabla_{\tilde{g}} = \nabla$.

We further remark that the trace operator $T: H^1(M \times \R_+, x_{n+1}^{1-2s}) \to H^s(\R^n)$ such that $\tilde{u}(\cdot,x_{n+1}) \to T\tilde{u}$ in $L^2(M)$ as $x_{n+1} \to 0$ is well-defined (see for example Lemma 4.4 in \cite{RS20}).

\subsection{Function spaces and problem formulation for the spectral Dirichlet fractional Laplacian}\label{sec:spaces_frac_Dirichlet}

Here and in what follows, $\Omega \subset \R^n$ will be assumed to be an open, bounded, Lipschitz and connected set.

The spectral Dirichlet fractional Laplacian is naturally connected with the following spectrally defined Sobolev spaces. Let $(\phi_k)_{k \in \N}$ and $(\lambda_k)_{k \in \N}$ denote an $L^2(\Omega)$-orthonormal family of ordered eigenfunctions and associated eigenvalues of the Dirichlet Laplacian on $\Omega$ with $0<\lambda_1 \leq \dots \leq \lambda_k \leq \cdots$, i.e.,
\begin{equation*}
\begin{cases}
\begin{alignedat}{2}
(-\D) \phi_k & = \lambda_k \phi_k \quad &&\text{in } \Omega,\\
\phi_k & = 0 \quad &&\text{on } \partial \Omega.
\end{alignedat}
\end{cases}
\end{equation*}
For $r\in \R$, we then define the spaces
\begin{align*}
\mathcal{H}^r(\Omega) := \{u \in \mathcal{D}'(\Omega): \ \sum\limits_{k=1}^{\infty} \lambda_k^r |(u,\phi_k)_{L^2(\Omega)}|^2< \infty \}.
\end{align*}
Here, with slight abuse of notation, for $r< 0$, $(u,\phi_k)_{L^2(\Omega)}$ is interpreted in a duality sense.
The spectral Dirichlet fractional Laplacian $L_a^s$ as defined in \eqref{eq:frac_Lapl_spec} is then well-defined in the space $\mathcal{H}^s(\Omega)$ with values in $\mathcal{H}^{-s}(\Omega)$. We comment on the dependence of the $\mathcal{H}^s$-norm on the choice of the eigenfunction basis and its corresponding eigenvalues in the next section, Section \ref{ssec:prel_equivalence_of_norms}.

By construction, we have that $\mathcal{H}^1(\Omega) = H^1_0(\Omega)$, by duality, further $\mathcal{H}^{-1}(\Omega) = (H^1_0(\Omega))^{\ast} = H^{-1}(\Omega)$.
By virtue of the results of \cite[Sections 2.1, 2.3]{CS16}, it further holds that 
\begin{align*}
\mathcal{H}^s(\Omega) = H^s,
\end{align*}
where 
\begin{align*}
H^s:= \left\{ \begin{array}{ll}
H^s(\Omega), \ s\in (0,1/2),\\
H^{1/2}_{00}(\Omega), \ s= \frac{1}{2},\\
H^s_0(\Omega), \ s \in (1/2,1).
\end{array}
\right.
\end{align*}
The spaces $H^s(\Omega)$ and $H^{s}_0(\Omega)$ are defined as the closures of $C_c^{\infty}(\Omega)$ with respect to the norm
\begin{align}
\label{eq:frac_norm}
\|u\|_{H^s}:= \|u\|_{L^2(\Omega)} + \left( \int\limits_{\Omega}\int\limits_{\Omega} \frac{|u(x)-u(y)|^2}{|x-y|^{n+2s}} dx dy \right)^{1/2},
\end{align}
and $H^{1/2}_{00}(\Omega)$ denotes the Lions-Magenes space (see \cite{LM12})
\begin{align*}
H^{1/2}_{00}(\Omega):=\left\{ u \in H^{\frac{1}{2}}(\Omega): \ \int\limits_{\Omega} \dist(x,\partial \Omega)^{-1} |u(x)|^2 dx < \infty  \right\}.
\end{align*}
It is endowed with the norm
\begin{align*}
\|u\|_{H^{\frac{1}{2}}_{00}(\Omega)}:= \|u\|_{H^{\frac{1}{2}}} + \|\dist(\cdot,\partial \Omega)^{-1/2} u\|_{L^2(\Omega)},
\end{align*}
where $\|u\|_{H^{\frac{1}{2}}} $ is defined as in \eqref{eq:frac_norm}.

We also refer to \cite[Section 7]{BSV15} for a discussion of these function spaces.

Similarly to the setting for the Sobolev spaces on the full underlying domain, for $O \subset \Omega$, we define $\mathcal{H}^s(O)$ as the quotient space
\begin{align*}
\mathcal{H}^s(O) := \{ \restr{u}{O}: u \in \mathcal{H}^s(\Omega) \}
\end{align*}
and we associate to it the quotient norm $\Vert u \Vert_{\mathcal{H}^s(O)} := \inf \{ \Vert U \Vert_{H^s(\Omega)} : \ \restr{U}{O} = u \}$. Additionally, we define the space
\begin{align*}
\widetilde{\mathcal{H}}^s(O) := \overline{C_c^\infty(O)}^{\mathcal{H}^s(\Omega)}.
\end{align*}

\subsection{Equivalence of norms under the a-priori assumptions}\label{ssec:prel_equivalence_of_norms}

Let $(M,g)$ be a closed, connected, smooth $n$-dimensional Riemannian manifold and let $O \subset M$. In general, for $s\in\R\setminus\N$ the $H^s(O;dV_g)$-norm on $O$ is a nonlocal norm, i.e. it heavily depends on the metric $g$ on all of $M$. However, since we seek to prove estimates uniform in the metric $g$, we need to avoid this dependency. We argue that for metrics $g$ satisfying our a-priori assumptions (A1) and (A2), all respective relevant, i.e. $s\in[-2,2]$, norms are uniformly equivalent with equivalence constants only depending on $\theta_1$ and $\theta_2$, but otherwise independent of the metric $g$. We first prove that the local $L^2(M;dV_g)$-, $H^1(M;dV_g)$- and $H^2(M;dV_g)$-norms are equivalent. Then by interpolation, all the $H^s(M;dV_g)$-norms are equivalent for $s \in (0,2)$ and consequently also the $H^s(O;dV_g)$-norms are equivalent. We use duality to prove the same for the negative Sobolev spaces $\widetilde{H}^{-s}(O;dV_g)$. Moreover, we prove the same dependence of the equivalence constants for the interpolation equality (in the sense of equivalent norms) of $H^s(O;dV_g) = (H^{s_0}(O;dV_g), H^{s_1}(O;dV_g))_{t,2}$ and $\widetilde{H}^s(O;dV_g) = (\widetilde{H}^{s_0}(O;dV_g), \widetilde{H}^{s_1}(O;dV_g))_{t,2}$, where $s_0,s_1 \in [-2,2]$ and $s=ts_0+(1-t)s_1$.

Let $(\eta_j)_{j\in\{1,\dots,N\}}$ be a partition of unity on $M$ with an associated pulled back open covering $(U_j)_{j\in\{1,\dots,N\}}$, $U_j \subset \R^n$. It then holds for these sets $U_j \subset \R^n$
\begin{align*}
\Vert u \Vert_{L^2(M;dV_g)}^2 = \int_M u^2 dV_g = \sum_{j=1}^N \int_{U_j} \eta_j u^2 \sqrt{\vert \det g \vert} dx,
\end{align*}
where $dV_g$ denotes the Riemannian volume element with respect to the metric $g$. Invoking the assumption (A1), we infer that $c \leq \sqrt{\vert \det g \vert} \leq C$ for some constants $c,C>0$ depending on $\theta_1$, and, thus, the $L^2$-norms for two metrics $g$ and $\bar{g}$ satisfying assumption (A1) are equivalent, i.e.
\begin{align*}
c\Vert u \Vert_{L^2(M;dV_{\bar{g}})} \leq \Vert u \Vert_{L^2(M;dV_g)} \leq C \Vert u \Vert_{L^2(M;dV_{\bar{g}})}
\end{align*}
for constants $c>0$ and $C>0$ depending on $\theta_1$, but otherwise independent of $g$ and $\bar{g}$.

Similarly, for the $H^1(M;dV_g)$-norm we observe
\begin{align*}
\Vert \nabla_g u \Vert_{L^2(M;dV_g)}^2 = \int_M \nabla_g u \cdot \nabla_g u \ dV_g = \sum_{j=1}^N \int_{U_j} \eta_j (g^{-1} \nabla u) \cdot (g^{-1} \nabla u) \sqrt{\vert \det g \vert} dx.
\end{align*}
Again, applying the ellipticity and boundedness assumption (A1) together with the equivalence of the $L^2$-norms yields
\begin{align}\label{eq:norm_equivalence_1}
c \sum_{j=1}^N \int_{U_j} \eta_j \left( \vert u \vert^2 + \vert \nabla u \vert^2 \right) dx \leq \Vert u \Vert_{H^1(M;dV_g)}^2 \leq C \sum_{j=1}^N \int_{U_j} \eta_j \left( \vert u \vert^2 + \vert \nabla u \vert^2 \right) dx,
\end{align}
which implies the equivalence of the $H^1(M;dV_g)$-norms with the claimed dependencies of the equivalence constants.

Lastly, for the equivalence of the $H^2(M;dV_g)$-norms we will additionally invoke the assumption (A2). For this, observe that
\begin{equation}\label{eq:equivalence_H^2-norms_1}
\begin{aligned}
\Vert \nabla_g^2 u \Vert_{L^2(M;dV_g)}^2 &= \int_M \vert \nabla_g^2 u \vert^2 dV_g = \sum_{j=1}^N \int_{U_j} \eta_j \big\vert [ g^{-1} \nabla ]^2 u \big\vert^2 \sqrt{\vert \det g \vert} dx\\
&= \sum_{j=1}^N \int_{U_j} \eta_j \sqrt{\vert \det g \vert} \left\vert (g^{-1})^2 \nabla^2 u + g^{-1} (\nabla g^{-1}) \nabla u \right\vert^2 dx
\end{aligned}
\end{equation}
Then, on the one hand, the boundedness assumptions (A1) and (A2) imply
\begin{align*}
\Vert \nabla_g^2 u \Vert_{L^2(M;dV_g)}^2 &\leq C \sum_{j=1}^N \int_{U_j} \eta_j \left( \vert \nabla^2 u \vert^2 + \vert \nabla u \vert \vert \nabla^2 u \vert + \vert \nabla u \vert^2 \right) dx\\
&\leq C \sum_{j=1}^N \int_{U_j} \eta_j \left( \vert \nabla^2 u \vert^2 + \vert \nabla u \vert^2 \right) dx,
\end{align*}
where for the second inequality we have used Young's inequality for products. This already gives together with the previous upper bounds
\begin{align*}
\Vert u \Vert_{H^2(M;dV_g)} \leq C \sum_{j=1}^N \int_{U_j} \eta_j \left( \vert u \vert^2 + \vert \nabla u \vert^2 + \vert \nabla^2 u \vert^2 \right) dx.
\end{align*}
On the other hand, we apply the uniform ellipticity and the boundedness assumptions on $g$ and argue by Young's inequality with $\varepsilon$ small enough to get
\begin{equation}\label{eq:equivalence_H^2-norms_2}
\begin{aligned}
\Vert \nabla_g^2 u &\Vert_{L^2(M;dV_g)}^2 \geq \sum_{j=1}^N \int_{U_j} \eta_j \left( c \vert \nabla^2 u \vert^2 - C(\vert \nabla^2 u \vert \vert \nabla u \vert + \vert \nabla u \vert^2) \right) dx\\
&\geq \sum_{j=1}^N \int_{U_j} \eta_j \left( c \vert \nabla^2 u \vert^2 - (\varepsilon \vert \nabla^2 u \vert^{2} + C_\varepsilon \vert \nabla u \vert^2) \right) dx \geq \sum_{j=1}^N \int_{U_j} \eta_j \left( c \vert \nabla^2 u \vert^2 - C_\varepsilon \vert \nabla u \vert^2 \right) dx.
\end{aligned}
\end{equation}
Let $C_1>0$ be large enough. Using the observations for the $H^1(M;dV_g)$-norms from \eqref{eq:norm_equivalence_1} above we find that
\begin{equation}\label{eq:equivalence_H^2-norms_3}
\begin{aligned}
C_1 \Vert u \Vert_{H^2(M;dV_g)}^2 &\geq \Vert \nabla_g^2 u \Vert_{L^2(M;dV_g)}^2 + C_1 \Vert u \Vert_{H^1(M;dV_g)}^2\\
&\geq c \sum_{j=1}^N \int_{U_j} \eta_j \vert \nabla^2 u \vert^2 dx - C_\varepsilon \sum_{j=1}^N \int_{U_j} \eta_j \vert \nabla u \vert^2 dx + C_1 \sum_{j=1}^N \int_{U_j} \eta_j \left( \vert u \vert^2 + \vert \nabla u \vert^2 \right) dx\\
&\geq c \sum_{j=1}^N \int_{U_j} \eta_j \left( \vert u \vert^2 + \vert \nabla u \vert^2 + \vert \nabla^2 u \vert^2 \right) dx.
\end{aligned}
\end{equation}
Together with the upper bound, this implies the equivalence also for the $H^2(M;dV_g)$-norms, with equivalence constants depending on $\theta_1$ and $\theta_2$ from assumptions (A1) and (A2), but otherwise independent of the metric $g$.

Now, since $H^s(M;dV_g) = (H^1(M;dV_g), L^2(M;dV_g))_{s,2}$ for $s \in (0,1)$ and by the precise construction of the interpolation space via (e.g.) the K-method, the equivalence of the $L^2(M;dV_g)$- and the $H^1(M;dV_g)$-norms imply the equivalence of the $H^s(M;dV_g)$-spaces, see e.g. \cite{BL76}. Note here, that the equality $H^s(M;dV_g) = (H^1(M;dV_g), L^2(M;dV_g))_{s,2}$ is given with equal norms (cf. Theorem 3.1 in \cite{CWHM15}). Therefore, as before, the equivalence constants only depend on $\theta_1$ and $\theta_2$, but otherwise they are independent of the metrics $g$. A similar argument yields the same result for $H^s(M;dV_g)$ for $s\in(1,2)$.

As a consequence of the above equivalence results, by the definition of the restricted spaces $H^s(O,dV_g)$ as quotient spaces, we also obtain that these do not depend on the choice of the metric, i.e. there exist constants  $c,C>0$ such that
\begin{align*}
c \Vert u \Vert_{H^s(O;dV_{\bar{g}})} \leq \Vert u \Vert_{H^s(O;dV_{g})} \leq C \Vert u \Vert_{H^s(O;dV_{\bar{g}})}.
\end{align*}
The constants $c$ and $C$ only depend on $\theta_1$ and $\theta_2$ from assumptions (A1) and (A2), but are otherwise independent of $g$ and $\bar{g}$.

For the negative Sobolev spaces, we apply a simple duality argument. By the equivalence of the $L^2(M;dV_g)$-, the $H^1(M;dV_g)$- and the $H^2(M;dV_g)$-norms, also the respective dual spaces $L^2(M;dV_g)$, $H^{-1}(M;dV_g)$ and $H^{-2}(M;dV_g)$ are equivalent. By an interpolation argument as above the same holds for all spaces $H^{-s}(M;dV_g)$ for $s\in[0,2]$. Then it is an easy consequence to infer the equivalence of the $H^{-s}(O;dV_g)$- and the $\widetilde{H}^{-s}(O;dV_g)$-spaces from their respective definitions. As before, the equivalence constants only depend on $\theta_1$ and $\theta_2$, but are otherwise independent of the metric $g$.

Related to the above discussion, as a result of the regularity of the set $O$, we can also identify these spaces on $O$ as interpolation spaces: 
\begin{equation}
\label{eq:interpol_id}
H^s(O;dV_g) = (H^{s_1}(O;dV_g), H^{s_0}(O;dV_g))_{t,2}, \mbox{ where } s_0, s_1 \in [-2,2], \ s_0 \leq s_1, \ s = ts_1 + (1-t)s_0. 
\end{equation}
Indeed, this follows since by the smoothness of $O$ and the smoothness of $M$ there exists a simultaneous extension operator $\mathcal{E}: H^{s_0}(O;dV_g) \to H^{s_0}(M;dV_g)$ for $s\in[0,2]$ which is uniformly (in $g$) bounded by a constant $C>0$ (depending on $M$, $\theta_1$ and $\theta_2$; see Theorem 2.2 in \cite{Rychkov99}, Theorem 4.6 in \cite{CWHM15}), i.e.
\begin{align*}
\Vert \mathcal{E}u \Vert_{H^{s_0}(M;dV_g)} \leq C \Vert u \Vert_{H^{s_0}(O;dV_g)} \quad \text{and} \quad \Vert \mathcal{E}u \Vert_{H^{s_1}(M;dV_g)} \leq C \Vert u \Vert_{H^{s_1}(O;dV_g)}.
\end{align*}
By Lemma 4.2 in \cite{CWHM15} this implies the identity \eqref{eq:interpol_id} with the equality being considered in the sense of equivalent norms. Moreover, the equivalence constants only depend on $\theta_1$ and $\theta_2$ but are otherwise independent of $g$. 
By a duality argument, an analogous interpolation identity holds for the spaces $\widetilde{H}^{-s}(O;dV_g) = (\widetilde{H}^{-s_0}(O;dV_g), \widetilde{H}^{-s_1}(O;dV_g))_{t,2}$, $s_0,s_1 \in [-2,2]$ (cf. Corollary 4.9 in \cite{CWHM15}). 

Due to this equivalence result, from now on we will no longer specify the metric, with respect to which the norm is considered.
\\

Now we turn to the setting of the fractional Dirichlet spectral Laplacian. We argue in a similar fashion as before that all norms are equivalent when the respective eigenfunction basis and eigenvalues, with respect to which we take the norm, are produced from two different operators $\nabla\cdot a \nabla$ and $\nabla\cdot \bar{a} \nabla$. The equivalence constants only depend on $\theta_1$ and $\theta_2$ from assumptions (A1') and (A2'), but are otherwise independent of the metrics $a$ and $\bar{a}$.

More precisely, let $(\phi_k^a)_{k\in\N}$ denote an $L^2(\Omega)$-orthonormal basis of ordered eigenfunctions with respective eigenvalues $(\lambda_k^a)_{k\in\N}$, $0 < \lambda_1^a \leq \lambda_2^a \leq \dots$, such that
\begin{equation*}
\begin{cases}
\begin{alignedat}{2}
-\nabla\cdot a \nabla \phi_k^a &= \lambda_k^a \phi_k^a \quad &&\text{in } \Omega,\\
\phi_k^a &= 0 \quad &&\text{on } \partial\Omega.
\end{alignedat}
\end{cases}
\end{equation*}

Firstly, we note that $\Vert u \Vert_{L^2(\Omega)}^2 = \sum_{k=1}^\infty \vert (u,\phi_k^a)_{L^2(\Omega)} \vert^2 = \Vert u \Vert_{\mathcal{H}_a^0(\Omega)}^{2}$.
In particular it holds that
\begin{align*}
\Vert u \Vert_{\mathcal{H}_a^0(\Omega)} = \Vert u \Vert_{L^2(\Omega)} = \Vert u \Vert_{\mathcal{H}_{\bar{a}}^0(\Omega)}
\end{align*}
for two metrics $a$ and $\bar{a}$ satisfying the assumption (A1'), where $\Vert \cdot \Vert_{\mathcal{H}_a^s}$ denotes the norm with respect to the eigenfunction expansion belonging to the metric $a$.

Secondly, for $u \in H^1_0(\Omega)$, we infer
\begin{align*}
\Vert \nabla u \Vert_{L^2(\Omega)}^2 &= \int_{\Omega} \nabla u \cdot \nabla u dx \sim \int_{\Omega} a \nabla u \cdot \nabla u dx = - \int_{\Omega} (\nabla\cdot a \nabla u) u dx\\
&= \int_{\Omega} \sum_{k,l=1}^\infty \Big( \lambda_k^a (u,\phi_k^a)_{L^2(\Omega)} \phi_k^a \Big) \Big( (u,\phi_l^a)_{L^2(\Omega)} \phi_l^a \Big) dx = \sum_{k=1}^\infty \lambda_k^a \vert (u,\phi_k^a)_{L^2(\Omega)} \vert^2 = \Vert u \Vert_{\mathcal{H}_a^1(\Omega)}^2
\end{align*}
with the equivalence constants in the second equality step only depending on $\theta_1$ from assumption (A1'). This yields (together with Poincaré's inequality)
\begin{align*}
\Vert u \Vert_{\mathcal{H}_{\bar{a}}^1(\Omega)} \sim \Vert u \Vert_{H^1(\Omega)} \sim \Vert u \Vert_{\mathcal{H}_a^1(\Omega)} \mbox{ for } u \in H^1_0(\Omega).
\end{align*}
with the equivalence constants only depending on $\theta_1$.

Thirdly, we will derive this equivalence result also for the $\mathcal{H}^2$-norm. On the one hand, we observe by the uniform boundedness assumptions that for $u\in C_c^{\infty}(\Omega)$
\begin{align*}
\Vert u \Vert_{\mathcal{H}_a^2(\Omega)}^{2} 
&= \sum_{k=1}^\infty (\lambda_k^a)^2 \vert (u,\phi_k^a)_{L^2(\Omega)} \vert^2 
= \sum_{k=1}^\infty  \vert (\nabla u,a \nabla \phi_k^a)_{L^2(\Omega)} \vert^2 
= \Vert \nabla\cdot a \nabla u \Vert_{L^2(\Omega)}^{2} 
\leq C \Vert u \Vert_{H^2(\Omega)}^2.
\end{align*}
On the other hand, we find that 
\begin{align*}
\Vert u \Vert_{\mathcal{H}_a^2(\Omega)}^2 &= \Vert \nabla\cdot a \nabla u \Vert_{L^2(\Omega)}^2 
\geq c \Vert \nabla^2 u \Vert_{L^2(\Omega)}^2 - C_\varepsilon \Vert \nabla u \Vert_{L^2(\Omega)}^2.
\end{align*}
For the last inequality we used the same arguments as in \eqref{eq:equivalence_H^2-norms_1} and \eqref{eq:equivalence_H^2-norms_2}. Then similarly as in \eqref{eq:equivalence_H^2-norms_3} for $C_1>0$ large enough we infer
\begin{align*}
C_1 \Vert u \Vert_{\mathcal{H}_a^2(\Omega)} &\geq \Vert \nabla\cdot a \nabla u \Vert_{L^2(\Omega)} + C_1 \Vert u \Vert_{\mathcal{H}_a^1(\Omega)} \geq c \Vert \nabla^2 u \Vert_{L^2(\Omega)} - C_\varepsilon \Vert \nabla u \Vert_{L^2(\Omega)} + C_1 \Vert u \Vert_{H^1(\Omega)}\\
&\geq c \Vert u \Vert_{H^2(\Omega)}.
\end{align*}
This implies that
\begin{align*}
\Vert u \Vert_{\mathcal{H}_a^2(\Omega)} \sim \Vert u \Vert_{H^2(\Omega)} \sim \Vert u \Vert_{\mathcal{H}_{\bar{a}}^2(\Omega)} \mbox{ for } u \in H^2_0(\Omega) := \overline{C^{\infty}_c(\Omega)}^{H^2(\Omega)}
\end{align*}
with the equivalence constants only depending on $\theta_1$ and $\theta_2$ from assumptions (A1') and (A2').

Using the interpolation and duality arguments as before in the manifold setting also yields the equivalence for $\mathcal{H}^s(O)$ and $\widetilde{\mathcal{H}}^s(O)$  for all $s \in [-2,2]$ and all $O \subset \Omega$, i.e. 
\begin{align}\label{eq:norm_equivalence_2}
\Vert u \Vert_{\mathcal{H}_a^{s}(O)} \sim \Vert u \Vert_{H^{s}(O)} \sim \Vert u \Vert_{\mathcal{H}_{\bar{a}}^{s}(O)}
\end{align}
for $u \in \overline{C^{\infty}_c(\Omega)}^{H^{s}(\Omega)}$
with equivalence constants only depending on $\theta_1$ and $\theta_2$. More precisely, we first prove this equivalence for $C_c^\infty(\Omega)$-functions by extending them by $0$ to $\R^n$ and by using the interpolation equality on $\R^n$ from Theorem 4.1 in \cite{CWHM15}. Then by density the result follows for all $u \in \overline{C_c^{\infty}(\Omega)}^{H^{s}(\Omega)}$ and by the definition of the restricted spaces as quotient spaces, the result also holds for $u \in \mathcal{H}^{s}(O)$ and $u \in \widetilde{\mathcal{H}}^{s}(O)$.
We remark that, if $O \Subset \Omega$ is sufficiently regular (Lipschitz-regular is enough), by the result of Theorem 2.2 in \cite{Rychkov99} (see also Theorem 4.6 in \cite{CWHM15}) and possibly using a cut-off argument, there exists a zero extension operator with equivalent norms. Consequently, it also holds that $\Vert u \Vert_{\mathcal{H}_a^{s}(O)} \sim \Vert u \Vert_{H^{s}(O)} \sim \Vert u \Vert_{\mathcal{H}_{\bar{a}}^{s}(O)}$ for all $u\in C^{\infty}(O)$.

Also in this context, note that for $s_0, s_1 \in [-2,2]$, $s_0 \leq s_1$ there exists a simultaneous and uniformly bounded (in $a$) extension operator $\mathcal{E}: \mathcal{H}^{s_0}(O) \to \mathcal{H}^{s_0}(\Omega)$ for $O \Subset \Omega$ such that $\Vert \mathcal{E}u \Vert_{\mathcal{H}^{s_j}(\Omega)} \leq C \Vert u \Vert_{\mathcal{H}^{s_j}(O)}$ for $u \in \mathcal{H}^{s_1}(O)$ and $j \in \{0,1\}$. For this, use the result of Theorem 2.2 in \cite{Rychkov99} (see also Theorem 4.6 in \cite{CWHM15}) and if needed a cut-off argument and additionally the previously derived equivalence for the $\mathcal{H}^s$- and $H^s$-spaces. As before, this existence implies the interpolation identity
\begin{align*}
\mathcal{H}^s(O) = \left( \mathcal{H}^{s_1}(O), \mathcal{H}^{s_0}(O) \right)_{t,2}, \quad s = ts_1 + (1-t) s_0,
\end{align*}
with equality in the sense of equivalence of norms. By the uniformity of the extension operator in $a$, the equivalence constants are also independent of the precise metric $a$ (cf. Lemma 4.2 in \cite{CWHM15}).

Again, due to this equivalence result we will not specify the dependence of the norm on the metric $a$ any longer.

\subsection{Caffarelli-Silvestre extension}
\label{sec:CS}

The essential tool in our analysis is the following Caffarelli-Silvestre type extension result. Again, let $(M,g)$ be a closed, connected, smooth $n$-dimensional Riemannian manifold. Building on the work \cite{CS07}, in \cite{ST10} it was proven that the nonlocal fractional Laplace-Beltrami operator $(-\Delta_g)^s$, $s\in(0,1)$, can be realized as a local operator at the cost of adding one dimension (see also \cite{CFSS24}). More precisely, for $u \in H^s(M)$ it holds that
\begin{align*}
(-\Delta_g)^s u(x') = -\bar{c}_s \lim_{x_{n+1}\to0} x_{n+1}^{1-2s} \partial_{n+1} \tilde{u}(x',x_{n+1}), \quad x \in M,
\end{align*}
where $\bar{c}_s = \frac{4^s\Gamma(s)}{2s\vert\Gamma(-s)\vert}>0$ and $\tilde{u} \in \dot{H}^1(M\times\R_+, x_{n+1}^{1-2s})$ is the unique weak solution to
\begin{equation*}
\begin{cases}
\begin{alignedat}{2}
\left( x_{n+1}^{1-2s}  (-\Delta_g) - \partial_{n+1} x_{n+1}^{1-2s} \partial_{n+1} \right) \tilde{u} &= 0, \quad &&\text{in } M \times \R_+,\\
\tilde{u} &= u \quad &&\text{on } M\times\{0\}.
\end{alignedat}
\end{cases}
\end{equation*}
In particular, the nonlocal problem
\begin{align*}
(-\Delta_g)^s u^f = f \quad \text{in } M
\end{align*}
for $f \in H^{-s}(M)$ can be reformulated in terms of the Caffarelli-Silvestre type extension
\begin{equation*}
\begin{cases}
\begin{alignedat}{2}
\left( x_{n+1}^{1-2s}  (-\Delta_g) - \partial_{n+1} x_{n+1}^{1-2s} \partial_{n+1} \right) \tilde{u}^f &= 0 \quad &&\text{in } M \times \R_+,\\
-\bar{c}_s \lim_{x_{n+1}\to0} x_{n+1}^{1-2s} \partial_{n+1} \tilde{u}^f &= f \quad &&\text{on } M\times\{0\},
\end{alignedat}
\end{cases}
\end{equation*}
and $u^f = \tilde{u}^f(\cdot, 0)$.
For uniqueness, we require that $f \in H^{-s}(M)$ is such that $(f,1)_{L^2(M)} = 0$ and that $\tilde{u}^f$ satisfies $(\tilde{u}^f(\cdot,x_{n+1}), 1)_{L^2(M)} = 0$ for all $x_{n+1} > 0$. Consequently, by Poincaré's inequality (in tangential directions) and by the compactness of $M$, we have $\tilde{u}^f \in H^1(M \times \R_+, x_{n+1}^{1-2s})$ (and not only $\tilde{u}^f \in \dot{H}^1(M\times \R_+, x_{n+1}^{1-2s})$).

In the setting for the spectral fractional Laplacian on $\Omega \subset \R^n$ bounded, Lipschitz, we impose additional boundary conditions. By \cite{CS16}, the problem formulation in terms of the Caffarelli-Silvestre extension then reads
\begin{equation*}
\begin{cases}
\begin{alignedat}{2}
-\nabla\cdot x_{n+1}^{1-2s} \tilde{a} \nabla \tilde{u}^f &= 0 \quad &&\text{in } \Omega \times \R_+,\\
\tilde{u}^f &= 0 \quad &&\text{on } \partial\Omega \times \R_+,\\
-\bar{c}_s \lim_{x_{n+1}\to0} x_{n+1}^{1-2s} \partial_{n+1} \tilde{u}^f &= f \quad &&\text{on } \Omega\times\{0\},
\end{alignedat}
\end{cases}
\end{equation*}
where $\tilde{a} = \begin{pmatrix} a & 0 \\ 0 & 1 \end{pmatrix}$. Here, the homogeneous boundary conditions together with the boundedness of $\Omega$ guarantee, by Poincaré's inequality, that also in this case $\tilde{u}^f \in H^1(\Omega\times\R_+, x_{n+1}^{1-2s})$.

\subsection{Additional notation}

In concluding this section, we introduce some additional notation. Let $M$ be a closed, connected, smooth $n$-dimensional Riemannian manifold or $M=\R^n$. We define $\R_+ := \{y\in\R:\ y>0\}$ and by $B_r^+(x)$ we denote the ball of radius $r$ centered at $x = (x',x_{n+1}) \in M \times \R_{\geq0}$ restricted to the upper half plane, i.e. $B_r^+(x) := B_r(x) \cap (M \times \R_{\geq0})$. Additionally, we usually write $B_r'(x) := B_r(x) \cap (M \times \{0\})$ for $x = (x',0)\in M \times \{0\}$. In particular, $x'$ will usually denote an $n$-dimensional point with $x' \in M$ or $x' \in \R^n$ and $x$ will usually denote an $(n+1)$-dimensional point with $x \in M\times\R_+$ or $x \in \R^{n+1}_+$. For a set $E \subset \Omega$ we denote its characteristic function by $\chi_{E}(x):= \left\{ \begin{array}{ll} 1 & \mbox{ for } x \in E,\\
0 & \mbox{ for } x \in \Omega \setminus E.  \end{array} \right. $

Given a normed vector space $Y$, we denote by $Y^*$ the continuous dual space of $Y$. We write $\langle \cdot, \cdot \rangle_{Y,Y^*}$ to denote the dual pairing of $Y^*$ with $Y$.

Moreover, let $A,B \geq 0$. We will write $A \sim B$ if there exists a constant $C>0$ such that
\begin{align*}
\frac{1}{C} B \leq A \leq CB.
\end{align*}

In what follows below, without further comment, we will always assume that $s\in (0,1)$.

\section{Proof of the main results for the closed manifold setting}\label{sec:proof_manifold}

In this section we will prove the main results of this work in the setting of closed, smooth, connected $n$-dimensional Riemannian manifolds.

\subsection{A priori estimates}\label{sec:apriori}
First of all, we recall the solution representation for the Caffaralli-Silvestre extension in terms of the eigenfunctions of the Laplace-Beltrami operator $(-\Delta_g)$ (see, e.g., Lemma A.1 in \cite{R23}, which recalls earlier results from \cite{ST10}).

\begin{lem}\label{lem:EigenfunctionExpansion}
Let $(M,g)$ be a smooth, closed, connected $n$-dimensional Riemannian manifold. Let $f \in C^\infty(M)$ with $(f,1)_{L^2(M)} = 0$. Let $\tilde{u}^f \in H^1(M\times\R_+ , x_{n+1}^{1-2s})$ denote the solution to the Caffarelli-Silvestre extension associated to the source-to-solution problem with source term $f$, see equation \eqref{eq:CS}. Let $(\phi_k)_{k\in\N_0}$, $(\lambda_k)_{k\in\N_0}$ with $0=\lambda_0<\lambda_1\leq\dots\leq\lambda_k\leq\dots$ denote an $L^2$-normalized basis of eigenfunctions and the sequence of associated eigenvalues for the Laplace-Beltrami operator $(-\Delta_g)$ on $(M,g)$. Then there is a constant $\tilde{c}_s = \frac{2^{1-s}}{\Gamma(s)}$ such that for all $(x',x_{n+1}) \in \R^{n+1}_+$ in an $L^2$-sense
\begin{align*}
\tilde{u}^f(x',x_{n+1}) = \tilde{c}_s \sum_{k=1}^\infty (f,\phi_k)_{L^2(M)} K_s(\sqrt{\lambda_k} x_{n+1}) \left(\frac{x_{n+1}}{\sqrt{\lambda_k}}\right)^s \phi_k(x').
\end{align*}
Here, $K_s$ denotes the modified Bessel function of the second kind.
\end{lem}

We refer to \cite{Olver10} for the definition and properties of the modified Bessel functions of the second kind.

The following a-priori estimates will allow us to restrict our analysis to finite height integrals instead of the full integral on $O \times \R_+$ for $O \subset M$ open.

\begin{lem}\label{lem:apriori}
Let $(M,g)$ be a smooth, closed, connected $n$-dimensional Riemannian manifold and let $O \subset M$ be open and Lipschitz. Let $(\phi_k)_{k\in\N_0}$, $(\lambda_k)_{k\in\N_0}$ be as in Lemma \ref{lem:EigenfunctionExpansion}. Let $f \in \widetilde{H}^{-s}_{\diamond}(O)$ with $(f,1)_{L^2(M)} = 0$ and let $\tilde{u}^f \in H^1(M\times\R_+, x_{n+1}^{1-2s})$ be the solution to \eqref{eq:CS} with metric $g$ and data $f$. Then for any $L>0$ it holds
\begin{align}
\label{eq:first_abstract_truncation}
\left\| \int\limits_{L}^{\infty} t^{1-2s} \tilde{u}^f(\cdot,t) dt \right\|_{H^1(O)} & \leq \tilde{c}_s \left( \int_{\sqrt{\lambda_1}L}^\infty z^{1-s} K_s(z) dz \right) \|f\|_{H^{-1}(M)},
\end{align}
where $K_s$ denotes the modified Bessel function of the second kind.
\end{lem}

\begin{rmk}
We observe that our assumption (A2) allows us to control the dependence of the lower integral limit on the first eigenvalue $\lambda_1$ independently of the unknown metric $g$ (see equation \eqref{eq:uniform_bounds_first_eigenvalue}).
\end{rmk}

\begin{proof}
At first, let $f \in C_c^\infty(O)$. As a consequence of Lemma \ref{lem:EigenfunctionExpansion} we find that (c.f. the proof of Lemma 5.1 in \cite{R23}) for $x' \in \R^n$
\begin{align*}
\int\limits_{L}^{\infty} t^{1-2s} \tilde{u}^f(x',t) dt
&= \tilde{c}_s \sum\limits_{k=1}^{\infty} (f,\phi_k)_{L^2(M)} \int\limits_{L}^{\infty} t^{1-2s} K_{s}(\sqrt{\lambda_k} t) \left( \frac{t}{\sqrt{\lambda_k}} \right)^s dt \ \phi_k(x').
\end{align*}
Thus, we estimate by substitution and Hölder's inequality
\begin{align*}
&\left\| \int\limits_{L}^{\infty} t^{1-2s} \tilde{u}^f(\cdot,t) dt \right\|_{H^1(O)} \\
&= \sup_{\substack{h \in \widetilde{H}^{-1}(O),\\ \Vert h \Vert_{H^{-1}(M)} = 1}} \tilde{c}_s \sum\limits_{k=1}^{\infty} (f,\phi_k)_{L^2(M)} (h,\phi_k)_{L^2(O)} \int\limits_{L}^{\infty} t^{1-2s} K_{s}(\sqrt{\lambda_k} t) \left( \frac{t}{\sqrt{\lambda_k}} \right)^s dt\\
&\quad \leq \sup_{\substack{h \in \widetilde{H}^{-1}(O),\\ \Vert h \Vert_{H^{-1}(M)} = 1}} \ \tilde{c}_s \sum_{k=1}^\infty \left(\int\limits_{\sqrt{\lambda_k}L}^{\infty} z^{1-s} K_s(z) dz \right) \lambda_k^{-1} \vert (f,\phi_k)_{L^2(M)} \vert \ \vert (h,\phi_k)_{L^2(O)} \vert \\
&\quad \leq \sup_{\substack{h \in \widetilde{H}^{-1}(O),\\ \Vert h \Vert_{H^{-1}(M)} = 1}} \ \tilde{c}_s \left(\int\limits_{\sqrt{\lambda_1} L}^{\infty} z^{1-s} K_s(z) dz \right) \left( \sum_{k=1}^\infty \lambda_k^{-1} \vert (f,\phi_k)_{L^2(M)} \vert^2 \right)^{1/2} \left( \sum_{k=1}^\infty \lambda_k^{-1} \vert (h,\phi_k)_{L^2(O)} \vert^2 \right)^{1/2}\\
&\quad = \tilde{c}_s \left(\int\limits_{\sqrt{\lambda_1}L}^{\infty} z^{1-s} K_s(z) dz \right) \Vert f \Vert_{H^{-1}(M)}.
\end{align*}
By density of $C_c^\infty(O)$ in $\widetilde{H}^{-s}(O)$ this concludes the proof.
\end{proof}

Next, using that, by Assumption (A2), $\lambda_1$ is uniformly bounded, we provide estimates for the integral on the right hand side of the inequality \eqref{eq:first_abstract_truncation} in Lemma \ref{lem:apriori}.

\begin{lem}\label{lem:estimate_integral_modified_Bessel_function}
Let $K_s$ denote the modified Bessel function of the second kind and let $\omega(\varepsilon)$ be a given modulus of continuity such that $\lim_{\varepsilon\to0} \omega(\varepsilon) = 0$. Let $\kappa \in (0,1)$ and $\varepsilon_0>0$ be such that $\omega(\varepsilon) \leq \kappa$ for all $\varepsilon \in (0,\varepsilon_0)$. Choose $L = 2 \log \left(\frac{1}{\omega(\varepsilon)} \right)$.

Then there exists a constant $C>0$ (which depends on $\kappa$ and the precise form of $\omega$) such that for all $\varepsilon \in (0, \varepsilon_0)$
\begin{align*}
\left\vert \int_L^\infty z^{1-s} K_s(z) dz \right\vert \leq C \omega(\varepsilon).
\end{align*}
\end{lem}

\begin{proof}
We first of all note that by 10.25.3 in \cite{Olver10} we have
\begin{align*}
K_s(z) \sim z^{-1/2} e^{-z} \quad \text{as } z \to \infty.
\end{align*}
With this at hand we estimate for $L$ large enough
\begin{align*}
\left\vert \int_L^\infty z^{1-s} K_s(z) dz \right\vert \leq C \left\vert \int_L^\infty z^{\frac{1}{2}-s} e^{-z} dz \right\vert \leq C L^{\frac{1}{2}-s} e^{-L}.
\end{align*}
Now, choosing $L$ as in the statement, we infer
\begin{align*}
\left\vert \int_L^\infty z^{1-s} K_s(z) dz \right\vert \leq C \left( 2 \log \left( \frac{1}{\omega(\varepsilon)}\right) \right)^{\frac{1}{2}-s} \omega(\varepsilon)^2 \leq C \omega(\varepsilon),
\end{align*}
where we have used that $\omega(\varepsilon)$ will be small in the limit as $\varepsilon \to 0$.
\end{proof}

\subsection{Quantitative unique continuation}\label{sec:quant_UCP}

The proof of our main proposition, Proposition \ref{prop:stability_estimate}, is based on two propagation of smallness results, Propositions \ref{prop:bbucp} and \ref{prop:3ballsinequ}. The first proposition is a quantitative boundary-bulk unique continuation result from \cite{RS20} to propagate smallness on the boundary $\R^n\times\{0\}$ to an open half-ball in the upper half-space. The second one is a three-balls-inequality in the bulk to propagate the inferred smallness upwards in the $x_{n+1}$-direction. For the proof of the latter we use a Carleman estimate, Lemma \ref{lem:Carleman_estimate_bulk}.

\subsubsection{Two propagation of smallness results}

We start by recalling the following boundary-bulk unique continuation estimate from \cite{RS20}. In order to apply it in the setting of our Riemannian manifolds, we invoke the condition (A3). As outlined in the introduction, we expect that this condition can be relaxed and that our argument remains valid in more general settings. In order to emphasize the main ideas without additional technicalities, in the present article, we however restrict to the setting outlined in assumption (A3).

\begin{prop}[Quantitative boundary-bulk unique continuation, Proposition 5.13 in \cite{RS20}]\label{prop:bbucp}
Let $(M,g)$ be a smooth, closed, connected $n$-dimensional Riemannian manifold. Let $O', O \subset M$ be open and Lipschitz such that $O' \Subset O$ and that $O$ satisfies the flatness condition from assumption (A3), in particular, assume that, in local coordinates, $(O,\restr{g}{O}) = (O, \Id_{n \times n})$. Let $\tilde{w} \in H^1(M\times\R_+, x_{n+1}^{1-2s})$ satisfy
\begin{align*}
\left( x_{n+1}^{1-2s} (-\Delta_g) - \partial_{n+1} (x_{n+1}^{1-2s} \partial_{n+1}) \right) \tilde{w} = 0 \quad \text{in } O \times \R_+,
\end{align*}
with $\tilde{w}(\cdot, 0) \in H^s(M)$. Let $x^0 \in O' \times \{0\}$ and let $r>0$ be such that $B_{4r}^+(x^0) \subset O \times \R_+$. Then, there exist constants $\alpha = \alpha(s,n) \in (0,1)$, $c = c(s,n) \in (0,\frac{1}{2})$ and $C = C(s,n) > 0$ such that
\begin{align*}
&\Vert \tilde{w} \Vert_{L^2(B_{cr}^+(x^0), x_{n+1}^{1-2s})}\\
&\qquad \leq C \left( \Vert \tilde{w}(\cdot,0) \Vert_{H^s(B_{3r}'(x^0))} + \Vert \lim_{x_{n+1}\to0} x_{n+1}^{1-2s} \partial_{n+1} \tilde{w} \Vert_{H^{-s}(B_{3r}'(x^0))} \right)^{1-\alpha}\\
&\qquad\qquad \times \left( \Vert \tilde{w} \Vert_{L^2(B_r^+(x^0), x_{n+1}^{1-2s})} + \Vert \tilde{w}(\cdot,0) \Vert_{H^s(B_{3r}'(x^0))} + \Vert \lim_{x_{n+1}\to0} x_{n+1}^{1-2s} \partial_{n+1} \tilde{w} \Vert_{H^{-s}(B_{3r}'(x^0))} \right)^{\alpha}.
\end{align*}
\end{prop}

The following three-balls-inequality is similar to the result of Proposition 5.4 in \cite{RS20}. However, we cannot directly rely on the result in \cite{RS20} since in our case the radius $r$ of the balls centered at $x^0 \in O'\times\R_+$ must be independent of the height $x_{n+1}^0$, it cannot become arbitrarily large. Thus, we cannot directly carry out the rescaling argument as in \cite{RS20} and apply the three-balls-inequality for uniformly elliptic equations. Instead, we will use the assumed fixed boundary distance and will reduce it to a constant coefficient Carleman estimate for the Laplacian.

\begin{prop}[Three-balls-inequality]\label{prop:3ballsinequ}
Let $(M,g)$, $O$ and $O'$ be as in Proposition \ref{prop:bbucp}. Let $\tilde{w} \in H^1(M\times\R_+, x_{n+1}^{1-2s})$ satisfy
\begin{align*}
\left( x_{n+1}^{1-2s} (-\Delta_g) - \partial_{n+1} (x_{n+1}^{1-2s} \partial_{n+1}) \right) \tilde{w} = 0 \quad \text{in } O \times \R_+.
\end{align*}
Let $x^0 \in O' \times \R_+$ and let $\delta_0>0$ and $r>0$ be such that $B_{4r}(x^0) \subset O \times (\delta_0,\infty)$. Then, there exist constants $C$ and $\alpha$, with $\alpha \in (0,1)$, such that
\begin{align*}
\|\tilde{w}\|_{L^2(B_{2r}(x^0), x_{n+1}^{1-2s})}
\leq C \|\tilde{w}\|_{L^2(B_{r}(x^0), x_{n+1}^{1-2s})}^{\alpha} \Vert \tilde{w} \Vert_{L^2(B_{4r}(x^0), x_{n+1}^{1-2s})}^{1-\alpha}.
\end{align*}
Here, the constant $C$ depends on $n$, $s$, $\delta_0$ and $r$ and the parameter $\alpha$ depends (mildly) on $r$, more precisely $\alpha \in (a,b)$ for some $0<a<b<1$, which are independent of $r$.
\end{prop}

The to us central aspect of this quantitative unique continuation result is that while the estimate depends on the distance to the boundary through the parameter $\delta_0$, up to this dependence, it is \emph{independent} of the vertical position of the ball under consideration. In particular, it does not diverge as $x^0_{n+1}\rightarrow \infty$.

We derive the three-balls-inequality as a consequence of the following Carleman estimate.

\begin{lem}\label{lem:Carleman_estimate_bulk}
Let $s \in (0,1)$ and let $O \subset \R^n$ be an open, smooth set. 
Let $z = (z',z_{n+1}) \in O \times \R_+$ be such that $B_{R}(z) \subset O \times (\delta_0,\infty)$ for some $R>0$ and $\delta_0>0$. Set $\phi(x) = \widetilde{\phi}(\ln(\vert x-z \vert))$ with
\begin{align*}
\widetilde{\phi}(t) = - t + \frac{1}{10} \left( t \arctan(t) - \frac{1}{2} \ln( 1 + t^2) \right).
\end{align*}  
Assume that $\overline{w} \in H^1(\R^{n+1}_+, x_{n+1}^{1-2s})$ satisfies $\supp(\overline{w}) \subset B_{R}(z) \setminus B_r(z)$ for some $0 < r < R$.

Then there exists $\tau_0 > 0$ such that for any $\tau \geq \tau_0$ we have
\begin{equation}\label{eq:Carleman_estimate}
\begin{aligned}
&\tau \left\Vert e^{\tau\phi} \left( 1 + \ln(\vert x-z \vert)^2 \right)^{-\frac{1}{2}} \vert x-z \vert^{-1} x_{n+1}^{\frac{1-2s}{2}} \overline{w} \right\Vert_{L^2(O \times \R_+)}\\
&\qquad + \left\Vert e^{\tau\phi} \left( 1 + \ln(\vert x-z \vert)^2 \right)^{-\frac{1}{2}} x_{n+1}^{\frac{1-2s}{2}} \nabla \overline{w} \right\Vert_{L^2(O \times \R_+)}\\
&\quad \leq C \tau^{-\frac{1}{2}} \left\Vert e^{\tau\phi} \vert x-z \vert x_{n+1}^{\frac{2s-1}{2}} \Big( x_{n+1}^{1-2s} (-\Delta) - \partial_{n+1} (x_{n+1}^{1-2s} \partial_{n+1}) \Big) \overline{w} \right\Vert_{L^2(O \times \R_+)}
\end{aligned}
\end{equation}
for some constant $C>0$. Here, $C$ and $\tau_0$ only depend on $R$, $n$, $s$, $\delta_0$ and the precise form of $\widetilde{\phi}$.
\end{lem}

We highlight that the for our purposes main contribution of this proposition is the independence of the estimate \eqref{eq:Carleman_estimate} of the vertical variable as $x_{n+1} \rightarrow \infty$. By the smoothness and growth behaviour of the weight, we deduce this Carleman estimate from a weighted estimate for the Laplacian.

\begin{proof}
We reduce the problem to a Carleman estimate for the constant coefficient Laplacian. Indeed, for the constant coefficient Laplacian, we have (see, for instance, \cite{KT01} and references therein)
\begin{equation}\label{eq:Carleman_estimate_Laplace}
\begin{aligned}
&\tau \left\Vert e^{\tau\phi} \left( 1 + \ln(\vert x-z \vert)^2 \right)^{-\frac{1}{2}} \vert x-z \vert^{-1}  \overline{u} \right\Vert_{L^2(O \times \R_+)} + \left\Vert e^{\tau\phi} \left( 1 + \ln(\vert x-z \vert)^2 \right)^{-\frac{1}{2}}  \nabla \overline{u} \right\Vert_{L^2(O \times \R_+)}\\
&\qquad \leq C \tau^{-\frac{1}{2}} \left\Vert e^{\tau\phi} \vert x-z \vert \Big(  (-\Delta) - \partial_{n+1}^2 ) \Big) \overline{u} \right\Vert_{L^2(O \times \R_+)}.
\end{aligned}
\end{equation}
Now, we rewrite our equation with the weight as follows: For $\overline{w} = x_{n+1}^{\frac{2s-1}{2}} \overline{u}$ (note that by the safety distance to the boundary there are no singularities), we obtain by expanding the equation
\begin{align*}
x_{n+1}^{\frac{2s-1}{2}} \Big( x_{n+1}^{1-2s} (-\Delta) - \partial_{n+1} (x_{n+1}^{1-2s} \partial_{n+1}) \Big) \overline{w} 
&= x_{n+1}^{\frac{2s-1}{2}} \Big( x_{n+1}^{1-2s} (-\Delta) - \partial_{n+1} (x_{n+1}^{1-2s} \partial_{n+1}) \Big) x_{n+1}^{\frac{2s-1}{2}} \overline{u} \\
&=  ((-\Delta) - \partial_{n+1}^2) \overline{u} - \frac{1-4s^2}{4} x_{n+1}^{-2} \overline{u}.
\end{align*}
Thus, it holds $((-\Delta) - \partial_{n+1}^2) \overline{u} = x_{n+1}^{\frac{2s-1}{2}} \Big( x_{n+1}^{1-2s} (-\Delta) - \partial_{n+1} (x_{n+1}^{1-2s} \partial_{n+1}) \Big) \overline{w} + \frac{1-4s^2}{4} x_{n+1}^{-2} \overline{u}$. We plug this into our Carleman estimate \eqref{eq:Carleman_estimate_Laplace} for the Laplacian and obtain
\begin{equation}\label{eq:Carleman_estimate_Laplace1}
\begin{aligned}
&\tau \left\Vert e^{\tau\phi} \left( 1 + \ln(\vert x-z \vert)^2 \right)^{-\frac{1}{2}} \vert x-z \vert^{-1}  \overline{u} \right\Vert_{L^2(O \times \R_+)} + \left\Vert e^{\tau\phi} \left( 1 + \ln(\vert x-z \vert)^2 \right)^{-\frac{1}{2}}  \nabla \overline{u} \right\Vert_{L^2(O \times \R_+)}\\
&\qquad \leq C \tau^{-\frac{1}{2}} \left\Vert e^{\tau\phi} \vert x-z \vert \Big(  (-\Delta) - \partial_{n+1}^2 ) \Big) \overline{u} \right\Vert_{L^2(O \times \R_+)} \\
&\qquad \leq C\tau^{-\frac{1}{2}} \left\Vert e^{\tau\phi} \vert x-z \vert \left( x_{n+1}^{\frac{2s-1}{2}} \Big( x_{n+1}^{1-2s} (-\Delta) - \partial_{n+1} (x_{n+1}^{1-2s} \partial_{n+1}) \Big) \overline{w} + \frac{1-4s^2}{4} x_{n+1}^{-2} \overline{u} \right) \right\Vert_{L^2(O \times \R_+)}\\
&\qquad \leq C\tau^{-\frac{1}{2}} \left\Vert e^{\tau\phi} \vert x-z \vert x_{n+1}^{\frac{2s-1}{2}} \Big( x_{n+1}^{1-2s} (-\Delta) - \partial_{n+1} (x_{n+1}^{1-2s} \partial_{n+1}) \Big) \overline{w} \right\Vert_{L^2(O \times \R_+)}\\
&\qquad\qquad + C\tau^{-\frac{1}{2}} \delta_0^{-2} \left\Vert e^{\tau\phi} \vert x-z \vert \overline{u} \right\Vert_{L^2(O \times \R_+)}.
\end{aligned}
\end{equation}
Choosing $\tau_0 >0$ sufficiently large (depending on $R>0$ and $\delta_0>0$, e.g. $\tau_0 \geq \left( \frac{2CR^{3}}{\delta_0^2} \right)^2$), we can absorb the second right hand side contribution into the left hand side. Also plugging back in $\overline{u} = x_{n+1}^{\frac{1-2s}{2}} \overline{w}$, this leads to 
\begin{equation}\label{eq:Carleman_estimate_Laplace2}
\begin{aligned}
&\tau \left\Vert e^{\tau\phi} \left( 1 + \ln(\vert x-z \vert)^2 \right)^{-\frac{1}{2}} \vert x-z \vert^{-1} x_{n+1}^{\frac{1-2s}{2}} \overline{w} \right\Vert_{L^2(O \times \R_+)}\\
&\qquad\qquad + \left\Vert e^{\tau\phi} \left( 1 + \ln(\vert x-z \vert)^2 \right)^{-\frac{1}{2}}  \nabla ( x_{n+1}^{\frac{1-2s}{2}} \overline{w} ) \right\Vert_{L^2(O \times \R_+)}\\
&\qquad \leq C\tau^{-\frac{1}{2}} \left\Vert e^{\tau\phi} \vert x-z \vert x_{n+1}^{\frac{2s-1}{2}} \Big( x_{n+1}^{1-2s} (-\Delta) - \partial_{n+1} (x_{n+1}^{1-2s} \partial_{n+1}) \Big) \overline{w} \right\Vert_{L^2(O \times \R_+)}.
\end{aligned}
\end{equation}
We note that
\begin{equation}\label{eq:Carleman_estimate_Laplace3}
\begin{aligned}
\Vert \nabla (x_{n+1}^{\frac{1-2s}{2}} \overline{w}) \Vert_{L^2(O \times \R_+)} &\geq \Vert x_{n+1}^{\frac{1-2s}{2}} \nabla \overline{w} \Vert_{L^2(O \times \R_+)} - C \Vert x_{n+1}^{\frac{-1-2s}{2}} \overline{w} \Vert_{L^2(O \times \R_+)}\\
&\geq \Vert x_{n+1}^{\frac{1-2s}{2}} \nabla \overline{w} \Vert_{L^2(O \times \R_+)} - C \delta_0^{-1} \Vert x_{n+1}^{\frac{1-2s}{2}} \overline{w} \Vert_{L^2(O \times \R_+)}.
\end{aligned}
\end{equation}
Finally, we can modify the estimate \eqref{eq:Carleman_estimate_Laplace2} by choosing $\tau_0$ sufficiently large (again, depending on $\delta_0$ and $R$, e.g. $\tau_0 \geq \frac{2CR^{2}}{\delta_0}$) and absorbing the error term in \eqref{eq:Carleman_estimate_Laplace3} into the first contribution on the left hand side to arrive at the desired inequality.
\end{proof}

With the Carleman inequality at our availability, the proof of the three-balls-inequality is now almost identical to the proof of Proposition 5.3 in \cite{RS20}. We complement the outline of proof given there by some more details which are based on the proof of Proposition 3.1 in \cite{B13}. In what follows, we will work in local coordinates.

\begin{proof}[Proof of Proposition \ref{prop:3ballsinequ}]
In the following we will denote by $A_{r_1,r_2} := B_{r_2} \setminus \overline{B}_{r_1}$ the annuli with outer radius $r_2$ and inner radius $r_1$ centered at $x^0$.\\
Consider $\overline{w} := \tilde{w}\eta$, where $\eta(x) = \widetilde{\eta}(\vert x-x^0 \vert)$ is a radially symmetric (centred at $x^0$) cut-off function such that
\begin{align*}
\widetilde{\eta}(t) = 0 \text{ for } t \in (0,\frac{r}{4}) \cup (\frac{7r}{2}, \infty), \quad \widetilde{\eta}(t) = 1 \text{ for } t \in (\frac{r}{2},3r), \quad \vert \widetilde{\eta}' \vert \leq \frac{C}{r}, \quad \vert \widetilde{\eta}'' \vert \leq \frac{C}{r^2}
\end{align*}
for some $C>0$. It is clear that $\supp(\overline{w}) \subset B_{7r/2}(x^0) \setminus B_{r/4}(x^0)$. Hence, we can apply the Carleman estimate from Lemma \ref{lem:Carleman_estimate_bulk} to $\overline{w}$.

We expand the term on the right hand side of the Carleman estimate \eqref{eq:Carleman_estimate}
\begin{align*}
x_{n+1}^{\frac{2s-1}{2}} &\nabla \cdot \left( x_{n+1}^{1-2s} \nabla (\tilde{w}\eta) \right) = x_{n+1}^{\frac{2s-1}{2}} \nabla \cdot \left( x_{n+1}^{1-2s} \eta \nabla \tilde{w} \right) + x_{n+1}^{\frac{2s-1}{2}} \nabla \cdot \left( x_{n+1}^{1-2s} \tilde{w} \nabla \eta \right)\\
&= x_{n+1}^{\frac{2s-1}{2}} \eta \left( \nabla \cdot \left( x_{n+1}^{1-2s} \nabla \tilde{w} \right) \right) + x_{n+1}^{\frac{1-2s}{2}} \nabla \eta \cdot \nabla \tilde{w} + x_{n+1}^{\frac{2s-1}{2}} \tilde{w} \nabla \cdot \left( x_{n+1}^{1-2s} \nabla	\eta \right) + x_{n+1}^{\frac{1-2s}{2}} \nabla \tilde{w} \cdot \nabla \eta\\
&= x_{n+1}^{\frac{1-2s}{2}} \tilde{w} \left( x_{n+1}^{2s-1} \nabla \cdot (x_{n+1}^{1-2s} \nabla \eta) \right) + 2 x_{n+1}^{\frac{1-2s}{2}} \nabla \tilde{w} \cdot \nabla \eta,
\end{align*}
where for the last inequality we used the equation for $\tilde{w}$. Inserting this into the Carleman inequality \eqref{eq:Carleman_estimate} (only keeping the first term on the left hand side of the inequality) and using that $\widetilde{\phi}$ is monotonically decreasing yields for $\tau \geq \tau_0$
\begin{align*}
e^{\tau \widetilde{\phi}(\ln(2r))}& \Vert x_{n+1}^{\frac{1-2s}{2}} \tilde{w} \Vert_{L^2(A_{\frac{r}{2}, 2r})}\\
&\leq C \bigg( e^{\tau\widetilde{\phi}(\ln(\frac{r}{4}))} \Vert x_{n+1}^{\frac{1-2s}{2}} \tilde{w} \left( x_{n+1}^{2s-1} \nabla \cdot (x_{n+1}^{1-2s} \nabla \eta) \right) \Vert_{L^2(A_{\frac{r}{4},\frac{r}{2}})}\\
&\hspace{30pt} + e^{\tau\widetilde{\phi}(\ln(3r))} \Vert x_{n+1}^{\frac{1-2s}{2}} \tilde{w} \left( x_{n+1}^{2s-1} \nabla \cdot (x_{n+1}^{1-2s} \nabla \eta) \right) \Vert_{L^2(A_{3r,\frac{7r}{2}})}\\
&\hspace{30pt} + e^{\tau	\widetilde{\phi}(\ln(\frac{r}{4}))} \Vert x_{n+1}^{\frac{1-2s}{2}} \nabla \tilde{w} \Vert_{L^2(A_{\frac{r}{4},\frac{r}{2}})} + e^{\tau \widetilde{\phi}(\ln(3r))} \Vert x_{n+1}^{\frac{1-2s}{2}} \nabla \tilde{w} \Vert_{L^2(A_{3r,\frac{7r}{2}})} \bigg)\\
&\leq C \bigg( e^{\tau	\widetilde{\phi}(\ln(\frac{r}{4}))} \Vert x_{n+1}^{\frac{1-2s}{2}} \tilde{w} \Vert_{L^2(A_{\frac{r}{8},r})} + e^{\tau	\widetilde{\phi}(\ln(3r))} \Vert x_{n+1}^{\frac{1-2s}{2}} \tilde{w} \Vert_{L^2(A_{2r,4r})} \bigg)\\
&\leq C \bigg( e^{\tau	\widetilde{\phi}(\ln(\frac{r}{4}))} \Vert x_{n+1}^{\frac{1-2s}{2}} \tilde{w} \Vert_{L^2(B_r)} + e^{\tau	\widetilde{\phi}(\ln(3r))} \Vert x_{n+1}^{\frac{1-2s}{2}} \tilde{w} \Vert_{L^2(B_{4r})} \bigg),
\end{align*}
where for the second inequality we used Caccioppoli's inequality to bound the gradient contributions and that we are bounded away from $x_{n+1} = 0$ to bound $ x_{n+1}^{2s-1} \partial_{n+1} (x_{n+1}^{1-2s} \partial_{n+1} \eta)$. We observe that here the constant $C$ depends on $\delta_0$ and $r$. Dividing by $e^{\tau\widetilde{\phi}(\ln(2r))}$ gives
\begin{align*}
\Vert x_{n+1}^{\frac{1-2s}{2}} \tilde{w} \Vert_{L^2(A_{\frac{r}{2}, 2r})} \leq C \left( e^{\tau P} \Vert x_{n+1}^{\frac{1-2s}{2}} \tilde{w} \Vert_{L^2(B_r)} + e^{-\tau Q} \Vert x_{n+1}^{\frac{1-2s}{2}} \tilde{w} \Vert_{L^2(B_{4r})} \right),
\end{align*}
where $P:= \widetilde{\phi}(\ln(\frac{r}{4})) - \widetilde{\phi}(\ln(2r))$ and $Q := -(\widetilde{\phi}(\ln(3r)) - \widetilde{\phi}(\ln(2r)))$. We note that $c^{-1} \leq P,Q \leq c$ for some constant $c > 0$ uniformly in $r$. In particular, since $P$ is uniformly bounded away from $0$, we can add $\Vert x_{n+1}^{\frac{1-2s}{2}} \tilde{w} \Vert_{L^2(B_{\frac{r}{2}})}$ to both sides to fill up the annuli. By possibly changing $C$, we thus infer
\begin{align}\label{eq:preCarleman}
\Vert x_{n+1}^{\frac{1-2s}{2}} \tilde{w} \Vert_{L^2(B_{2r})} \leq C \left( e^{\tau P} \Vert x_{n+1}^{\frac{1-2s}{2}} \tilde{w} \Vert_{L^2(B_r)} + e^{-\tau Q} \Vert x_{n+1}^{\frac{1-2s}{2}} \tilde{w} \Vert_{L^2(B_{4r})} \right).
\end{align}
We seek to find $\tau \geq \tau_0$ (recall that $\tau_0$ depended on $R$, $n$, $s$ and $\delta_0$) such that
\begin{align*}
Ce^{-\tau Q} \Vert x_{n+1}^{\frac{1-2s}{2}} \tilde{w} \Vert_{L^2(B_{4r})} \leq \frac{1}{2} \Vert x_{n+1}^{\frac{1-2s}{2}} \tilde{w} \Vert_{L^2(B_{2r})}.
\end{align*}
We achieve this by choosing
\begin{align*}
\tau = \tau_0 - \frac{1}{Q} \ln \left( \frac{\Vert x_{n+1}^{\frac{1-2s}{2}} \tilde{w} \Vert_{L^2(B_{2r})}}{ 2C\Vert x_{n+1}^{\frac{1-2s}{2}} \tilde{w} \Vert_{L^2(B_{4r})}} \right).
\end{align*}
From \eqref{eq:preCarleman} we then deduce
\begin{align*}
\Vert x_{n+1}^{\frac{1-2s}{2}} \tilde{w} \Vert_{L^2(B_{2r})} \leq C e^{\tau_0} \frac{\Vert x_{n+1}^{\frac{1-2s}{2}} \tilde{w} \Vert_{L^2(B_{4r})}^{P/Q}}{\Vert x_{n+1}^{\frac{1-2s}{2}} \tilde{w} \Vert_{L^2(B_{2r})}^{P/Q}} \Vert x_{n+1}^{\frac{1-2s}{2}} \tilde{w} \Vert_{L^2(B_r)}.
\end{align*}
Rearranging terms and choosing $\alpha := \frac{Q}{P+Q}$ finally yields the three-balls-inequality
\begin{align*}
\Vert x_{n+1}^{\frac{1-2s}{2}} \tilde{w} \Vert_{L^2(B_{2r})} \leq C \Vert x_{n+1}^{\frac{1-2s}{2}} \tilde{w} \Vert_{L^2(B_r)}^\alpha \Vert x_{n+1}^{\frac{1-2s}{2}} \tilde{w} \Vert_{L^2(B_{4r})}^{1-\alpha}
\end{align*}
and the proof is finished.
\end{proof}

\subsubsection{Main Proposition}

We now proceed to the main proposition, which eventually yields the main results of this article. Here, we will apply the propagation of smallness arguments from the previous two propositions to estimate $v_1^f - v_2^f := \int_0^\infty t^{1-2s} \left( \tilde{u}_1^f (\cdot,t) - \tilde{u}_2^f (\cdot,t) \right) dt$.

\begin{prop}\label{prop:stability_estimate}
Let $(M,g_j)$ be two smooth, closed, connected $n$-dimensional Riemannian manifolds satisfying assumptions (A1) and (A2), $j \in \{1,2\}$. Let $O', O \subset M$ be smooth, open non-empty sets such that $O$ satisfies assumption (A3) and such that $O' \Subset O$. Let $L_{s,O}^j$ and $L_{1,O}^j$, $j \in \{1,2\}$, be the source-to-solution maps as in \eqref{eq:def_nonlocal_StoSOperator_manifold} and \eqref{eq:def_local_StoSOperator_manifold}, respectively. Let $f \in \widetilde{H}^{-s}_{\diamond}(O)$ with $(f,1)_{L^2(M)} = 0$.

There exist constants $\beta>0$ and $C>0$ such that, if for some $\varepsilon\in(0,\frac{1}{2})$
\begin{align*}
\|L_{s,O}^1(f)- L_{s,O}^2(f)\|_{H^{s}(O)} \leq \varepsilon \|f\|_{H^{-s}(M)},
\end{align*}
then it holds
\begin{align*}
\Vert L_{1,O}^1(f)- L_{1,O}^2(f) \Vert_{H^{2-s}(O')} \leq C \vert \log(\varepsilon) \vert^{-\beta} \Vert f \Vert_{H^{-s}(M)}.
\end{align*}
Here, $\beta$ and $C$ only depend on $s$, $n$, $M$, $O'$, $O$, $\theta_1$ and $\theta_2$ (from assumptions (A1) and (A2)).
\end{prop}

Compared to the result of Theorem \ref{thm:stability_nonlocal_local_single_meas} the proposition is not yet set in the most natural measurement topology. As a consequence, we will upgrade this by interpolation in a subsequent step. Proposition \ref{prop:stability_estimate} however already contains the essential information in the transference of stability.

\begin{proof}
In the following $C$ will denote a generic constant which may only depend on $n$, $s$, $O'$, $O$, $\theta_1$ and $\theta_2$, whereas the constants $C_1$ and $C_2$ may additionally depend on $\beta$. They may change from line to line.

Moreover, for notational convenience and to point out when the logarithmic estimate is needed, we will initially write
\begin{align*}
\omega(\varepsilon) := \vert\log(\varepsilon)\vert^{-\beta} < 1.
\end{align*}

Let $\tilde{u}_j^f \in H^1(M \times \R_+, x_{n+1}^{1-2s})$, $j \in \{1,2\}$, be the solutions to the Caffarelli-Silvestre extension equation with metric $g_j$ and source term $f$, see equation \eqref{eq:CS}. Since $L_{s,O}^j(f) = \restr{\tilde{u}_j^f(\cdot,0)}{O}$ it holds that
\begin{align}\label{eq:smallness_boundary_data_CSExt}
\Vert \tilde{u}_1^f(\cdot,0)- \tilde{u}_2^f(\cdot,0) \Vert_{H^{s}(O)} = \Vert L_{s,O}^1(f)- L_{s,O}^2(f) \Vert_{H^{s}(O)} \leq \varepsilon \Vert f \Vert_{H^{-s}(M)}.
\end{align}
Additionally let us recall that by \cite{R23} we have
\begin{align*}
L_{1,O}^j(f)(x) = c_s \int_0^\infty t^{1-2s} \tilde{u}_j^f(x,t) dt =: v_j^f(x),
\end{align*}
and thus
\begin{align}\label{eq:norm_local_maps_in_CS_terms}
\Vert L_{1,O}^1(f)- L_{1,O}^2(f) \Vert_{H^{2-s}(O')} = c_s \left\Vert \int_0^\infty t^{1-2s} \left( \tilde{u}_1^f (\cdot,t) - \tilde{u}_2^f (\cdot,t) \right) dt \right\Vert_{H^{2-s}(O')}.
\end{align}
Since $f$ has vanishing mean, we know that the functions $\tilde{u}_j^f(\cdot,x_{n+1})$ for $j \in \{1,2\}$ have vanishing mean for each $x_{n+1}>0$. Thus, by Poincaré's inequality (in tangential directions) and compactness of $M$, by assumption (A1) and the a-priori estimates for $\tilde{u}_j^f$ we have
\begin{align}\label{eq:Poincare}
\Vert \tilde{u}_j^f \Vert_{L^2(M \times \R_+, x_{n+1}^{1-2s})} \leq C \Vert \tilde{u}_j^f \Vert_{\dot{H}^1(M\times\R_+, x_{n+1}^{1-2s})} \leq C \Vert f \Vert_{H^{-s}(M)}.
\end{align}
We note that here the constant $C$ in particular depends on $\theta_1$ (from assumption (A1)).

Let $O'' \subset M$ be open, smooth such that $O' \Subset O'' \Subset O$. Since $v_1^f - v_2^f$ solves $(-\Delta_g) (v_1^f - v_2^f) = 0$ in $O$, we have by Caccioppoli's inequality, using that $g = \Id_{n \times n}$ in $O$, and the equation for $\nabla' \tilde{u}_j^f$ (where $\nabla'$ denotes the gradient taken with respect to the tangential variables)
\begin{align}\label{eq:Caccioppoli}
\begin{split}
\left\Vert \int_0^\infty t^{1-2s} \left( \tilde{u}_1^f(\cdot,t) - \tilde{u}_2^f(\cdot,t) \right) dt \right\Vert_{H^{2-s}(O')}
&\leq C \left\Vert \int_0^\infty t^{1-2s} \left( \tilde{u}_1^f(\cdot,t) - \tilde{u}_2^f(\cdot,t) \right) dt \right\Vert_{H^{2}(O')} \\
& \leq C \left\Vert \int_0^\infty t^{1-2s} \left( \tilde{u}_1^f(\cdot,t) - \tilde{u}_2^f(\cdot,t) \right) dt \right\Vert_{L^2(O'')},
\end{split}
\end{align}
where $C$ here depends on $O'$ and $O$ (and the precise choice of $O''$).

\textit{Step 1: Reduction to finite height in $t$.} We choose (for $\bar{\lambda}$ as in \eqref{eq:uniform_bounds_first_eigenvalue}, depending on $M$ and $\theta_1$)
\begin{align}\label{eq:relation_cutoff_modulus_cont}
L(\varepsilon) = \frac{2}{\bar{\lambda}} \vert \log(\omega(\varepsilon)) \vert.
\end{align}
Then, by Lemma \ref{lem:estimate_integral_modified_Bessel_function} we have
\begin{align}\label{eq:upper_integral}
\int_{\bar{\lambda} L(\varepsilon)}^\infty z^{1-s} K_s(z) dz \leq C_1 \omega(\varepsilon),
\end{align}
for some constant $C_1$ depending on $s$ and the precise form of the modulus of continuity $\omega$, in particular, it depends on $\beta$. By Lemma \ref{lem:apriori} and using \eqref{eq:upper_integral} we thus deduce
\begin{align}\label{eq:norm_upper_integral}
\begin{split}
&\left\Vert \int_{L(\varepsilon)}^\infty t^{1-2s} \left( \tilde{u}_1^f (\cdot,t) - \tilde{u}_2^f (\cdot,t) \right) dt \right\Vert_{L^2(O'')} \leq \left\Vert \int_{L(\varepsilon)}^\infty t^{1-2s} \left( \tilde{u}_1^f (\cdot,t) - \tilde{u}_2^f (\cdot,t) \right) dt \right\Vert_{H^1(O'')} \\
&\qquad \leq \left\Vert \int_{L(\varepsilon)}^\infty t^{1-2s} \tilde{u}_1^f (\cdot,t) dt \right\Vert_{H^{1}(O'')} + \left\Vert \int_{L(\varepsilon)}^\infty t^{1-2s} \tilde{u}_2^f (\cdot,t) dt \right\Vert_{H^{1}(O'')}\\
&\qquad \leq 2c_s \left( \int_{\bar{\lambda} L(\varepsilon)}^\infty z^{1-s} K_s(z) dz \right) \Vert f \Vert_{H^{-1}(M)} \leq C_1 \omega(\varepsilon) \Vert f \Vert_{H^{-s}(M)}.
\end{split}
\end{align}
Then, by \eqref{eq:norm_local_maps_in_CS_terms}, \eqref{eq:Caccioppoli} and \eqref{eq:norm_upper_integral} we infer
\begin{align*}
&\Vert L_{1,O}^1(f)- L_{1,O}^2(f) \Vert_{H^{2-s}(O')} = c_s \left\Vert \int_0^\infty t^{1-2s} \left( \tilde{u}_1^f(\cdot,t) - \tilde{u}_2^f(\cdot,t) \right) dt \right\Vert_{H^{2-s}(O')}\\
&\qquad \leq C \left( \left\Vert \int_0^{L(\varepsilon)} t^{1-2s} \left( \tilde{u}_1^f (\cdot,t) - \tilde{u}_2^f (\cdot,t) \right) dt \right\Vert_{L^2(O'')} + \left\Vert \int_{L(\varepsilon)}^\infty t^{1-2s} \left( \tilde{u}_1^f (\cdot,t) - \tilde{u}_2^f (\cdot,t) \right) dt \right\Vert_{L^2(O'')} \right)\\
&\qquad \leq C \left\Vert \int_0^{L(\varepsilon)} t^{1-2s} \left( \tilde{u}_1^f (\cdot,t) - \tilde{u}_2^f (\cdot,t) \right) dt \right\Vert_{L^2(O'')} + C_1 \omega(\varepsilon) \Vert f \Vert_{H^{-s}(M)}.
\end{align*}
Thus, it only remains to show that there exists some constant $C_2$ (which may additionally depend on $\beta$) such that for all $\varepsilon\in(0,\frac{1}{2})$ it holds that
\begin{align}\label{eq:claim_norm_lower_integral}
\left\Vert \int_0^{L(\varepsilon)} t^{1-2s} \left( \tilde{u}_1^f (\cdot,t) - \tilde{u}_2^f (\cdot,t) \right) dt \right\Vert_{L^2(O'')} \leq C_2 \omega(\varepsilon) \Vert f \Vert_{H^{-s}(M)}.
\end{align}

\textit{Step 2: Estimate of the finite height integral \eqref{eq:claim_norm_lower_integral}.} This step itself is divided into several substeps. For a visualization and an outline of the proof see Figure \ref{fig:schematic_argument_stability}.

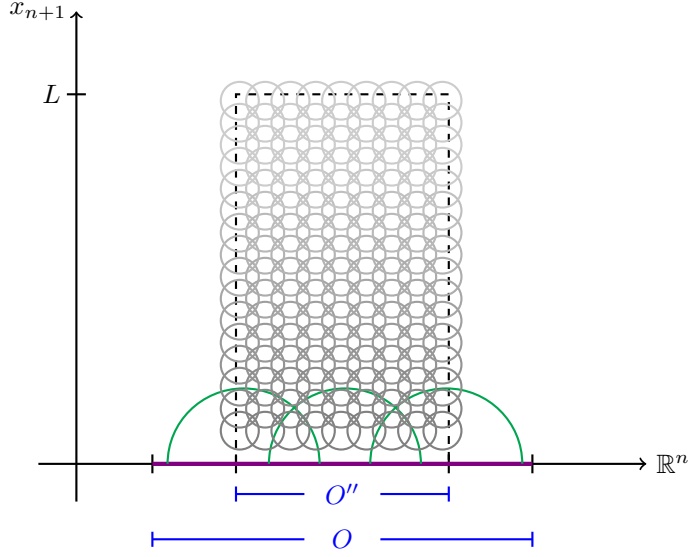
\begin{figure}
	\begin{center}
  	\begin{tikzpicture}[thick,scale=1,every node/.style={scale=1}]
  	\def\r{1};
  	\draw[->] (-0.5,0) -- (7.5,0) node[right] {$\R^n$};
  	\draw[->] (0,-0.5) -- (0,6) node[left] {$x_{n+1}$};
  	\draw (-1/8,4.9) -- (1/8,4.9) node[left=2mm] {$L$};
  	\draw[dashed] (2.1,0) -- (2.1,4.9) -- (4.9,4.9) -- (4.9,0);
  	\draw (2.1,-1/8) -- (2.1,1/8); \draw (4.9,-1/8) -- (4.9,1/8);
  	\draw (1,-1/8) -- (1,1/8); \draw (6,-1/8) -- (6,1/8);
  	\draw[ultra thick,violet] (1,0) -- (6,0);
  	\draw[thick,blue] (2.1,-0.4) -- (3,-0.4); \draw[thick,blue] (4,-0.4) -- (4.9,-0.4); \draw[thick,blue] (2.1,-0.5) -- (2.1,-0.3); \draw[thick,blue] (4.9,-0.5) -- (4.9,-0.3);
  	\node[blue] at (3.5,-0.4) {$O''$};
  	\draw[thick,blue] (1,-1) -- (3,-1); \draw[thick,blue] (4,-1) -- (6,-1); \draw[thick,blue] (1,-1.1) -- (1,-0.9); \draw[thick,blue] (6,-1.1) -- (6,-0.9);
  	\node[blue] at (3.5,-1) {$O$};
  	
  	\foreach \i in {0,...,2}{
  		\draw[Green] (3.2+4/3*\i,0) arc(0:180:\r);
  	}

  	\foreach \j in {0,...,15}{
  		\foreach \i in {0,...,8}{
  			\tikzmath{
  				\opacity = max(40, 100-5*\j);
  			}
  			\draw[gray!\opacity] (2.15+\i*\r/3, 11/25*\r + 7/24*\j*\r) circle(\r/4);
		}
	}
  	\end{tikzpicture}
	\end{center}
	\caption{Schematic visualization of Step 2 in the proof of Proposition \ref{prop:stability_estimate}. We have smallness of the data on the boundary $O \times \{0\}$ (violet part). By Proposition \ref{prop:bbucp} we transfer the smallness into the bulk, more precisely, to the $(n+1)$-dimensional half-balls $B_{c\bar{r}}^+(x_j)$ (green half-balls). Then by an iterative application of Proposition \ref{prop:3ballsinequ} we are able to propagate the smallness upwards in the bulk along a sequence of balls (gray balls) to a certain cut-off height $L$. For visualization, sizes and positions of balls are not true-to-scale.}
\label{fig:schematic_argument_stability}
\end{figure}

\textit{Step 2.1: Set-up.} The following set-up is depicted in more detail in Figure \ref{fig:detailed_setup_balls}. Let $\left\{ B_{c\bar{r}}'(x_j') \right\}_{j=1,\dots,\bar{N}_O}$ be a covering of $O''$ with $n$-dimensional balls of radius $c\bar{r}$ centered at $x_j' \in O''$, where $c = c(s,n) \in (0,\frac{1}{2})$ is the constant from Proposition \ref{prop:bbucp}, such that
\begin{align*}
O'' \subset \bigcup_{j=1}^{\bar{N}_O} B_{c\bar{r}}'(x_j') \subset \bigcup_{j=1}^{\bar{N}_O} B_{4\bar{r}}'(x_j') \subset O.
\end{align*}
The number of balls, $\bar{N}_O$, only depends on $s$, $n$, $O''$ and $O$. Let $B_{c\bar{r}}^{+}(x_j) := B_{c\bar{r}}(x_j) \cap \R^{n+1}_+$ be the (restricted to the upper half-plane) $(n+1)$-dimensional half-balls of radius $c\bar{r}$ centered at $x_j := (x_j', 0)$. There exists $0 < \delta_0 \leq 1$ such that
\begin{align*}
O'' \times [0, \delta_0) \subset \bigcup_{j=1}^{\bar{N}_O} \overline{B_{c\bar{r}}^+(x_j)}.
\end{align*}
Adding balls if necessary, one can achieve that $\delta_0 > \frac{c\bar{r}}{2}$. Here, $\delta_0$ also only depends on $s$, $n$, $O''$ and $O$.

Now, let $\left\{ B_{2r}'((x_i',\frac{\delta_0}{2})) \right\}_{i=1,\dots,N_O}$ be a covering of $O'' \times \{\frac{\delta_0}{2}\}$ with $n$-dimensional balls of radius $2r$ with $r \in (0,\frac{\delta_0}{16})$ centered at $x_i'$ such that
\begin{align*}
O'' \times \left\{\frac{\delta_0}{2}\right\} \ \subset \ \bigcup_{i=1}^{N_O} B_{2r}'((x_i', \frac{\delta_0}{2})) \ \subset \ O \times \left\{ \frac{\delta_0}{2} \right\}.
\end{align*}
The number of balls $N_O$ only depends on $s$, $n$, $O''$ and $O$. Let $B_{2r}(x_i^{(k)})$ be the $(n+1)$-dimensional balls (in the upper half plane) of radius $2r$ centered at $x_i^{(k)} := (x_i', y_k)$, where $y_k$ is such that
\begin{align*}
O'' \times [t_{k-1},t_k) \subset \bigcup_{i=1}^{N_O} B_{2r}(x_i^{(k-1)})
\end{align*}
for an appropriately chosen partition $t_0, t_1, \dots, t_N$ of the interval $[\frac{\delta_0}{2},L(\varepsilon))$ and such that $B_r(x_i^{(k)}) \subset B_{2r}(x_i^{(k-1)})$ for $k \in \{1,\dots,N\}$. Let $B_{2r}^k$ denote the ball for which the quantity $\Vert \tilde{u}_1^f - \tilde{u}_2^f \Vert_{L^2(B_{2r}^+(x_i^{(k)}), x_{n+1}^{1-2s})}$ is maximal over $i \in \{1,\dots, N_O\}$.

We remark that for all $i \in \{1,\dots,N_O\}$ there exists $j \in \{1,\dots,\bar{N}_O\}$ such that $B_r(x_i^{(0)}) \subset B_{c\bar{r}}^+(x_j')$ and that $B_{4r}(x_i^{(k)}) \subset O \times (\frac{\delta_0}{4}, \infty)$ for all $i \in \{1,\dots,N_O\}$ and all $k \in \{0,\dots,N\}$. Hence, Propositions \ref{prop:bbucp} and \ref{prop:3ballsinequ} will be applicable.

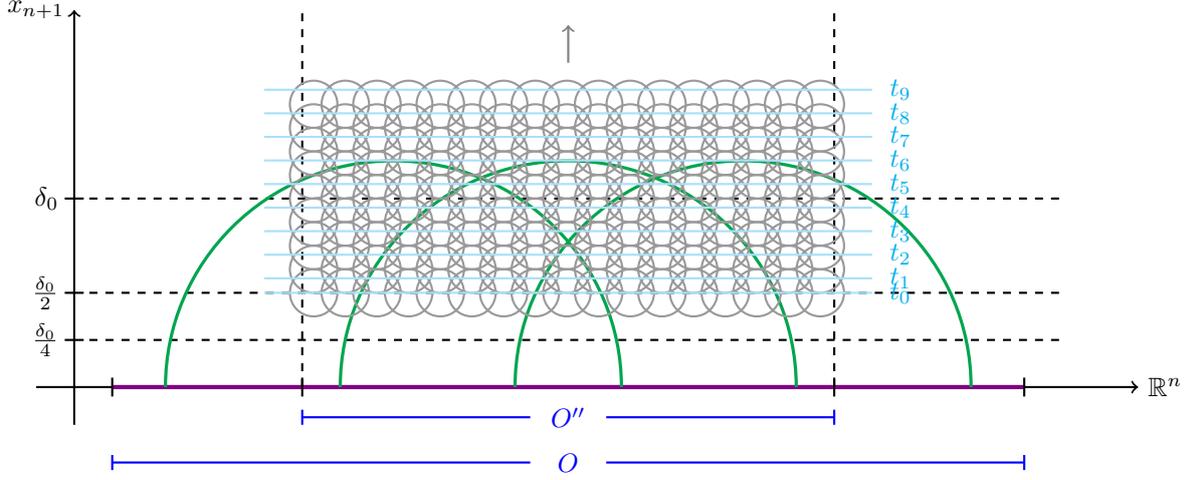
\begin{figure}
	\begin{center}
  	\begin{tikzpicture}[thick,scale=1,every node/.style={scale=1}]
  	\def\r{3}; \def\d0{2.5};
  	\draw[->] (-0.5,0) -- (14,0) node[right] {$\R^n$};
  	\draw[->] (0,-0.5) -- (0,5) node[left] {$x_{n+1}$};
  	\draw[ultra thick,violet] (0.5,0) -- (12.5,0);
  	\draw (-1/8,\d0) -- (1/8,\d0) node[left=2mm] {$\delta_0$}; \draw[dashed] (0.2,\d0) -- (13,\d0);
  	\draw (-1/8,\d0/2) -- (1/8,\d0/2) node[left=2mm] {$\frac{\delta_0}{2}$}; \draw[dashed] (0.2,\d0/2) -- (13,\d0/2);
  	\draw (-1/8,\d0/4) -- (1/8,\d0/4) node[left=2mm] {$\frac{\delta_0}{4}$}; \draw[dashed] (0.2,\d0/4) -- (13,\d0/4);
  	\draw[dashed] (3,0) -- (3,5); \draw[dashed] (10,0) -- (10,5);
  	\draw (3,-1/8) -- (3,1/8); \draw (10,-1/8) -- (10,1/8);
  	\draw (0.5,-1/8) -- (0.5,1/8); \draw (12.5,-1/8) -- (12.5,1/8);
  	\draw[thick,blue] (3,-0.4) -- (6,-0.4); \draw[thick,blue] (7,-0.4) -- (10,-0.4); \draw[thick,blue] (3,-0.5) -- (3,-0.3); \draw[thick,blue] (10,-0.5) -- (10,-0.3);
  	\node[blue] at (6.5,-0.4) {$O''$};
  	\draw[thick,blue] (0.5,-1) -- (6,-1); \draw[thick,blue] (7,-1) -- (12.5,-1); \draw[thick,blue] (0.5,-1.1) -- (0.5,-0.9); \draw[thick,blue] (12.5,-1.1) -- (12.5,-0.9);
  	\node[blue] at (6.5,-1) {$O$};
  	
  	\draw[very thick,Green] (7.2,0) arc(0:180:\r); \draw[very thick,Green] (9.5,0) arc(0:180:\r); \draw[very thick,Green] (11.8,0) arc(0:180:\r);
  	
  	\draw[->,black!50] (6.5,4.3) -- (6.5,4.8);
  	
  	\foreach \j in {0,...,8}{
  		\foreach \i in {0,...,16}{
  			\draw[black!40] (3.15 + \i*\d0/6, \d0/2 + \j*\d0/8) circle(\d0/8);
		}
	}
	
	\draw[cyan!30] (2.5,\d0/2) -- (10.5,\d0/2) node[cyan,right=1mm] {$t_0$};
	\foreach \j in {0,...,8}{
		\tikzmath{
			int \index;
			\index = \j+1;
		}
		\draw[cyan!30] (2.5,\d0/2+\d0/13+\j*\d0/8) -- (10.5,\d0/2+\d0/13+\j*\d0/8) node[cyan,right=1mm] {$t_{\index}$};
	}
  	\end{tikzpicture}
	\end{center}
	\caption{More precise visualization of the set-up of balls in Step 2 of the proof of Proposition \ref{prop:stability_estimate}. The green coloured half-balls represent the $(n+1)$-dimensional half-balls $B_{c\bar{r}}^+(x_j)$ and the gray coloured balls are the $(n+1)$-dimensional balls $B_{2r}(x_i^{(k)})$. For visualization, sizes or positions of the balls may not be perfectly true-to-scale.}
\label{fig:detailed_setup_balls}
\end{figure}

We will split the estimate into two terms, one term near the boundary and one for the bulk,
\begin{equation}\label{eq:norm_lower_integral_1}
\begin{aligned}
&\left\Vert \int_0^{L(\varepsilon)} t^{1-2s} \left( \tilde{u}_1^f (\cdot,t) - \tilde{u}_2^f (\cdot,t) \right) dt \right\Vert_{L^2(O'')}^2 = \int_{O''} \left\vert \int_0^{L(\varepsilon)} t^{1-2s} \left( \tilde{u}_1^f(x,t) - \tilde{u}_2^f(x,t) \right) dt \right\vert^2 dx\\
&\qquad \leq 2 \int_{O''} \left\vert \int_0^{\delta_0} t^{1-2s} \left( \tilde{u}_1^f(x,t) - \tilde{u}_2^f(x,t) \right) dt \right\vert^2 dx \\
& \qquad  \qquad  + 2 \int_{O''} \left\vert \int_{\delta_0}^{L(\varepsilon)} t^{1-2s} \left( \tilde{u}_1^f(x,t) - \tilde{u}_2^f(x,t) \right) dt \right\vert^2 dx.
\end{aligned}
\end{equation}

\textit{Step 2.2: Estimate near the boundary.} We estimate the first term on the right hand side of \eqref{eq:norm_lower_integral_1}, i.e., the term near the boundary, by observing that
\begin{equation}\label{eq:norm_lower_integral_boundary_1}
\begin{aligned}
\int_{O''} &\left\vert \int_0^{\delta_0} t^{1-2s} \left( \tilde{u}_1^f(x,t) - \tilde{u}_2^f(x,t) \right) dt \right\vert^2 dx\\
&\qquad \leq  \Vert 1 \Vert_{L^2((0,\delta_0),x_{n+1}^{1-2s})}^2 \int_{O''} \int_0^{\delta_0} t^{1-2s} \vert \tilde{u}_1^f(x,t) - \tilde{u}_2^f(x,t) \vert^2 dt dx\\
&\qquad \leq \frac{1}{2-2s} \delta_0^{2-2s} \Vert \tilde{u}_1^f - \tilde{u}_2^f \Vert_{L^2(O'' \times [0,\delta_0), x_{n+1}^{1-2s})}^2\\
&\qquad \leq C \sum_{j=1}^{\bar{N}_O} \Vert \tilde{u}_1^f - \tilde{u}_2^f \Vert_{L^2(B_{c\bar{r}}^+(x_j), x_{n+1}^{1-2s})}^2\\
&\qquad \leq C \Vert f \Vert_{H^{-s}(M)}^2 \varepsilon^{2(1-\bar{\alpha})},
\end{aligned}
\end{equation}
where for the first inequality we used Hölder's inequality in the $t$-variable. The last inequality holds by applying Proposition \ref{prop:bbucp} to $\tilde{w} := \tilde{u}_1^f - \tilde{u}_2^f$ in the following way
\begin{equation}\label{eq:norm_lower_integral_PropBBUCP}
\begin{aligned}
\Vert \tilde{u}_1^f& - \tilde{u}_2^f \Vert_{L^2(B_{c\bar{r}}^+, x_{n+1}^{1-2s})}\\
&\leq C \left( \Vert \tilde{u}_1^f(\cdot,0) - \tilde{u}_2^f(\cdot,0) \Vert_{H^s(B'_{3\bar{r}})} + \Vert \lim_{x_{n+1}\to0} x_{n+1}^{1-2s} \partial_{n+1} \left( \tilde{u}_1^f - \tilde{u}_2^f \right) \Vert_{H^{-s}(B'_{3\bar{r}})} \right)^{1-\bar{\alpha}}\\
&\qquad \times \bigg( \Vert \tilde{u}_1^f - \tilde{u}_2^f \Vert_{L^2(B_{\bar{r}}^+,x_{n+1}^{1-2s})} + \Vert \tilde{u}_1^f(\cdot,0) - \tilde{u}_2^f(\cdot,0) \Vert_{H^s(B'_{3\bar{r}})}\\
& \qquad \qquad + \Vert \lim_{x_{n+1}\to0} x_{n+1}^{1-2s} \partial_{n+1} \left( \tilde{u}_1^f - \tilde{u}_2^f \right) \Vert_{H^{-s}(B'_{3\bar{r}})} \bigg)^{\bar{\alpha}}\\
&\leq C \Vert \tilde{u}_1^f(\cdot,0) - \tilde{u}_2^f(\cdot,0) \Vert_{H^s(B_{3\bar{r}}')} + C \Vert \tilde{u}_1^f(\cdot,0) - \tilde{u}_2^f(\cdot,0) \Vert_{H^s(B_{3 \bar{r}}')}^{1-\bar{\alpha}} \Vert \tilde{u}_1^f - \tilde{u}_2^f \Vert_{L^2(B_{\bar{r}}^+,x_{n+1}^{1-2s})}^{\bar{\alpha}}\\
&\leq C \left( \varepsilon \Vert f \Vert_{H^{-s}(M)} \right) + C \left( \varepsilon \Vert f \Vert_{H^{-s}(M)} \right)^{1-\bar{\alpha}} \Vert \tilde{u}_1^f - \tilde{u}_2^f \Vert_{L^2(B_{\bar{r}}^+, x_{n+1}^{1-2s})}^{\bar{\alpha}}\\
&\leq C \Vert f \Vert_{H^{-s}(M)} \varepsilon^{1-\bar{\alpha}},
\end{aligned}
\end{equation}
where the second inequality follows since $(-\Delta_{g_1})^s u_1 = f = (-\Delta_{g_2})^s u_2$ in $M$ (which by the Caffarelli-Silvestre type realization from Section \ref{sec:CS} implies that $\lim_{x_{n+1}\to0} x_{n+1}^{1-2s} \partial_{n+1} \left( \tilde{u}_1^f - \tilde{u}_2^f \right) = 0$ in $M$). For the second to last inequality we used \eqref{eq:smallness_boundary_data_CSExt} and for the last inequality we used \eqref{eq:Poincare} and that $\varepsilon<1$.

\textit{Step 2.3: Estimate in the bulk.} The second term on the right hand side of \eqref{eq:norm_lower_integral_1}, the bulk term, can be estimated by
\begin{equation}\label{eq:norm_lower_integral_bulk_1}
\begin{aligned}
&\int_{O''} \left\vert \int_{\delta_0}^{L(\varepsilon)} t^{1-2s} \left( \tilde{u}_1^f(x,t) - \tilde{u}_2^f(x,t) \right) dt \right\vert^2 dx\\
&\qquad \leq \Vert 1 \Vert_{L^2((0,L(\varepsilon)),x_{n+1}^{1-2s})}^2 \int_{O''} \int_{\delta_0}^{L(\varepsilon)} t^{1-2s} \vert \tilde{u}_1^f(x,t) - \tilde{u}_2^f(x,t) \vert^2 dt dx\\
&\qquad = \frac{1}{2-2s} L(\varepsilon)^{2-2s} \sum_{k=1}^N \Vert \tilde{u}_1^f - \tilde{u}_2^f \Vert_{L^2(O''\times[t_{k-1},t_k),x_{n+1}^{1-2s})}^2\\
&\qquad \leq C \vert \log(\omega(\varepsilon)) \vert^{2-2s} \sum_{k=1}^N \Vert \tilde{u}_1^f - \tilde{u}_2^f \Vert_{L^2(B_{2r}^k,x_{n+1}^{1-2s})}^2,
\end{aligned}
\end{equation}
where for the first inequality in \eqref{eq:norm_lower_integral_bulk_1} we have again used Hölder's inequality in the $t$ variable and for the last inequality we used the relation between the cut-off height $L(\varepsilon)$ and the modulus of continuity $\omega(\varepsilon)$ from \eqref{eq:relation_cutoff_modulus_cont}. We have $N \sim \frac{L(\varepsilon)}{r}$ and since $r$ only depends on $s$, $n$, $O''$ and $O$ we write
\begin{align}\label{eq:bound_on_N}
N \leq C \vert \log(\omega(\varepsilon)) \vert.
\end{align}
Now, using Proposition \ref{prop:3ballsinequ} for $\tilde{w} = \tilde{u}_1^f - \tilde{u}_2^f$ iteratively, we find
\begin{equation}\label{eq:norm_lower_integral_bulk_2}
\begin{aligned}
\sum_{k=1}^N &\Vert \tilde{u}_1^f - \tilde{u}_2^f \Vert_{L^2(B_{2r}^k, x_{n+1}^{1-2s})}^2 \leq \sum_{k=1}^N C \Vert \tilde{u}_1^f - \tilde{u}_2^f \Vert_{L^2(B_{4r}^k, x_{n+1}^{1-2s})}^{2(1-\alpha)} \Vert \tilde{u}_1^f - \tilde{u}_2^f \Vert_{L^2(B_{r}^k, x_{n+1}^{1-2s})}^{2\alpha}\\
&\leq \sum_{k=1}^N C \left( \Vert \tilde{u}_1^f \Vert_{L^2(M \times \R_+, x_{n+1}^{1-2s})} + \Vert \tilde{u}_2^f \Vert_{L^2(M \times \R_+, x_{n+1}^{1-2s})} \right)^{2(1-\alpha)} \Vert \tilde{u}_1^f - \tilde{u}_2^f \Vert_{L^2(B_{2r}^{k-1}, x_{n+1}^{1-2s})}^{2\alpha}\\
&\leq \sum_{k=1}^N C (2 \Vert f \Vert_{H^{-s}(M)})^{2\sum_{l=1}^k (1-\alpha) \alpha^{l-1}} \Vert \tilde{u}_1^f - \tilde{u}_2^f \Vert_{L^2(B_{r}^0, x_{n+1}^{1-2s})}^{2\alpha^k}\\
&\leq \sum_{k=1}^N C \Vert f \Vert_{H^{-s}(M)}^2 \left(\varepsilon^{1-\bar{\alpha}}\right)^{2\alpha^k}\\
&\leq CN \Vert f \Vert_{H^{-s}(M)}^2 \max_{k\in\{1,\dots,N\}} \left\{ \left( \varepsilon^{1-\bar{\alpha}} \right)^{2\alpha^k} \right\}\\
&\leq CN \Vert f \Vert_{H^{-s}(M)}^2 \varepsilon^{(1-\bar{\alpha}) 2\alpha^N}.
\end{aligned}
\end{equation}
We recall that $\bar{\alpha}$ was the constant from \eqref{eq:norm_lower_integral_boundary_1}.
Here, for the third inequality we have used \eqref{eq:Poincare}. The fourth inequality follows by noting that $B_r^0 \subset B_{c\bar{r}}^+(x_j)$ for some $j \in \{1,\dots,\bar{N}_O\}$ and by using the estimate \eqref{eq:norm_lower_integral_PropBBUCP} from above. Additionally, for the fourth inequality we have used that $\sum_{l=1}^k (1-\alpha)\alpha^{l-1} = 1-\alpha^k$. 

Thus, from \eqref{eq:norm_lower_integral_bulk_1}, \eqref{eq:bound_on_N} and \eqref{eq:norm_lower_integral_bulk_2} we deduce
\begin{equation}\label{eq:norm_lower_integral_bulk_3}
\begin{aligned}
&\int_{O''} \left\vert \int_{\delta_0}^{L(\varepsilon)} t^{1-2s} \left( \tilde{u}_1^f(x,t) - \tilde{u}_2^f(x,t) \right) dt \right\vert^2 dx\\
&\qquad \leq C \Vert f \Vert_{H^{-s}(M)}^2 \vert \log(\omega(\varepsilon)) \vert^{2-2s+1} \varepsilon^{2(1-\bar{\alpha}) \alpha^{C \vert \log(\omega(\varepsilon)) \vert}}\\
&\qquad = C \Vert f \Vert_{H^{-s}(M)}^2 \vert \log(\omega(\varepsilon)) \vert^{3-2s} \varepsilon^{2(1-\bar{\alpha}) \omega(\varepsilon)^{-C\log(\alpha)}}.
\end{aligned}
\end{equation}
For the last equality we have used that $\omega(\varepsilon)<1$ and hence, $\log(\omega(\varepsilon)) < 0$.

\textit{Step 2.4: Conclusion.} At this point we now switch back to writing $\omega(\varepsilon) = \vert \log(\varepsilon) \vert^{-\beta}$. In order to conclude our argument we first combine estimates \eqref{eq:norm_lower_integral_1}, \eqref{eq:norm_lower_integral_boundary_1} and \eqref{eq:norm_lower_integral_bulk_3} to get
\begin{equation}\label{eq:norm_lower_integral_2}
\begin{aligned}
&\left\Vert \int_0^{L(\varepsilon)} t^{1-2s} \left( \tilde{u}_1^f (\cdot,t) - \tilde{u}_2^f (\cdot,t) \right) dt \right\Vert_{L^2(O'')}\\
&\qquad \leq C \Vert f \Vert_{H^{-s}(M)} \varepsilon^{1-\bar{\alpha}} + C \Vert f \Vert_{H^{-s}(M)} \log\left( \vert \log(\varepsilon) \vert^{\beta} \right)^{\frac{3}{2}-s} \varepsilon^{(1-\bar{\alpha}) \vert\log(\varepsilon)\vert^{C\beta\log(\alpha)}}.
\end{aligned}
\end{equation}
At this point it is essential that the constant $C$ did not depend on $\beta$. Hence, we can take
\begin{align}\label{eq:constraint_on_beta}
0 < \beta < \frac{1}{C \vert\log(\alpha)\vert}.
\end{align}

We claim that the term on the right hand side of \eqref{eq:norm_lower_integral_2} converges faster to $0$ than $\omega(\varepsilon) = \vert \log(\varepsilon)\vert^{-\beta}$ as $\varepsilon\to0$, giving the existence of some constant $C_2>0$ (which may now additionally depend on $\beta$) such that \eqref{eq:claim_norm_lower_integral} holds for all $\varepsilon\in(0,\frac{1}{2})$.

Indeed, for the first term the faster convergence is clear. On the other hand, we compare $\vert\log(\varepsilon)\vert^{-\beta}$ to the second term on the right hand side of \eqref{eq:norm_lower_integral_2}
\begin{align*}
&\frac{1}{\vert\log(\varepsilon)\vert^{-\beta}} \left( \log\left( \vert \log(\varepsilon) \vert^{\beta} \right)^{\frac{3}{2}-s} \varepsilon^{ (1-\bar{\alpha}) \vert\log(\varepsilon)\vert^{C\beta\log(\alpha)}} \right)\\
&\qquad = \vert \log(\varepsilon) \vert^\beta \log\left(\vert\log(\varepsilon)\vert^\beta\right)^{\frac{3}{2}-s} \exp\left( \log(\varepsilon) (1-\bar{\alpha}) \vert\log(\varepsilon)\vert^{C\beta\log(\alpha)} \right)\\
&\qquad = \vert\log(\varepsilon)\vert^\beta \log\left(\vert\log(\varepsilon)\vert^\beta\right)^{\frac{3}{2}-s} \exp\left[ -(1-\bar{\alpha}) \frac{\vert\log(\varepsilon)\vert}{\vert\log(\varepsilon)\vert^{C\beta\vert\log(\alpha)\vert}} \right]\\
&\qquad \longrightarrow 0
\end{align*}
as $\varepsilon\to0$, since (writing $z=\vert\log(\varepsilon)\vert$)
\begin{align*}
z^\beta \log(z^\beta)^{\frac{3}{2}-s} \exp\left[ -(1-\bar{\alpha}) \frac{z}{z^{C\beta\vert\log(\alpha)\vert}} \right] \longrightarrow 0
\end{align*}
as $z\to\infty$. Thus, there exists $C_2>0$ such that \eqref{eq:claim_norm_lower_integral} holds for all $\varepsilon\in(0,\frac{1}{2})$ with $C_2$ depending on $s$, $n$, $M$, $O''$, $O$, $\theta_1$, $\theta_2$, $\alpha$ and $\beta$. Since $\alpha$ only depends on $r$, which in turn depends on $s$, $n$, $O''$ and $O$, this concludes the proof.
\end{proof}

\begin{rmk}
We remark that with our approach we cannot do substantially better than the logarithmic stability estimate we deduced. Indeed, if we would try to prove some algebraic decay, say $\omega(\varepsilon) = \varepsilon^\gamma$ for some $\gamma>0$, i.e.
\begin{align*}
\Vert L_{1,O}^1(f) - L_{1,O}^2(f) \Vert_{H^1(O')} \leq C \varepsilon^{\gamma} \Vert f \Vert_{H^{-s}(M)},
\end{align*}
then even switching to balls exponentially increasing in size, we would at best get a bound for $N$ of $N \leq C \log( \vert \log(\varepsilon^{\gamma}) \vert)$. However, here we would already have to be careful since large balls would eventually reach out of the set $O \times \R_+$, and the three-balls-inequality, Proposition \ref{prop:3ballsinequ}, cannot be applied any longer. Instead one would rather have to consider ellipses or blocks increasing exponentially in size only in $x_{n+1}$-direction. With this strategy one may be able to derive a modulus of continuity in between algebraic and logarithmic decay, in the form of $\exp\left( -C\vert\log(\varepsilon)\vert^{\overline{\gamma}} \right)$ for $\overline{\gamma} \in (0,1)$.

Nevertheless, assuming the new bound on $N$, for the algebraic modulus of continuity the essential estimate (based on the estimates in \eqref{eq:norm_lower_integral_bulk_2}, \eqref{eq:norm_lower_integral_bulk_3} and ignoring some factors) in Step 2.4 in the previous proof would read
\begin{align*}
\varepsilon^{\alpha^N} &\leq \varepsilon^{\alpha^{C\log(\vert\log(\varepsilon^{\gamma})\vert)}} = \exp \left( \log(\varepsilon) \alpha^{C \log(\vert\log(\varepsilon^\gamma)\vert)} \right)
\end{align*}
which converges more slowly to $0$ than $\varepsilon^\gamma = \exp( \log(\varepsilon) \gamma)$ as $\varepsilon \to 0$, since $\alpha \in (0,1)$ and therefore $\alpha^{C \log(\vert\log(\varepsilon^\gamma)\vert)} \to 0$ as $\varepsilon \to 0$.

We do not know whether this logarithmic loss is necessary or only an artefact of our proof (see also the discussion in Section \ref{sec:instability} below).
\end{rmk}

\subsection{Proof of the main results}

With the above quantitative unique continuation results in hand, we turn to the proofs of our main results. We begin with the setting of the Calder\'on problem on closed, smooth, connected manifolds.
In order to prove Theorem \ref{thm:stability_nonlocal_local_single_meas} it remains to upgrade the deduced estimates to the ones in the natural topologies.

\begin{proof}[Proof of Theorem \ref{thm:stability_nonlocal_local_single_meas}]
On the one hand, in our discussion of Proposition \ref{prop:stability_estimate} above we have shown that
\begin{align*}
\|L_{1,O}^1(f)-L_{1,O}^2(f)\|_{H^{2-s}(O')} \leq C \vert \log(\varepsilon) \vert^{-\beta} \|f\|_{H^{-s}(M)}.
\end{align*}
On the other hand, by ellipticity, we have that
\begin{align*}
\|L_{1,O}^1(f)-L_{1,O}^2(f)\|_{H^{1-s}(O')} \leq \Vert L_{1,O}^1(f) \Vert_{H^{1-s}(M)} + \Vert L_{1,O}^2(f) \Vert_{H^{1-s}(M)} \leq C \Vert f \Vert_{H^{-1-s}(M)}.
\end{align*}
Indeed, this follows from an  estimate as in Lemma \ref{lem:apriori}:
\begin{align*}
& \Vert L_{1,O}^j(f) \Vert_{H^{1-s}(M)} \leq C \left\| \int\limits_{0}^{\infty} t^{1-2s} \tilde{u}_j^f(\cdot,t) dt \right\|_{H^{1-s}(M)} \\
&\quad\leq C \sup_{\substack{h \in H^{s-1}(M),\\ \Vert h \Vert_{H^{s-1}(M)} = 1}} \sum\limits_{k=1}^{\infty} (f,\phi_k)_{L^2(M)} (h,\phi_k)_{L^2(M)} \int\limits_{0}^{\infty} t^{1-2s} K_{s}(\sqrt{\lambda_k} t) \left( \frac{t}{\sqrt{\lambda_k}} \right)^s dt\\
&\quad \leq C \sup_{\substack{h \in H^{s-1}(M),\\ \Vert h \Vert_{H^{s-1}(M)} = 1}} \ \sum_{k=1}^\infty \left(\int\limits_{0}^{\infty} z^{1-s} K_s(z) dz \right) \lambda_k^{-1} \vert (f,\phi_k)_{L^2(M)} \vert \ \vert (h,\phi_k)_{L^2(M)} \vert \\
&\quad \leq C \sup_{\substack{h \in H^{s-1}(M),\\ \Vert h \Vert_{H^{s-1}(M)} = 1}} \ \left(\int\limits_{0}^{\infty} z^{1-s} K_s(z) dz \right) \left( \sum_{k=1}^\infty \lambda_k^{-1-s} \vert (f,\phi_k)_{L^2(M)} \vert^2 \right)^{1/2} \left( \sum_{k=1}^\infty \lambda_k^{s-1} \vert (h,\phi_k)_{L^2(M)} \vert^2 \right)^{1/2}\\
&\quad = C \left(\int\limits_{0}^{\infty} z^{1-s} K_s(z) dz \right) \Vert f \Vert_{H^{-1-s}(M)}.
\end{align*}
Using that since $O'$, $O$, $M$ are assumed to be smooth, the interpolation properties (see e.g. \cite{CWHM15,T95})
\begin{align*}
H^{1}(O')=(H^{2-s}(O'),H^{1-s}(O'))_{s,2} \quad \text{and} \quad  \widetilde{H}^{-1}(O)=(\widetilde{H}^{-s}(O),\widetilde{H}^{-1-s}(O))_{s,2}
\end{align*}
with norm-equivalence up to constants depending only on $\theta_1$ and $\theta_2$ (but otherwise independent of the metric $g$; recall the results from Section \ref{ssec:prel_equivalence_of_norms} above) then imply by real interpolation (see \cite{BL76})
\begin{align*}
\|L_{1,O}^1(f)-L_{1,O}^2(f)\|_{H^{1}(O')} \leq C \vert \log(\varepsilon) \vert^{-s\beta} \|f\|_{H^{-1}(M)}.
\end{align*}
Redefining $\beta$ concludes the proof of Theorem \ref{thm:stability_nonlocal_local_single_meas}.
\end{proof}

With Theorem \ref{thm:stability_nonlocal_local_single_meas} in hand, the result of Theorem \ref{thm:stability_nonlocal_local_infinite_meas} follows immediately.

\begin{proof}[Proof of Theorem \ref{thm:stability_nonlocal_local_infinite_meas}]
We apply Theorem \ref{thm:stability_nonlocal_local_single_meas} for all $f \in \widetilde{H}^{-s}_{\diamond}(O')$ with $\Vert f \Vert_{H^{-s}(M)} = 1$.
\end{proof}

Lastly, we show how the previous reduction results can be applied in order to derive stability estimates for the nonlocal problem from the local problem.

\begin{proof}[Proof of Corollary \ref{cor:transfer_of_stability_manifold_setting}]
From Theorem \ref{thm:stability_nonlocal_local_infinite_meas} we know that we can estimate the norm of the difference of the local source-to-solution maps by the norm difference of the nonlocal maps, i.e.
\begin{align*}
\Vert L_{1,O'}^1 - L_{1,O'}^2 \Vert_{H^{-1}_{\diamond}(O') \to H^1(O')} \leq C \vert \log( \Vert L_{s,O}^1 - L_{s,O}^2 \Vert_{H^{-1}_{\diamond}(O') \to H^1(O')} ) \vert^{-\beta}.
\end{align*}
With the assumed stability estimate for the local problem, \eqref{eq:assump_cor}, we then infer
\begin{align*}
\|g_1 - g_2 \|_{X} \leq \omega \left( C \vert \log( \Vert L_{s,O}^1 - L_{s,O}^2 \Vert_{\widetilde{H}^{-s}_{\diamond}(O) \to H^s(O)}) \vert^{-\beta} \right),
\end{align*}
which finishes the proof.
\end{proof}

\section{Proof of the main results for the spectral fractional Laplacian}\label{sec:proof_spectral}

We next turn to the discussion of the main results for the fractional Calder\'on problem for the spectral fractional Laplacian.

\subsection{Uniqueness}\label{sec:uniqueness__local_StoSCalderon}

We begin by providing the arguments for the qualitative uniqueness result of Theorem \ref{thm:qualitative_spectral}.
Although there is essentially nearly no difference with respect to the observations from \cite{R23}, for the convenience of the reader, we present the full argument for Theorem \ref{thm:qualitative_spectral} (cf. Lemma 5.1 in \cite{R23}).

\begin{proof}[Proof of Theorem \ref{thm:qualitative_spectral}]
As in Lemma A.1 in \cite{R23} (c.f. Lemma \ref{lem:EigenfunctionExpansion} above), here using the Poisson kernel representation from Theorem 2.5 in \cite{CS16}, we have for $f \in C_c^\infty(O)$, $\tilde{c}_s = \frac{2^{1-s}}{\Gamma(s)}$ and $(x',x_{n+1}) \in \R^{n+1}_+$
\begin{align}\label{eq:eigenfunction_expansion_spectral}
\tilde{u}^f(x',x_{n+1}) = \tilde{c}_s \sum_{k=1}^\infty (f,\phi_k)_{L^2(\Omega)} K_s(\sqrt{\lambda_k}x_{n+1}) \left( \frac{x_{n+1}}{\sqrt{\lambda_k}} \right)^s \phi_k(x').
\end{align}
With this identity in hand we deduce for $x'\in \R^n$
\begin{align*}
\int_0^\infty t^{1-2s} \tilde{u}^f(x',t)dt &= \tilde{c}_s \sum_{k=1}^\infty (f,\phi_k)_{L^2(\Omega)} \int_0^\infty t^{1-2s} K_s(\sqrt{\lambda_k}t) \left( \frac{t}{\sqrt{\lambda_k}} \right)^s dt \ \phi_k(x')\\
&= \tilde{c}_s \left( \int_0^\infty z^{1-s} K_s(z) dz \right) \sum_{k=1}^\infty \lambda_k^{-1} (f,\phi_k)_{L^2(\Omega)} \phi_k(x')\\
&= \tilde{c}_s \left( \int_0^\infty z^{1-s} K_s(z) dz \right) (L_a^{-1}f)(x'),
\end{align*}
where we have changed variables and invoked the definition of the spectral Laplacian. By the asymptotics of the modified Bessel function of the second kind (see, for instance, \cite{Olver10}), it follows that for $s\in(0,1)$ the integrals $\int_0^\infty z^{1-s}K_s(z) dz$ are well-defined and finite. More precisely, it holds $\int_0^\infty z^{1-s}K_s(z) dz = 2^{-s}\Gamma(1-s)$. In particular, we find for $c_s = \frac{2^{2s-1}\Gamma(s)}{\Gamma(1-s)}$
\begin{align}\label{eq:ProofUniquenessEq1}
c_s \int_0^\infty t^{1-2s} \tilde{u}^f(\cdot,t) dt = L_a^{-1}f.
\end{align}

We now prove the density of $\widetilde{\mathcal{C}_1}$ in $\mathcal{C}_1$. Given the representation \eqref{eq:ProofUniquenessEq1} this follows directly by the density of $C_c^{\infty}(O)$ in $\widetilde{\mathcal{H}}^{-1}(O)$ and the continuity of the map $L_{a}^{-1}$ as a map from $\widetilde{\mathcal{H}}^{-1}(O)$ to $\mathcal{H}^{1}(\Omega)$.
\end{proof}

Before turning to the proof of Corollary \ref{cor:spectral_uniqueness}, we provide the following density result.

\begin{lem}\label{lem:density_StoS}
Let $\Omega$, $A$, $O$ be open, bounded, Lipschitz such that $A\Subset\Omega$ and $O \subset \Omega \setminus \overline{A}$ and let $a \in L^\infty(\Omega, \R^{n \times n})$ be a symmetric, uniformly elliptic matrix. Define the sets
\begin{align*}
S_1 (A)&:= \{w \in H^1(A): \ -\nabla \cdot a \nabla w = 0 \text{ in } A\},\\
S_2 (A)&:=\{\restr{v}{A} \in H^1(A): \ v \in H^1(\Omega), \ -\nabla \cdot a \nabla v = f \text{ in } \Omega, \ v = 0 \text{ on } \partial \Omega, \ f \in C_c^{\infty}(O) \}.
\end{align*}
Then the set $S_2(A)$ is dense in $S_1(A)$ with respect to the $H^1(A)$-topology.
\end{lem}

\begin{proof}
By definition, $S_2(A)$ is a subset of $S_1(A)$. We seek to prove that $S_2(A) \subset S_1(A)$ is dense with respect to the $H^1(A)$ topology. 

To this end, we define the following adjoint problem:
\begin{equation*}
\begin{cases}
\begin{alignedat}{2}
-\nabla \cdot a \nabla u & = h \chi_{A} \quad &&\text{in } \Omega,\\
u & = 0 \quad &&\text{on } \partial \Omega.
\end{alignedat}
\end{cases}
\end{equation*}
Here $h \in \widetilde{H}^{-1}(A)$. It suffices to prove that if for $h \in \widetilde{H}^{-1}(A)$ it holds $\langle v,h \rangle_{H^1(A),\widetilde{H}^{-1}(A)} =0$ for all $v \in S_2(A)$, then also $\langle w,h \rangle_{H^1(A),\widetilde{H}^{-1}(A)} =0$ for all $w \in S_1(A)$.

Indeed, the condition $\langle v,h \rangle_{H^1(A),\widetilde{H}^{-1}(A)} =0$ for all $v \in S_2(A)$, implies by integration by parts and the boundary conditions on $\partial \Omega$ that 
\begin{align*}
0 &= \langle v,h \rangle_{H^1(A),\widetilde{H}^{-1}(A)} 
= \langle v,h \chi_{A} \rangle_{H_{0}^1(\Omega),H^{-1}(\Omega)}
= \langle v, -\nabla \cdot a \nabla u\rangle_{H_{0}^1(\Omega),H^{-1}(\Omega)}\\
&= -\langle f, u \rangle_{H^{-1}(\Omega),H_0^1(\Omega)}.
\end{align*}
Since this identity is true for any $f \in C_c^\infty(O)$, it holds that $u = 0$ in $O$. By unique continuation we infer $u = 0$ in $\Omega \setminus A$, $h= -\nabla \cdot a \nabla (u|_{A}) $ and $u|_{\partial A} = 0 = \restr{\left(a \nabla u \cdot \nu\right)}{\partial A}$. Hence, again by integration by parts and the equation for $w \in S_1(A)$,
\begin{align*}
\langle w,h \rangle_{H^1(A),\widetilde{H}^{-1}(A)} 
&= \langle w, \nabla \cdot a \nabla (u|_{A}) \rangle_{H^1(A),\widetilde{H}^{-1}(A)} = \langle \nabla \cdot a  \nabla w, u|_{A} \rangle_{H^{-1}(A),H_0^1(A)}  = 0.
\end{align*}
This concludes the density proof.
\end{proof}

The uniqueness result is now an easy consequence of Theorem \ref{thm:qualitative_spectral} and Lemma \ref{lem:density_StoS}.

\begin{proof}[Proof of Corollary \ref{cor:spectral_uniqueness}]
By Theorem \ref{thm:qualitative_spectral}, since $\tilde{L}_{s,O}^1 = \tilde{L}_{s,O}^2$, we have $\tilde{L}_{1,O}^1 = \tilde{L}_{1,O}^2$. We seek to deduce that the local Dirichlet-to-Neumann maps $\Lambda_{\gamma_j}: H^{\frac{1}{2}}(\partial A) \rightarrow H^{-\frac{1}{2}}(\partial A)$, $j\in \{1,2\}$, with $g \mapsto \gamma_j \p_{\nu} w_j|_{\partial \Omega}$ agree. Here, $w_j \in H^1(A)$ is a solution of
\begin{equation*}
\begin{cases}
\begin{alignedat}{2}
\nabla \cdot \gamma_j \nabla w_j & = 0 \quad &&\text{in } A,\\
w_j & = g \quad &&\text{on } \partial A.
\end{alignedat}
\end{cases}
\end{equation*}
We claim that by the density result of Lemma \ref{lem:density_StoS} (with $a_j = \gamma_j \Id_{n \times n}$, $j \in \{1,2\}$) and the fact that $\tilde{L}_{1,O}^1 = \tilde{L}_{1,O}^2$, the equality of the local Dirichlet-to-Neumann maps follows. 

Indeed, firstly, by the density result from Lemma \ref{lem:density_StoS}, we have that for $j\in\{1,2\}$
\begin{align*}
\{v|_{\partial A}: \ v \in S_2^j(A)\} \subset H^{\frac{1}{2}}(A)
\end{align*}
is dense, where $S_2^j(A)$ is defined as the set $S_2(A)$ above with the metric $a_j$, $j\in \{1,2\}$, respectively. Hence, by continuity of the Dirichlet-to-Neumann map, also the inclusion
\begin{align}
\label{eq:dense_aux}
\{(v|_{\partial A}, \gamma_j \p_{\nu} v|_{\partial A}): \ v \in S_2^j(A)\} \subset \{(g, \Lambda_{\gamma_j}(g)): \ g \in H^{\frac{1}{2}}(\partial A)\}
\end{align}
is an inclusion of a dense set with respect to the $H^{\frac{1}{2}}(\partial A) \times H^{-\frac{1}{2}}(\partial A)$ topology for $j\in \{1,2\}$.

Secondly, by the facts that $\tilde{L}_{1,O}^1 = \tilde{L}_{1,O}^2$ and that $a_j = \text{Id}_{n\times n}$ in $\Omega \setminus A$, $j\in \{1,2\}$, by unique continuation, it follows that for any $f \in \widetilde{H}^{-1}(O)$ it holds for the solutions $v_1^f$ and $v_2^f$ of
\begin{equation*}
\begin{cases}
\begin{alignedat}{2}
\nabla \cdot \gamma_j \nabla v_j^f & = f \quad &&\text{in } \Omega,\\
v_j^f & = 0 \quad &&\text{on } \partial \Omega,
\end{alignedat}
\end{cases}
\end{equation*}
for $j \in \{1,2\}$,
that
\begin{align}
\label{eq:equal_ext} 
 v_1^f|_{\Omega \setminus A} =  v_2^f|_{\Omega \setminus A}.
 \end{align}
 This follows as, by definition, $v_1^f|_{O}=\tilde{L}_{1,O}^1(f) = \tilde{L}_{1,O}^2(f) = v_2^f|_{O}$ which then extends to the whole of $\Omega \setminus A$ by the unique continuation property since the metrics coincide outside of $A$, i.e., $a_1 = a_2 = \text{Id}_{n\times n}$ outside of $A$. In particular, from \eqref{eq:equal_ext} combined with \eqref{eq:dense_aux}, we finally deduce the equality of the associated Dirichlet-to-Neumann maps.

By virtue of the equality of the Dirichlet-to-Neumann maps, the uniqueness result then is a consequence of the uniqueness result from Theorem 1.1 in \cite{CR16} (see also \cite{SU87} and \cite{HT13}). 
\end{proof}

\subsection{Stability}

With the qualitative results in hand, we address the quantitative results of Theorem \ref{thm:quantitative_spectral} and of Corollary \ref{cor:spectral_stability}. 

We begin by proving the ``abstract'' statement about the quantitative transfer of uniqueness from the local source-to-solution Calderón problem to the fractional source-to-solution Calderón problem for the spectral fractional Laplacian, Theorem \ref{thm:quantitative_spectral}. In a second step, by making use of the stability of the local Calder\'on problem with source-to-solution data, Theorem \ref{thm:stability_local_StoSCalderon}, we will provide the proof of Corollary \ref{cor:spectral_stability}.

\begin{proof}[Proof of Theorem \ref{thm:quantitative_spectral}]
The proof proceeds as the one of Theorem \ref{thm:stability_nonlocal_local_single_meas}  but uses a combination of the spectral and of the real space Sobolev spaces.

More precisely, let $O''$ be such that $O' \Subset O'' \Subset O$. By the quantitative unique continuation arguments in the Caffarelli-Silvestre extension perspective in Sections \ref{sec:apriori}-\ref{sec:quant_UCP} above, we obtain for some $\beta>0$ small enough (depending on $s$, $n$, $O''$, $O$, $\theta_1$ and $\theta_2$ (from Assumptions (A1') and (A2'))
\begin{align}
\label{eq:spectral1}
\|\tilde{L}_{1,O}^1(f) - \tilde{L}_{1,O}^2(f)\|_{L^2(O'')} \leq C\vert\log(\varepsilon)\vert^{-\beta} \|f\|_{\mathcal{H}^{-s}(\Omega)}.
\end{align}
The constant $C$ depends on $s$, $n$, $O'$, $O$, $\theta_1$, $\theta_2$ and $\beta$.

Indeed, since the determination of the local source-to-solution operator from the nonlocal operator in the setting of the spectral Dirichlet Laplacian is identical to the one in the manifold setting (see Theorem \ref{thm:qualitative_spectral} vs. equation \eqref{eq:from_nonlocal_to_local_manifold}) and since we have analogous eigenfunction expansions for both settings (see equation \eqref{eq:eigenfunction_expansion_spectral} vs. Lemma \ref{lem:EigenfunctionExpansion}), we can apply the same arguments as before to derive an estimate as in Lemma \ref{lem:apriori}
\begin{align*}
&\Vert \int_L^\infty t^{1-2s} \tilde{u}^f(\cdot,t) dt \Vert_{\mathcal{H}^1(O)}\\
&\quad \leq \sup_{\substack{h \in \widetilde{\mathcal{H}}^{-1}(O),\\ \Vert h \Vert_{\mathcal{H}^{-1}(\Omega)} = 1}} \ \tilde{c}_s \left(\int\limits_{\sqrt{\lambda_1} L}^{\infty} z^{1-s} K_s(z) dz \right) \left( \sum_{k=1}^\infty \lambda_k^{-1} \vert (f,\phi_k)_{L^2(\Omega)} \vert^2 \right)^{1/2} \left( \sum_{k=1}^\infty \lambda_k^{-1} \vert (h,\phi_k)_{L^2(\Omega)} \vert^2 \right)^{1/2}\\
&\quad \leq C \left(\int\limits_{\sqrt{\lambda_1}L}^{\infty} z^{1-s} K_s(z) dz \right) \Vert f \Vert_{\mathcal{H}^{-1}(\Omega)}.
\end{align*}
Here, in the last inequality we have used the equivalence of the $\mathcal{H}^s$-norms (independently of the respective metric) as described in Section \ref{ssec:prel_equivalence_of_norms}. This estimate then transforms with the same arguments as in the manifold setting into an estimate like \eqref{eq:norm_upper_integral} for the upper integral with $O' \Subset O'' \Subset O$, i.e.
\begin{align*}
\left\Vert \int_{L(\varepsilon)}^\infty t^{1-2s} \left( \tilde{u}_1^f (\cdot,t) - \tilde{u}_2^f (\cdot,t) \right) dt \right\Vert_{L^2(O'')} \leq C \vert\log(\varepsilon)\vert^{-\beta} \Vert f \Vert_{\mathcal{H}^{-s}(\Omega)}.
\end{align*}
For the estimate of the remaining finite height integral (cf. \eqref{eq:claim_norm_lower_integral}) we first find that, using Poincaré's inequality and elliptic regularity, we are able to derive a similar estimate as in \eqref{eq:Poincare}, i.e.
\begin{align*}
\Vert \tilde{u}_j^f \Vert_{L^2(\Omega \times \R_+, x_{n+1}^{1-2s})} \leq C \Vert \tilde{u}_j^f \Vert_{\dot{H}^1(\Omega\times\R_+, x_{n+1}^{1-2s})} \leq C \Vert f \Vert_{\mathcal{H}^{-s}(\Omega)}.
\end{align*}
Secondly, we observe that the quantitative boundary-bulk unique continuation result, Proposition \ref{prop:bbucp}, and the three-balls-inequality, Proposition \ref{prop:3ballsinequ}, remain valid. Then the same arguments as in the proof of Proposition \ref{prop:stability_estimate} yield the estimate for the finite height integral, i.e.
\begin{align*}
\left\Vert \int_0^{L(\varepsilon)} t^{1-2s} \left( \tilde{u}_1^f (\cdot,t) - \tilde{u}_2^f (\cdot,t) \right) dt \right\Vert_{L^2(O'')} \leq C \vert\log(\varepsilon)\vert^{-\beta} \Vert f \Vert_{\mathcal{H}^{-s}(\Omega)},
\end{align*}

As before, we may combine \eqref{eq:spectral1} with Caccioppoli's inequality. Now using that in Caccioppoli's inequality one works with a zero extension and that hence the norm equivalence from \eqref{eq:norm_equivalence_2} in Section \ref{ssec:prel_equivalence_of_norms} is valid also in the range of exponents $\sigma\in [-2,2]$, we obtain that
\begin{align*}
\Vert \tilde{L}_{1,O}^1(f) - \tilde{L}_{1,O}^2(f) \Vert_{\mathcal{H}^{2-s}(O')} &\leq C \Vert \tilde{L}_{1,O}^1(f) - \tilde{L}_{1,O}^2(f) \Vert_{H^{2-s}(O')} \leq C \Vert \tilde{L}_{1,O}^1(f) - \tilde{L}_{1,O}^2(f) \Vert_{L^2(O'')}\\
&\leq C \vert\log(\varepsilon)\vert^{-\beta} \|f\|_{\mathcal{H}^{-s}(\Omega)}.
\end{align*}
Moreover, the elliptic estimates in the spectral spaces remain valid (with an identical argument in the $\mathcal{H}^\sigma$-spaces, also using the norm equivalences derived in Section \ref{ssec:prel_equivalence_of_norms}), which yields that
\begin{align*}
\|\tilde{L}_{1,O}^1(f) - \tilde{L}_{1,O}^2(f)\|_{\mathcal{H}^{1-s}(O')}
\leq C \|f\|_{\mathcal{H}^{-1-s}(\Omega)}.
\end{align*}
Finally, using that by interpolation it holds (also here, recall the norm equivalence results from Section \ref{ssec:prel_equivalence_of_norms})
\begin{align*}
\mathcal{H}^1(O') = (\mathcal{H}^{2-s}(O'), \mathcal{H}^{1-s}(O'))_{s,2} 
\quad \text{and} \quad \widetilde{\mathcal{H}}^{-1}(O) = (\widetilde{\mathcal{H}}^{-s}(O), \widetilde{\mathcal{H}}^{-1-s}(O))_{s,2},
\end{align*}
as above in the manifold setting, then allows us to infer via interpolation that
\begin{align*}
\|\tilde{L}_{1,O}^1(f) - \tilde{L}_{1,O}^2(f)\|_{\mathcal{H}^{1}(O')}
 \leq C \vert\log(\varepsilon)\vert^{-s\beta} \|f\|_{\mathcal{H}^{-1}(\Omega)}.
\end{align*}
This concludes the argument also in the spectral setting.
\end{proof}

With Theorem \ref{thm:quantitative_spectral} in hand, we turn to the proof of Corollary \ref{cor:spectral_stability}.
As an auxiliary result for this, we will rely on the following stability estimate for the local Calder\'on problem with source-to-solution data. 

\begin{thm}\label{thm:stability_local_StoSCalderon}
Let $n\geq 3$ and let $\Omega \subset \R^n$, $A \Subset \Omega$, $O \Subset \Omega \setminus \overline{A}$ and $\tilde{L}_{1,O}^j$, $j\in\{1,2\}$, be the sets and local source-to-solution operators associated to the metric $a_j = \gamma_j\Id_{n \times n}$ from Section \ref{sec:Intro_Spectral_StoSCalderon}. Assume that $\gamma_j \in H^{t+2}(\Omega)$ for some $t > \frac{n}{2}$ with
\begin{align*}
\Vert \gamma_j \Vert_{H^{t+2}(\Omega)} \leq M < \infty.
\end{align*}
Then there exist $C>0$ and $\sigma>0$ such that
\begin{align*}
\Vert \gamma_1 - \gamma_2 \Vert_{L^\infty(\Omega)} \leq C \vert \log(\Vert \tilde{L}_{1,O}^1 - \tilde{L}_{1,O}^2 \Vert_{\widetilde{H}^{-1}(O) \to H^1(O)}) \vert^{-\sigma}.
\end{align*}
\end{thm}

Since we were not able to find such a result in the literature, we will prove this stability result for the local source-to-solution Calderón problem in what follows below. However, we postpone the proof of the statement to Section \ref{sec:stability_local_StoSCalderon}. 

With the results of Theorems \ref{thm:quantitative_spectral} and \ref{thm:stability_local_StoSCalderon} in hand, we can now easily conclude the proof of Corollary \ref{cor:spectral_stability}.

\begin{proof}[Proof of Corollary \ref{cor:spectral_stability}]
Let $O' \Subset O$ be open, Lipschitz. It holds that
\begin{align*}
\Vert \gamma_1 - \gamma_2 \Vert_{L^\infty(\Omega)} &\leq C \vert \log(\Vert \tilde{L}_{1,O'}^1 - \tilde{L}_{1,O'}^2 \Vert_{\widetilde{H}^{-1}(O') \to H^1(O')}) \vert^{-\sigma}\\
&\leq C \vert \log( \vert \log (\Vert \tilde{L}_{s,O}^1 - \tilde{L}_{s,O}^2 \Vert_{\widetilde{\mathcal{H}}^{-s}(O) \to \mathcal{H}^s(O)} )\vert) \vert^{-\sigma},
\end{align*}
where the first inequality comes from Theorem \ref{thm:stability_local_StoSCalderon}, and for the second inequality we invoked the result of Theorem \ref{thm:quantitative_spectral} applied to all $f \in \widetilde{\mathcal{H}}^{-s}(O')$ with $\Vert f \Vert_{\mathcal{H}^{-s}(\Omega)} = 1$.
\end{proof}

\subsection{Proof of stability for the local source-to-solution Calderón problem}\label{sec:stability_local_StoSCalderon}

In this subsection we will now provide the proof of Theorem \ref{thm:stability_local_StoSCalderon}. Complementing the uniqueness proof in Section \ref{sec:uniqueness__local_StoSCalderon} by an assocaited stability result, we will reduce the stability question to a Dirichlet-to-Neumann map Calderón problem, for which such stability results are already known. The proof will rely on an $L^2$-based quantitative Runge approximation result (see Proposition \ref{prop:quantitative_Runge_approx}).

We will proceed in two steps.
First, we will prove the quantitative Runge approximation and then following up on that we provide the proof of the stability result, Theorem \ref{thm:stability_local_StoSCalderon}.

In some instances throughout this section we will use a slight abuse of notation. On the one hand, we will write at some points $(\cdot,\cdot)_{L^2(\Omega)}$ or $(\cdot,\cdot)_{L^2(\partial\Omega)}$ in order to denote the $H^1(\Omega)$, $H^{-1}(\Omega)$ or the $H^{1/2}(\partial\Omega)$, $H^{-1/2}(\partial\Omega)$ dual pairing, respectively.
On the other hand, for $w \in L^2(A)$ we write $w \chi_A \in L^2(\Omega)$ to denote the function
\begin{align*}
w \chi_A =
\begin{cases}
w \quad &\text{in } A,\\
0 \quad &\text{in } \Omega \setminus A.
\end{cases}
\end{align*}

Lastly, when writing $\partial_\nu^a w$ on $\partial\Omega$ we mean that $\partial_\nu^a w = \restr{\left(a \nabla w \cdot \nu\right)}{\partial\Omega}$, where $\nu$ is the outward unit normal to $\partial\Omega$.

\subsubsection{Quantitative Runge approximation}\label{sec:quantitative_Runge_approx}

The main ingredient for our proof will be a quantitative Runge approximation, see Proposition \ref{prop:quantitative_Runge_approx}. As this will not create any additional difficulties, for the first part we will work with a slightly more general operator than the one we are considering in Corollary \ref{cor:spectral_stability} and than what is actually needed for our result. Our proof follows the strategy from \cite{RS18a} very closely.

Let $\Omega \subset \R^n$, $A \Subset \Omega$ be open bounded Lipschitz domains such that $\Omega \setminus A$ is connected and $O \Subset \Omega \setminus \overline{A}$ open, bounded, Lipschitz. Consider the operator
\begin{align*}
\mathcal{L} := -\nabla\cdot a \nabla + q,
\end{align*}
where $a \in L^\infty(\Omega,\R^{n \times n}_{sym})$ is a symmetric matrix satisfying $\Vert \nabla a \Vert_{L^\infty(\Omega)} \leq M$ and the ellipticity assumption (A1') from above, and $q \in L^\infty(\Omega)$, with $\Vert q \Vert_{L^\infty(\Omega)} \leq M$ for some $M>0$, is such that $0$ is not a Dirichlet-eigenvalue. We note that $\mathcal{L}$ is formally self-adjoint, i.e. $\mathcal{L}^* = -\nabla\cdot a \nabla + q$.

Moreover, here we define the sets
\begin{align*}
S_1(A) &:= \{w \in H^1(A): \ \mathcal{L} w =0 \text{ in } A\},\\
S_2(A) &:= \{\restr{v^f}{A} \in H^1(A): \ v^f \in H^1(\Omega),\ \mathcal{L} v^f = f \text{ in } \Omega, \ v^f=0 \text{ on } \partial \Omega, \ f \in \widetilde{H}^{-1}(O)\}.
\end{align*}

In order to prove the quantitative Runge approximation, we will need the following preliminary result corresponding to Lemma 4.1 in \cite{RS18a}. Its proof is almost identical to the one in \cite{RS18a}.

\begin{lem}\label{lem:compactness_quantitative_Runge}
Let $\Omega$, $A$, $O$, $\mathcal{L}$, $S_1(A)$, $S_2(A)$ be the sets and operator as above. Denote by $X \subset L^2(A)$ the closure of $S_1(A)$ in $L^2(A)$ and define the operator $T: \widetilde{H}^{-1}(O) \to X$, $f \mapsto \restr{v^f}{A}$, where $v^f \in S_2(A)$ has the source term $f$.
Then, the operator $T$ is compact and has dense range.

The Banach space adjoint operator $T': X^* \to H^1(O)$ is defined by $T'w = \restr{h}{O}$ where $h$ and $w$ are related through
\begin{equation}\label{eq:adjoint_problem}
\begin{cases}
\begin{alignedat}{2}
\mathcal{L}^{*}h &= w \chi_A \quad &&\text{in } \Omega,\\
h &= 0 \quad &&\text{on } \partial \Omega.
\end{alignedat}
\end{cases}
\end{equation}
The Hilbert space adjoint operator is then given as $T^* = R_{H^1(O)} T' R_{X}: X \to \widetilde{H}^{-1}(O)$, where $R_H$ are the Riesz isomorphisms between a Hilbert space $H$ and its dual space. It holds that $\Vert T^*w \Vert_{\widetilde{H}^{-1}(O)} = \Vert T'w \Vert_{H^1(O)}$.

Moreover, there exists an orthonormal system $(\varphi_j)_{j\in\N}$ of $\widetilde{H}^{-1}(O)$ (such that $\widetilde{H}^{-1}(O) = \overline{\operatorname{span}}(\varphi_j)_{j\in\N} \bigoplus N(T)$, where $N(T)$ is the null space of $T$) and an orthonormal basis $(\psi_j)_{j\in\N}$ of $X$, such that $T\varphi_j = \sigma_j \psi_j$, where $\sigma_j>0$ are the positive singular values associated with the operator $T$.
\end{lem}

\begin{proof}
For compactness, let $(f_j)_{j\in\N}$ be a bounded sequence in $\widetilde{H}^{-1}(O)$. Then $\Vert Tf_j \Vert_{H^1(A)}$ is bounded. Hence, there exists an $H^1(A)$-weakly converging subsequence of $(Tf_j)$ to some $\overline{w} \in H^1(A)$. By Rellich's theorem this subsequence converges strongly in $L^2(A)$. Weak convergence implies that $\overline{w}$ solves $\mathcal{L}\overline{w} = 0$ in the weak sense, proving that $\overline{w} \in X$ and that $T$ is compact. It has dense range by the qualitative Runge approximation (the proof for the operator $\mathcal{L}$ follows exactly as in Lemma \ref{lem:density_StoS}, just replace the operator $\nabla\cdot a \nabla$ by the slightly more general operator $\mathcal{L}$).

We verify that the operator $T'$ as defined above is the Banach space adjoint operator. Let $f \in \widetilde{H}^{-1}(O)$ and let $w \in X$. Then, on the one hand, by definition of the adjoint operator
\begin{align*}
(f, T'w)_{L^2(\Omega)} = (Tf, w)_{L^2(A)} = (v^f, w)_{L^2(A)},
\end{align*}
and, on the other hand,
\begin{align*}
(f, T'w)_{L^2(\Omega)} = (\mathcal{L}v^f, T'w)_{L^2(\Omega)} = (v^f, \mathcal{L}^*(T'w))_{L^2(\Omega)}  + (\p_{\nu}^a v^f, T' w)_{L^2(\partial \Omega)}.
\end{align*}

A solution for this, and hence a characterization for $T'w$, is given by a weak solution to the equation
\begin{equation*}
\begin{cases}
\begin{alignedat}{2}
\mathcal{L}^*(T'w) & = w \quad &&\text{in } \Omega,\\
T' w & = 0 \quad &&\text{on } \partial \Omega.
\end{alignedat}
\end{cases}
\end{equation*}
This proves the claim on the form of the Banach space adjoint.

The statements about the Hilbert space adjoint operator follow from general functional analysis.

As a consequence, the operator $T^*T: \widetilde{H}^{-1}(O) \to \widetilde{H}^{-1}(O)$ is a compact, self-adjoint, positive semi-definite operator. By the spectral theorem, there exists an orthonormal system (not necessarily basis) $(\varphi_j)_{j\in\N}$ of $\widetilde{H}^{-1}(O)$ and a sequence of positive eigenvalues $(\mu_j)_{j\in\N}$ such that
\begin{align*}
T^* T \varphi_j = \mu_j \varphi_j.
\end{align*}
Define $\sigma_j := \mu_j^{1/2}$ and $\psi_j := \sigma_j^{-1} T \varphi_j \in X$. Then the $(\psi_j)_{j\in\N}$ are orthonormal by definition. Moreover, by another application of the qualitative Runge approximation result they are dense in $X$. Indeed, let $w \in X$ be such that $(w,\psi_j)_{L^2(A)} = 0$ for all $j\in\N$. We extend the orthonormal system $(\varphi_j)_{j\in\N}$ by an orthonormal basis of the null space of $T$, $(n^T_j)_{j\in\N} \subset \widetilde{H}^{-1}(O)$, to get an orthonormal basis $(\varphi_j)_{j\in\N} \cup (n^T_j)_{j\in\N}$ of $\widetilde{H}^{-1}(O)$. By density of $\operatorname{span}\left\{ (\varphi_j)_{j\in\N} \cup (n^T_j)_{j\in\N} \right\}$ in $\widetilde{H}^{-1}(O)$, 
it follows that $(w,Tf)_{L^2(A)} = 0$ for all $f \in \widetilde{H}^{-1}(O)$. Since $T$ has dense range in $X$ this in turn implies that $(w,v)_{L^2(A)} = 0$ for all $v \in X$. Choosing $v = w$ then entails that $w=0$. Hence, the functions $(\psi_j)_{j\in\N}$ indeed form an orthonormal basis and the proof is completed.
\end{proof}

Now, we are in the position to state and prove the quantitative Runge approximation result. It corresponds to Theorem 1.3 in \cite{RS18a}. The proof is again very similar to the one given there.

\begin{prop}\label{prop:quantitative_Runge_approx}
Let $\Omega$, $A$, $O$, $\mathcal{L}$, $S_1(A)$ and $S_2(A)$ be defined as above. Let $A_+$ be a bounded Lipschitz domain with $A \Subset A_+ \Subset \Omega$ such that $\Omega \setminus \overline{A_+}$ is connected. There are $C, \mu \geq 1$ (depending on $\Omega$, $A$, $A_+$, $O$, $n$, $\theta_1$, $M$) such that for any $w_+ \in S_1(A_+)$ and $\varepsilon \in (0,1)$ there exists $f \in \widetilde{H}^{-1}(O)$ and an associated function $v_{\varepsilon}^f \in S_2(A_+)$ such that
\begin{align*}
\Vert \restr{w_+}{A} - \restr{v_{\varepsilon}^f}{A} \Vert_{L^2(A)} \leq \varepsilon \Vert w_+ \Vert_{H^1(A_+)}, \qquad \Vert f \Vert_{H^{-1}(\Omega)} \leq C \varepsilon^{-\mu} \Vert \restr{w_+}{A} \Vert_{L^2(A)}.
\end{align*}
\end{prop}

\begin{proof}
Let $X$, $T$, $T'$, $T^*$, $(\sigma_j)_{j\in\N}$, $(\varphi_j)_{j\in\N}$ and $(\psi_j)_{j\in\N}$ be the space, operators, singular values and bases as in Lemma \ref{lem:compactness_quantitative_Runge}.

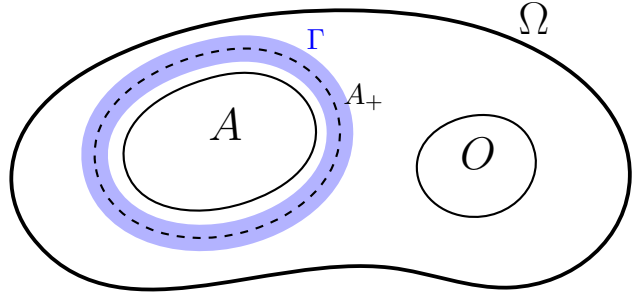
\begin{figure}
    	\begin{tikzpicture}[closed hobby, scale=0.35]
    	\draw[xshift = -3cm, yshift = -2cm, scale=1.5, line width = 0.05cm] plot coordinates {(1,1.5) (0.5,3) (3,6) (9,7.2) (14,6) (16,3) (13,0.2) (10,0.7)};
    	
    	\draw[thick] plot coordinates {(2,3) (4.5,6) (7,6.4) (9.2,4.3) (6.5,1.5)};
    	\draw[line width = 0.35cm, opacity = 0.3, blue] plot coordinates {(0.9,3) (4.5, 7) (7.2, 7.3) (10.1,4.3) (6.5,0.4)};
    	\draw[thick, dashed] plot coordinates {(0.9,3) (4.5, 7) (7.2, 7.3) (10.1,4.3) (6.5,0.4)};
	
    	\draw[thick] plot coordinates {(13,3) (14.6,4.7) (17.3, 3.8) (14.8,1)};
    	
    	\node at (17.4,8.5) {\huge{$\Omega$}};
    	\node[blue] at (9.2,7.7) {\large{$\Gamma$}};
    	\node at (5.8,4.5) {\huge{$A$}};
    	\node[dashed] at (11,5.5) {\large{$A_+$}};
    	\node at (15.3,3.4) {\huge{$O$}};
    	\end{tikzpicture}
\caption{Schematic set-up in the proof of the quantitative Runge approximation, Proposition \ref{prop:quantitative_Runge_approx}. Here, $A \Subset A_+ \Subset \Omega$ and $\Gamma$ is an open neighborhood of $\partial A_+$ with $\Gamma \Subset \Omega \setminus \overline{A}$.}
\label{fig:setting_proof_QRA}
\end{figure}

Take $w_+ \in S_1(A_+)$ and define $w := \restr{w_+}{A}$. Then $w \in S_1(A) \subset X$ and we write $w = \sum_{j=1}^\infty \beta_j \psi_j$. Define
\begin{align*}
f := \sum_{\sigma_j \geq \alpha} \frac{\beta_j}{\sigma_j} \varphi_j,
\end{align*}
where $\alpha > 0$ will be determined later. We will show that $v_{\varepsilon}^f := Tf$ satisfies the desired estimates. By orthonormality we have
\begin{align*}
\Vert f \Vert_{H^{-1}(\Omega)}^2 = \sum_{\sigma_j \geq \alpha} \frac{\beta_j^2}{\sigma_j^2} \leq \frac{1}{\alpha^2} \sum_{\sigma_j\geq \alpha} \beta_j^2 \leq \frac{1}{\alpha^2} \Vert w \Vert_{L^2(A)}^2.
\end{align*}
We set $r_\alpha := w-Tf = \sum_{\sigma_j < \alpha} \beta_j \psi_j$ and define $h_\alpha \in H^1(\Omega)$ as the solution to
\begin{equation*}
\begin{cases}
\begin{alignedat}{2}
\mathcal{L}^* h_\alpha &= r_\alpha \chi_A \quad &&\text{in } \Omega,\\
h_\alpha &= 0 \quad &&\text{on } \partial\Omega.
\end{alignedat}
\end{cases}
\end{equation*}
By orthogonality, integration by parts and trace estimates we infer
\begin{align}
\label{eq:error_bound1}
\begin{split}
\Vert Tf - w \Vert_{L^2(A)}^2 &= \Vert r_\alpha \Vert_{L^2(A)}^2 = (w, r_\alpha)_{L^2(A)} = (w_+, \mathcal{L}^*h_\alpha)_{L^2(A_+)}\\
&= (\mathcal{L}w_+, h_\alpha)_{L^2(A_+)} - (\partial_\nu^a w_+, h_\alpha)_{L^2(\partial A_+)} + (w_+, \partial_\nu^a h_\alpha)_{L^2(\partial A_+)}\\
&\leq \Vert \partial_\nu^a w_+ \Vert_{H^{-1/2}(\partial A_+)} \Vert h_\alpha \Vert_{H^{1/2}(\partial A_+)} + \Vert w_+ \Vert_{H^{1/2}(\partial A_+)} \Vert \partial_\nu^a h_\alpha \Vert_{H^{-1/2}(\partial A_+)}\\
&\leq C \Vert w_+ \Vert_{H^1(A_+)} \Vert h_\alpha \Vert_{H^1(\Gamma)},
\end{split}
\end{align}
where we have used that $\mathcal{L}w_+ = 0$ in $A_+$. Here, $\Gamma$ denotes an open neighborhood of $\partial A_+$ with $\Gamma \Subset \Omega \setminus \overline{A}$ (see Figure \ref{fig:setting_proof_QRA}).

Using the relation between $T^*$ and $T'$ and the fact that for the adjoint operator $T^*$ it holds that $T^*\psi_j = \sigma_j \varphi_j$ together with orthonormality, we derive 
\begin{equation}\label{eq:QRA_eq1}
\begin{aligned}
\Vert h_\alpha \Vert_{H^1(O)}^2 &= \Vert T'r_\alpha \Vert_{H^1(O)}^2 = \Vert T^* r_\alpha \Vert_{H^{-1}(\Omega)}^2 \leq \sum_{\sigma_j < \alpha} \beta_j^2 \Vert T^* \psi_j \Vert_{H^{-1}(\Omega)}^2 = \sum_{\sigma_j < \alpha} \sigma_j^2 \beta_j^2 \Vert \varphi_j \Vert_{H^{-1}(\Omega)}^2\\
&\leq \alpha^2 \Vert r_\alpha \Vert_{L^2(A)}^2.
\end{aligned}
\end{equation}
Now let $\Gamma_+$ be open, Lipschitz such that $\Gamma \Subset \Gamma_+ \subset \Omega \setminus \overline{A}$. By Caccioppoli's inequality and an iterative application of the three-balls-inequality (see e.g. Theorem 4.1 in \cite{ARRV09}) to propagate the smallness of $\Vert h_\alpha \Vert_{L^2(O)}$ from \eqref{eq:QRA_eq1} to $\Gamma_+$ we find that for some $t \in (0,1)$
\begin{equation}\label{eq:QRA_eq2}
\begin{aligned}
\Vert h_\alpha \Vert_{H^1(\Gamma)} \leq C \Vert h_\alpha \Vert_{L^2(\Gamma_+)} \leq C \Vert h_\alpha \Vert_{L^2(O)}^t \Vert h_\alpha \Vert_{L^2(\Omega)}^{1-t} \leq C \alpha^t \Vert r_\alpha \Vert_{L^2(A)}.
\end{aligned}
\end{equation}
Here, for the last inequality we used the elliptic a-priori estimate for $h_\alpha$. Combining \eqref{eq:error_bound1} and \eqref{eq:QRA_eq2}, we arrive at
\begin{align*}
\Vert Tf - w \Vert_{L^2(A)}^2 \leq C \Vert w_+ \Vert_{H^1(A_+)} \Vert r_\alpha \Vert_{L^2(A)} \alpha^t.
\end{align*}
Dividing by $\Vert r_\alpha \Vert_{L^2(A)} = \Vert Tf - w \Vert_{L^2(A)}$ and choosing $\alpha = \left(\varepsilon/C \right)^{\frac{1}{t}}$ then finishes the proof.
\end{proof}

\subsubsection{Proof of Theorem \ref{thm:stability_local_StoSCalderon}}\label{sec:stability_local_StoSCalderon_Schrödinger}

We will reduce the source-to-solution data to a Dirichlet-to-Neumann map Calderón problem, for which stability results are known. In that way we complement the uniqueness proof in Section \ref{sec:uniqueness__local_StoSCalderon} with a stability result.

We shortly recall and motivate the weak definition of the local source-to-solution operator $\tilde{L}_{1,O}: \widetilde{H}^{-1}(O) \to H^1(O)$ and provide some Alessandrini-type identities.

For $f \in \widetilde{H}^{-1}(O)$ let $v_j^f \in H_0^1(\Omega)$ be the unique solution to
\begin{equation}\label{eq:strong_StoS_equation}
\begin{cases}
\begin{alignedat}{2}
- \nabla \cdot a_j \nabla v_j^f &= f \quad &&\text{in } \Omega,\\
v_j^f &= 0 \quad &&\text{on } \partial\Omega.
\end{alignedat}
\end{cases}
\end{equation}
Here the metrics $a_j$, $j\in \{1,2\}$, are as in the formulation of the theorem.
In particular, in the weak formulation, for any $\phi \in H_0^1(\Omega)$ the function $v_j^f$ satisfies
\begin{align*}
\int_{\Omega} \nabla v_j^f \cdot a_j \nabla \phi dx = \int_{\Omega} f \phi dx.
\end{align*}
Formally, the associated source-to-solution operator is defined by
\begin{align}\label{eq:def_StoS_operator_conductivity}
\tilde{L}_{1,O}^j: \widetilde{H}^{-1}(O) \to H^1(O), \qquad f \mapsto \restr{v_j^f}{O},
\end{align}
where $v_j^f$ is a weak solution to \eqref{eq:strong_StoS_equation} with metric $a_j$ and source term $f$. In weak terms, we define $\tilde{L}_{1,O}^j$ for $f_1, f_2 \in \widetilde{H}^{-1}(O)$ by
\begin{align*}
\left( \tilde{L}_{1,O}^j f_1, f_2 \right)_{L^2(O)} := \int_\Omega \nabla v_j^f \cdot \bar{a} \nabla v_{\bar{a}}^{f_2} dx,
\end{align*}
where $v_j^f \in H_0^1(\Omega)$ is a weak solution to \eqref{eq:strong_StoS_equation} with metric $a_j$ and source term $f_1$ and $v_{\bar{a}}^{f_2} \in H_0^1(\Omega)$ is the unique weak solution to \eqref{eq:strong_StoS_equation} with metric $\bar{a}$ and source term $f_2$. Note that this weak definition is independent of the choice of $\bar{a}$ by the weak equation for $v_{\bar{a}}^{f_2}$ and since $v_j^{f_1} \in H_0^1(\Omega)$ is a suitable test function.

By standard arguments, we derive the Alessandrini-type identity
\begin{equation}\label{eq:Alessandrini_identity_StoS_conductivity}
\begin{aligned}
&\left( (\tilde{L}_{1,O}^1 - \tilde{L}_{1,O}^2) f_1, \ f_2 \right)_{L^2(O)} = \left( \tilde{L}_{1,O}^1 f_1, \ f_2 \right)_{L^2(O)} - \left( f_1, \ \tilde{L}_{1,O}^2 f_2 \right)_{L^2(O)}\\
&\qquad = \int_\Omega \nabla v_1^{f_1} \cdot a_2 \nabla v_2^{f_2} dx - \int_\Omega \nabla v_1^{f_1} \cdot a_1 \nabla v_2^{f_2} dx\\
&\qquad = \int_\Omega \nabla v_1^{f_1} \cdot (a_2 - a_1) \nabla v_2^{f_2} dx.
\end{aligned}
\end{equation}
\\

Moreover, let $A \subset \R^n$ be open, bounded, Lipschitz, $g \in H^{\frac{1}{2}}(\partial A)$ and consider the following conductivity equation for $w_j^g \in H^1(A)$
\begin{equation}\label{eq:strong_DtN_equation}
\begin{cases}
\begin{alignedat}{2}
- \nabla \cdot a_j \nabla w_j^g &= 0 \quad &&\text{in } A,\\
w_j^g &= g \quad &&\text{on } \partial A.
\end{alignedat}
\end{cases}
\end{equation}
The associated Dirichlet-to-Neumann map is defined by
\begin{align}\label{eq:def_DN_map_conductivity}
\Lambda_{a_j,A}: H^{\frac{1}{2}}(\partial A) \to H^{-\frac{1}{2}}(\partial A), \quad g \mapsto \partial_\nu^{a_j} \restr{w_j^g}{\partial A},
\end{align}
where $w_j^g \in H^1(A)$ is the unique weak solution to \eqref{eq:strong_DtN_equation}. In weak terms, for $g_1, g_2 \in H^{\frac{1}{2}}(\partial A)$ and $j \in \{1,2\}$, the operator $\Lambda_{a_j,A}$ is defined by
\begin{align*}
\left( \Lambda_{a_j,A} g_1, \ g_2 \right)_{L^2(\partial A)} := \int_A \nabla w_j^{g_1} \cdot a_j \nabla e^{g_2} dx,
\end{align*}
where $w_j^{g_1}$ is a solution to \eqref{eq:strong_DtN_equation} with metric $a_j$ and Dirichlet boundary data $g_1$, and $e^{g_2} \in H^1(\Omega)$ is any extension of $g_2$. Note that the definition is independent of the choice of the extension.
By similar considerations as above, the following Alessandrini-type identity holds
\begin{align}\label{eq:Alessandrini_identity_DN_conductivity}
\left( (\Lambda_{a_1,A}- \Lambda_{a_2,A}) g_1,\  g_2 \right)_{L^2(\partial A)} = \int_A \nabla w_1^{g_1} \cdot (a_1 - a_2) \nabla w_2^{g_2} dx.
\end{align}
The two Alessandrini identities, \eqref{eq:Alessandrini_identity_StoS_conductivity} and \eqref{eq:Alessandrini_identity_DN_conductivity}, show striking similarities, which we are going to exploit for the next proof.\\

Building on this Alessandrini identity, we present the proof of Proposition \ref{prop:reduction_StoS_DtN} and show that stability estimates for the source-to-solution Calderón problem can be reduced to stability estimates for the Dirichlet-to-Neumann map Calderón problem and vice versa.

\begin{proof}[Proof of Proposition \ref{prop:reduction_StoS_DtN}]
For abbreviation we will sometimes write $\Vert \tilde{L}_{1,O}^1 - \tilde{L}_{1,O}^2 \Vert_{\text{Op}}$ and $\Vert \Lambda_{a_1,A_+} - \Lambda_{a_2,A_+} \Vert_{\text{Op}}$ to denote $\Vert \tilde{L}_{1,O}^1 - \tilde{L}_{1,O}^2 \Vert_{\widetilde{H}^{-1}(O) \to H^1(O)}$ and $\Vert \Lambda_{a_1,A_+} - \Lambda_{a_2,A_+} \Vert_{H^{\frac{1}{2}}(\partial A_+) \to H^{-\frac{1}{2}}(\partial A_+)}$, respectively.

Since $a_1 = a_2$ in $\Omega \setminus A$, the Alessandrini-type identities \eqref{eq:Alessandrini_identity_StoS_conductivity} and \eqref{eq:Alessandrini_identity_DN_conductivity} imply that 
\begin{align*}
\left( (\tilde{L}_{1,O}^1 - \tilde{L}_{1,O}^2) f_1, \ f_2 \right)_{L^2(O)} &=
\int_\Omega \nabla v_1^{f_1} \cdot (a_2 - a_1) \nabla v_2^{f_2} dx = \int_{A_+} \nabla v_1^{f_1} \cdot (a_2 - a_1) \nabla v_2^{f_2} dx\\
&= - \left( (\Lambda_{a_1,A_+} - \Lambda_{a_2,A_+}) (\restr{v_1^{f_1}}{\partial A_+}), \ \restr{v_2^{f_2}}{\partial A_+} \right)_{L^2(\partial A_+)}.
\end{align*}
Then, we immediately infer the bound
\begin{align*}
&\Vert \tilde{L}_{1,O}^1 - \tilde{L}_{1,O}^2 \Vert_{\widetilde{H}^{-1}(O) \to H^1(O)}
= \sup_{\substack{f_1,f_2 \in \widetilde{H}^{-1}(O) \\ \Vert f_1 \Vert_{H^{-1}(\Omega)} = \Vert f_2 \Vert_{H^{-1}(\Omega)} = 1}} \left( (\tilde{L}_{1,O}^1 - \tilde{L}_{1,O}^2) f_1, \ f_2 \right)_{L^2(O)}\\
&\qquad = \sup_{\substack{f_1,f_2 \in \widetilde{H}^{-1}(O) \\ \Vert f_1 \Vert_{H^{-1}(\Omega)} = \Vert f_2 \Vert_{H^{-1}(\Omega)} = 1}} - \left( (\Lambda_{a_1,A_+} - \Lambda_{a_2,A_+}) (\restr{v_1^{f_1}}{\partial A_+}), \ \restr{v_2^{f_2}}{\partial A_+} \right)_{L^2(\partial A_+)}\\
&\qquad \leq \sup_{\substack{f_1,f_2 \in \widetilde{H}^{-1}(O) \\ \Vert f_1 \Vert_{H^{-1}(\Omega)} = \Vert f_2 \Vert_{H^{-1}(\Omega)} = 1}} \| \Lambda_{a_1,A_+} - \Lambda_{a_2,A_+}\|_{\text{Op}} \| \restr{v_1^{f_1}}{\partial A_+} \|_{H^{1/2}(\partial A_+)} \| \restr{v_2^{f_2}}{\partial A_+} \|_{H^{1/2}(\partial A_+)}\\
&\qquad \leq C\| \Lambda_{a_1,A_+} - \Lambda_{a_2,A_+}\|_{H^{\frac{1}{2}}(\partial A_+) \to H^{-\frac{1}{2}}(\partial A_+)}, 
\end{align*}
where in the last step, we used the a-priori estimates for $v_1^{f_1}$ and $v_2^{f_2}$, i.e. $\Vert \restr{v_j^{f_j}}{\partial A_+} \Vert_{H^{1/2}(\partial A_+)} \leq C \Vert v_j^{f_j} \Vert_{H^1(A_+)} \leq C \Vert f_j \Vert_{H^{-1}(\Omega)}$.

For the second inequality let $A_+' \subset \R^n$ be open, bounded, Lipschitz such that $A \Subset A_+' \Subset A_+$. We first use the Alessandrini identity \eqref{eq:Alessandrini_identity_DN_conductivity} and that $a_1 = a_2$ in $\Omega \setminus A$ to get
\begin{align}\label{eq:Alessandrini_observation}
\left( (\Lambda_{a_1,A_+} - \Lambda_{a_2,A_+}) g_1, \ g_2 \right)_{L^2(\partial A_+)} = \int_{A_+} \nabla w_1^{g_1} \cdot (a_1 - a_2) \nabla w_2^{g_2} dx = \int_A \nabla w_1^{g_1} \cdot (a_1 - a_2) \nabla w_2^{g_2} dx.
\end{align}
By the quantitative Runge approximation result, Proposition \ref{prop:quantitative_Runge_approx}, for $w_j^{g_j} \in H^1(A_+)$ there exists $f_j \in \widetilde{H}^{-1}(O)$ and $v_{j,\varepsilon}^{f_j} \in H_0^1(\Omega)$ solving \eqref{eq:strong_StoS_equation} with metric $a_j$ and source term $f_j$ such that
\begin{align*}
\Vert \restr{w_j^{g_j}}{A_+'} - \restr{v_{j,\varepsilon}^{f_j}}{A_+'} \Vert_{L^2(A_+')} \leq \varepsilon \Vert w_j^{g_j} \Vert_{H^1(A_+)}, \qquad \Vert f_j \Vert_{H^{-1}(\Omega)} \leq C \varepsilon^{-\mu} \Vert w_j^{g_j} \Vert_{L^2(A_+')}
\end{align*}
for some constants $C, \mu >0$. Since $-\nabla \cdot a_j \nabla (w_j^{g_j} - v_{j,\varepsilon}^{f_j}) = 0$ in $A_+$, we apply Caccioppoli's inequality to infer
\begin{align*}
\Vert \restr{w_j^{g_j}}{A} - \restr{v_{j,\varepsilon}^{f_j}}{A} \Vert_{H^1(A)} \leq C \Vert \restr{w_j^{g_j}}{A_+'} - \restr{v_{j,\varepsilon}^{f_j}}{A_+'} \Vert_{L^2(A_+')} \leq C \varepsilon \Vert w_j^{g_j} \Vert_{H^1(A_+)}.
\end{align*}
Using \eqref{eq:Alessandrini_observation} and the above approximation result we are able to infer
\begin{align*}
&\vert \left( (\Lambda_{a_1,A_+} - \Lambda_{a_2,A_+}) g_1, \ g_2 \right)_{L^2(\partial A_+)} \vert = \vert \left( \nabla w_1^{g_1}, \ (a_1 - a_2) \nabla w_2^{g_2} \right)_{L^2(A)} \vert\\
&\qquad \leq \vert \left( \nabla v_{1,\varepsilon}^{f_1}, \ (a_1 - a_2) \nabla v_{2,\varepsilon}^{f_2} \right)_{L^2(A)} \vert + C \varepsilon \Vert w_1^{g_1} \Vert_{H^1(A_+)} \Vert w_2^{g_2} \Vert_{H^1(A_+)}\\
&\qquad \leq \vert \left( \nabla v_{1,\varepsilon}^{f_1}, \ (a_1 - a_2) \nabla v_{2,\varepsilon}^{f_2} \right)_{L^2(\Omega)} \vert + C \varepsilon \Vert g_1 \Vert_{H^{\frac{1}{2}}(\partial A_+)} \Vert g_2 \Vert_{H^{\frac{1}{2}}(\partial A_+)}\\
&\qquad \leq \vert \left( (\tilde{L}_{1,O}^1 -\tilde{L}_{1,O}^2) f_1 , \ f_2 \right)_{L^2(O)} \vert + C \varepsilon \Vert g_1 \Vert_{H^{\frac{1}{2}}(\partial A_+)} \Vert g_2 \Vert_{H^{\frac{1}{2}}(\partial A_+)}\\
&\qquad \leq \Vert \tilde{L}_{1,O}^1 -\tilde{L}_{1,O}^2 \Vert_{\text{Op}} \Vert f_1 \Vert_{H^{-1}(\Omega)} \Vert f_2 \Vert_{H^{-1}(\Omega)} + C \varepsilon \Vert g_1 \Vert_{H^{\frac{1}{2}}(\partial A_+)} \Vert g_2 \Vert_{H^{\frac{1}{2}}(\partial A_+)}\\
&\qquad \leq C \Vert \tilde{L}_{1,O}^1 -\tilde{L}_{1,O}^2 \Vert_{\text{Op}} \varepsilon^{-2\mu} \Vert w_1^{g_1} \Vert_{L^2(A_+')} \Vert w_2^{g_2} \Vert_{L^2(A_+')} + C \varepsilon \Vert g_1 \Vert_{H^{\frac{1}{2}}(\partial A_+)} \Vert g_2 \Vert_{H^{\frac{1}{2}}(\partial A_+)}\\
&\qquad \leq C \Vert \tilde{L}_{1,O}^1 -\tilde{L}_{1,O}^2 \Vert_{\text{Op}} \varepsilon^{-2\mu} \Vert g_1 \Vert_{H^{\frac{1}{2}}(\partial A_+)} \Vert g_2 \Vert_{H^{\frac{1}{2}}(\partial A_+)} + C \varepsilon \Vert g_1 \Vert_{H^{\frac{1}{2}}(\partial A_+)} \Vert g_2 \Vert_{H^{\frac{1}{2}}(\partial A_+)},
\end{align*}
where we used the a-priori estimates for $w_j^{g_j}$ for several inequalities and the Alessandrini identity \eqref{eq:Alessandrini_identity_StoS_conductivity} for the fourth inequality. We optimize the right hand side in terms of $\Vert \tilde{L}_{1,O}^1 - \tilde{L}_{1,O}^2 \Vert_{\text{Op}}$ by choosing $\varepsilon = \Vert \tilde{L}_{1,O}^1 - \tilde{L}_{1,O}^2 \Vert_{\text{Op}}^{\frac{1}{1+2\mu}}$. Then taking the supremum in $\Vert g_1 \Vert_{H^{\frac{1}{2}}(\partial A_+)}=1$ and $\Vert g_2 \Vert_{H^{\frac{1}{2}}(\partial A_+)}=1$ yields
\begin{align*}
\Vert \Lambda_{a_1,A_+} - \Lambda_{a_2,A_+} \Vert_{H^{\frac{1}{2}}(\partial A_+) \to H^{-\frac{1}{2}}(\partial A_+)} \leq C \Vert \tilde{L}_{1,O}^1 - \tilde{L}_{1,O}^2 \Vert_{\widetilde{H}^{-1}(O) \to H^1(O)}^{\frac{1}{1+2\mu}}.
\end{align*}
This finishes the proof.
\end{proof}

With Proposition \ref{prop:reduction_StoS_DtN} at hand, we can now provide the proof of Theorem \ref{thm:stability_local_StoSCalderon} by combining this result with known stability estimates for the Dirichlet-to-Neumann map Calderón problem.

\begin{proof}[Proof of Theorem \ref{thm:stability_local_StoSCalderon}]
Let $A_+$ be Lipschitz and such that $A \Subset A_+ \Subset \Omega$, $\Omega \setminus \overline{A_+}$ is connected and $O \Subset \Omega \setminus \overline{A_+}$. It is well known, that under the given assumptions (see \cite{A88,Salo_lecturenotes} or \cite{Uhlmann09}) it holds that
\begin{align*}
\Vert \gamma_1 - \gamma_2 \Vert_{L^\infty(A_+)} \leq C \vert \log( \Vert \Lambda_{\gamma_1,A_+} - \Lambda_{\gamma_2,A_+} \Vert_{H^{\frac{1}{2}}(\partial A_+) \to H^{-\frac{1}{2}}(\partial A_+)} ) \vert^{-\sigma}.
\end{align*}
With this at hand, it is an immediate consequence of the fact that $\gamma_1 = \gamma_2$ in $\Omega \setminus A$ and Proposition \ref{prop:reduction_StoS_DtN} that
\begin{align*}
\Vert \gamma_1 - \gamma_2 \Vert_{L^\infty(\Omega)} &= \Vert \gamma_1 - \gamma_2 \Vert_{L^\infty(A_+)} \leq C \vert \log( \Vert \Lambda_{\gamma_1,A_+} - \Lambda_{\gamma_2,A_+} \Vert_{H^{\frac{1}{2}}(\partial A_+) \to H^{-\frac{1}{2}}(\partial A_+)} ) \vert^{-\sigma}\\
&\leq C \vert \log(\Vert \tilde{L}_{1,O}^1 - \tilde{L}_{1,O}^2 \Vert_{\widetilde{H}^{-1}(O) \to H^1(O)}) \vert^{-\sigma},
\end{align*}
which finishes the proof.
\end{proof}

\section{Remarks on optimality}\label{sec:instability}

Since the stability in the reconstruction problem for metrics in the nonlocal setting in the absence of a Liouville transform has not been discussed in the literature, we briefly comment on a lower bound for it in the spectral setting. For the setting of the condutivity equation which can be reduced to the fractional Schrödinger equation and a substantially more restrictive class of admissible geometries we refer to \cite{CRTZ24}. In our instability discussion, we do not consider the optimality of the reconstruction problem for the local source-to-solution map from the nonlocal one. Instead, we directly consider lower bounds for the reconstruction of (possibly) anisotropic metrics in the nonlocal problem from source-to-solution data and prove that this cannot be better than logarithmic. However, it is not clear whether this negative or our above positive results (e.g. from Corollary \ref{cor:spectral_stability}) are sharp, i.e., whether the modulus is not really double logarithmic rather than logarithmic.

Let us now assume that $\Omega \subset \R^n$ and $A \Subset \Omega$ are open, bounded, Lipschitz and connected such that $\Omega \setminus \overline{A}$ is connected. Let $O \Subset \Omega \setminus \overline{A}$ be open and Lipschitz. We consider uniformly elliptic and symmetric matrices $a \in C^1(\Omega, \R^{n \times n}_{sym})$ with $a(x) = \Id_{n \times n}$ for $x\in O_+$ where $O_+$ is an open, Lipschitz set such that $O \Subset O_+ \Subset \Omega \setminus \overline{A}$. In this framework the Caffarelli-Silvestre extension (see the discussion in Section \ref{sec:Intro_Spectral_StoSCalderon}) is given by
\begin{equation*}
\begin{cases}
\begin{alignedat}{2}
- \nabla \cdot x_{n+1}^{1-2s} \tilde{a} \nabla \tilde{u} &= 0 \quad &&\text{in } \Omega \times \R_+,\\
\tilde{u} &= 0 \quad &&\text{on } \partial\Omega \times \R_+,\\
-\bar{c}_s \lim_{x_{n+1}\to0} x_{n+1}^{1-2s} \partial_{n+1} \tilde{u} &= f \quad &&\text{on } \Omega \times \{0\},
\end{alignedat}
\end{cases}
\end{equation*}  
and the source-to-solution measurements then correspond to the Neumann-to-Dirichlet measurements for this problem:
\begin{align}\label{eq:source_to_solution_map_CSextension}
\tilde{L}_{s,O}^a: \widetilde{\mathcal{H}}^{-s}(O) \to \mathcal{H}^s(O), \qquad f \mapsto \restr{\tilde{u}(\cdot,0)}{O}.
\end{align}
Here $\tilde{a} = \begin{pmatrix} a & 0 \\ 0 & 1 \end{pmatrix}$.
In this setting, we will be able to apply the techniques of \cite{M01,KRS21} or \cite{BCR24} to show that the modulus of stability cannot be better than logarithmic. More precisely, the proof of this result will be a combination of the proof of Theorem 1.2 in \cite{BCR24} and a cut-off argument.

At this point we introduce some additional notation (c.f. Section 4 in \cite{BCR24}). For $D \subset \Omega$ open, Lipschitz and $R>0$ we denote $Q_{D,R} := D \times (0,R) \subset \R^{n+1}_+$ and $\Gamma_{D,R} := D \times [0,R) \subset \overline{\R^{n+1}_+}$. Additionally, we define by a slight abuse of notation
\begin{align*}
H_0^1(\Gamma_{D,R}, x_{n+1}^{1-2s}) := \{ \tilde{u} \in H^1(Q_{D,R}, x_{n+1}^{1-2s}): \ \tilde{u} = 0 \text{ on } (\partial D \times (0,R)) \cup (D\times\{R\}) \}.
\end{align*}
In the end, for us, $D \subset \Omega$ will be of the form $D = \Omega \setminus \overline{O}$.

Moreover, at some points in this section we write $\nabla'$ to denote the gradient with respect to the tangential variables in order to highlight the difference between the tangential and normal variables.

\subsection{Preliminary results}

Since it is an essential concept, in this section we first recall the definition and some properties of the entropy numbers $e_k(T)$ for a bounded linear operator $T$. 

\begin{defi}
Let $X$, $Y$ be Banach spaces and let $T: X \to Y$ be a bounded linear operator. We define for any $k\in\N$ the $k$-th entropy number $e_k(T)$ by
\begin{align*}
e_k(T) := \inf \left\{ \varepsilon>0: T(\overline{B}_1^X) \subset \bigcup_{j=1}^{2^{k-1}} (y_j + \varepsilon \overline{B}_1^Y) \text{ for some } y_1, \dots, y_{2^{k-1}} \in Y \right\}.
\end{align*}
\end{defi}

Here, $\overline{B}_r^X$ denotes the closed ball of radius $r$ in $X$.

\begin{lem}\label{lem:entropy_numbers_properties}
Let $X$, $Y$, $Z$ be Banach spaces and let $T,\overline{T}: X \to Y$, $S:Y \to Z$ be bounded linear operators. The following properties hold:
\begin{enumerate}[(i)]
\item (Lemma 3.8(iii) in \cite{KRS21}) For all $j,k \geq 1$ one has
\begin{align*}
e_{j+k-1}(T+\overline{T}) \leq e_j(T) + e_k(\overline{T}).
\end{align*}
\item (Lemma 3.8(iv) in \cite{KRS21}) $T$ is compact if and only if $e_k(T) \to 0$ as $k \to \infty$.
\item (see p.1089 in \cite{DS63}) Assume that $T$ and $S$ are compact and let $\sigma_k(T)$ denote the $k$-th singular value associated with the operator $T$. For all $j,k \geq 1$ one has
\begin{align*}
\sigma_{j+k-1} (S \circ T) \leq \sigma_j(S) \sigma_k(T).
\end{align*}
\end{enumerate}
\end{lem}

The following result is essentially taken from \cite{KRS21}. However, since at one point in our argument we need to be more precise about the change of the constants for one of the implications, we specify their dependence.

\begin{lem}[Lemma 3.9 in \cite{KRS21}]\label{lem:entropy_numbers_and_singular_values}
Let $X$, $Y$ be Hilbert spaces and let $T: X \to Y$ be compact. Let $\sigma_k(T)$ denote the singular values of $T$. For $\mu>0$ it holds
\begin{align*}
\sigma_k(T) \leq C e^{-ck^\mu} \text{ for some } C,c>0 \quad \Longleftrightarrow \quad e_k(T) \leq \tilde{C} e^{-\tilde{c}k^{\frac{\mu}{1+\mu}}} \text{ for some } \tilde{C},\tilde{c}>0.
\end{align*}
Moreover, if $L>0$ then
\begin{align*}
\sigma_k(T) \leq C e^{-cLk^\mu} \text{ for some } C,c>0 \quad \Longrightarrow \quad e_k(T) \leq \tilde{C} e^{-\tilde{c}L^{\frac{1}{\mu+1}} k^{\frac{\mu}{1+\mu}}} \text{ for some } \tilde{C},\tilde{c}>0.
\end{align*}
In this second statement, the constants $\tilde{C},\tilde{c}$ depend on $\mu, C, c$ but are independent of $L$.
\end{lem}

\begin{proof}
The first statement is given in Lemma 3.9 in \cite{KRS21}. The second statement is proven in exactly the same way as in the reference. We are just more precise about how the change of the constants depends on the parameter $L$.

We use the inequality (3.6) in \cite{KRS21}. It states that
\begin{align}\label{eq:entropy_singular_1}
\sup_{k \geq 1} e^{-\frac{N}{2k}} (\sigma_1 \cdots \sigma_k)^{\frac{1}{k}} \leq e_N(T) \leq C \sup_{k \geq 1} e^{-c_1 \frac{N}{k}} (\sigma_1 \cdots \sigma_k)^{\frac{1}{k}}.
\end{align}
If $\sigma_j$ satisfies $\sigma_j \leq C e^{-cL j^\mu}$, then evaluating a Riemann sum yields
\begin{align*}
(\sigma_1 \cdots \sigma_k)^{1/k} \leq C \exp(- \sum_{j=1}^k \frac{C}{k} L j^\mu ) \leq C \exp(-\frac{c}{\mu+1} Lk^\mu).
\end{align*}
Given $N$, the expression $e^{-c_1\frac{N}{k}} e^{-\frac{c}{\mu+1} Lk^\mu}$ is maximal when (in terms of $N$ and $L$) we choose $k \sim (N/L)^{\frac{1}{\mu+1}}$. Choosing this value of $k$ for the right hand side of \eqref{eq:entropy_singular_1} gives $e_N(T) \leq \tilde{C} e^{-\tilde{c} L^{\frac{1}{\mu+1}}N^{\frac{\mu}{\mu+1}}}$, which finishes the proof.
\end{proof}

The proof of the instability estimate will rely on the following abstract result from \cite{KRS21}, which allows us to infer the instability result ``by comparison" with a suitably defined auxiliary operator.

\begin{thm}[Theorem 5.14 in \cite{KRS21}]\label{thm:[KRS21]Thm5.14}
Let $X$ be a Banach space, $Y$ a separable Hilbert space and $F:X \to B(Y,Y^*)$ continuous. Let $X_1 \subset X$ be a closed subspace and let $K = \{ u \in X: \Vert u \Vert_{X_1} \leq r \}$ for some $r>0$. Assume that the embedding $i: X_1 \to X$ is compact with $e_k(i) \gtrsim k^{-m}$ for some $m>0$. Moreover, assume that there exists an orthonormal basis $(\varphi_j)_{j\in\N}$ of $Y$ and constants $C, \rho, \mu >0$ uniform over $u \in K$ such that $F(u)$ and $F(u)^*$ satisfy
\begin{align*}
\Vert F(u) \varphi_k \Vert_{Y^*}, \Vert F(u)^* \varphi_k \Vert_{Y^*} \leq C\exp(-\rho k^\mu), \qquad k\geq1.
\end{align*}
Then there exists $c>0$ such that for any $\varepsilon>0$ small enough there exist $u_1,u_2 \in K$ such that
\begin{align*}
\Vert u_1 - u_2 \Vert_X \geq \varepsilon, \quad \Vert u_j \Vert_{X_1} \leq r, \quad \Vert F(u_1) - F(u_2) \Vert_{B(Y,Y^*)} \leq \exp(-c\varepsilon^{-\frac{\mu}{m(\mu+2)}}).
\end{align*}
In particular, if one has the stability estimate
\begin{align*}
\Vert u_1 - u_2 \Vert_X \leq \omega( \Vert F(u_1) - F(u_2) \Vert_{B(Y,Y^*)}),
\end{align*}
then necessarily $\omega(t) \gtrsim \vert \log(t) \vert^{-\frac{m(\mu+2)}{\mu}}$ for $t>0$ small.
\end{thm}

As in the stability reduction result, we will here again split the integral into a lower and upper contribution and estimate them separately. Thus, as a first step, we prove an estimate for the upper part of the integral.

\begin{lem}\label{lem:instability_estimate_upper_integral}
Let $\Omega \Subset \R^n$ be open, bounded, Lipschitz and let $a$ and $\tilde{a}$ be as in the introduction of this section. Let $f \in \mathcal{H}^{-s}(\Omega)$ and let $\tilde{u}_0 \in H^1(\Omega\times\R_+, x_{n+1}^{1-2s})$ be the solution to
\begin{equation*}
\begin{cases}
\begin{alignedat}{2}
- \nabla\cdot x_{n+1}^{1-2s} \tilde{a} \nabla \tilde{u}_0 &= 0 \quad &&\text{in } \Omega\times\R_+,\\
\tilde{u}_0 &= 0 \quad &&\text{on } \partial\Omega \times \R_+,\\
-\bar{c}_s \lim_{x_{n+1}\to0} x_{n+1}^{1-2s} \partial_{n+1} \tilde{u}_0 &= f \quad &&\text{on } \Omega\times\{0\}.
\end{alignedat}
\end{cases}
\end{equation*}
Let $\lambda_1>0$ be the first eigenvalue value for the elliptic operator $(-\nabla'\cdot a \nabla')$ on $\Omega$ with Dirichlet boundary conditions. Then for $R>0$ large enough it holds
\begin{align*}
\Vert \tilde{u}_0 \Vert_{H^1(\Omega\times(R,\infty), x_{n+1}^{1-2s})} \leq C \Vert f \Vert_{\mathcal{H}^{-s}(\Omega)} e^{-\sqrt{\lambda_1} R}.
\end{align*}
\end{lem}

\begin{proof}
The proof follows similarly as the proofs of Lemma \ref{lem:apriori} and Lemma \ref{lem:estimate_integral_modified_Bessel_function}. Indeed, we use the representation in terms of the eigenfunction expansion (compare Lemma \ref{lem:EigenfunctionExpansion} or see the proof of Theorem \ref{thm:qualitative_spectral}), i.e. for $(x',x_{n+1}) \in \R^{n+1}_+$
\begin{align*}
\tilde{u}^f(x',x_{n+1}) = \tilde{c}_s \sum_{k=1}^\infty (f,\phi_k)_{L^2(\Omega)} K_s(\sqrt{\lambda_k}x_{n+1}) \left( \frac{x_{n+1}}{\sqrt{\lambda_k}} \right)^s \phi_k(x').
\end{align*}
On the one hand, we then have for the tangential derivative $\nabla'$ by changing variables $z = \sqrt{\lambda_k} t$
\begin{align*}
\int_R^\infty \int_\Omega &t^{1-2s} \vert \nabla' \tilde{u}_0(x',t) \vert^2 dx' dt \leq C \int_R^\infty t^{1-2s} \sum_{k=1}^\infty \vert (f, \phi_k)_{L^2(\Omega)} \vert^2 K_s^2(\sqrt{\lambda_k}t) \left( \frac{t}{\sqrt{\lambda_k}} \right)^{2s} \lambda_k dt\\
&\leq C \left( \int_{\sqrt{\lambda_1}R}^\infty z K_s^2(z) dz \right) \sum_{k=1}^\infty \lambda_k^{-s} \vert (f, \phi_k)_{L^2(\Omega)} \vert^2 \leq C \left( \int_{\lambda_1 R}^\infty z K_s^2(z) dz \right) \Vert f \Vert_{\mathcal{H}^{-s}(\Omega)}^{2}.
\end{align*}
{Here, for the first inequality we have used the uniform ellipticity condition on $\tilde{a}$.} On the other hand, for the normal derivative we use that by 10.29.2 in \cite{Olver10} it holds that
\begin{align*}
K_s'(z) = K_{1-s}(z) - \frac{s}{z} K_s(z),
\end{align*}
and thus
\begin{align*}
\partial_t \left( K_s(\sqrt{\lambda_k}t) \left( \frac{t}{\sqrt{\lambda_k}} \right)^s \right) = \lambda_k^{\frac{1-s}{2}} K_{1-s}(\sqrt{\lambda_k}t) t^s.
\end{align*}
As a consequence, we infer
\begin{align*}
\int_R^\infty \int_\Omega t^{1-2s} \vert \partial_t \tilde{u}_0 (x',t) \vert^2 dx' dt &= \tilde{c}_s^2 \int_R^\infty t^{1-2s} \sum	_{k=1}^\infty \vert (f, \phi_k)_{L^2(\Omega)} \vert^2 \left\vert \partial_t \left( K_s(\sqrt{\lambda_k}t) \left( \frac{t}{\sqrt{\lambda_k}} \right)^s \right) \right\vert^2 dt\\
&= \tilde{c}_s^2 \int_R^\infty t \sum_{k=1}^\infty \vert (f, \phi_k)_{L^2(\Omega)} \vert^2 \lambda_k^{1-s} K_{1-s}^2(\sqrt{\lambda_k}t) dt\\
&\leq \tilde{c}_s^2 \left( \int_{\sqrt{\lambda_1}R}^\infty z K_{1-s}^2(z) dz \right) \sum_{k=1}^\infty \lambda_k^{-s} \vert (f,\phi_k)_{L^2(\Omega)} \vert^2\\
&= \tilde{c}_s^2 \left( \int_{\sqrt{\lambda_1}R}^\infty z K_{1-s}^2(z) dz \right) \Vert f \Vert_{\mathcal{H}^{-s}(\Omega)}^{2}.
\end{align*}
Using, as before, that $K_s(z), K_{1-s}(z) \sim z^{-1/2} e^{-z}$ as $z\to\infty$ (see 10.25.3 in \cite{Olver10}), we find
\begin{align*}
\int_{\sqrt{\lambda_1} R}^\infty z K_s^2(z) dz \sim \int_{\sqrt{\lambda_1} R}^\infty e^{-2z} dz = \frac{1}{2}e^{-2\sqrt{\lambda_1} R}.
\end{align*}
Finally, we apply Poincaré's inequality (in tangential direction) to deduce
\begin{align*}
\Vert \tilde{u}_0 \Vert_{H^1(\Omega\times(R,\infty),x_{n+1}^{1-2s})} \leq C \Vert \nabla \tilde{u}_0 \Vert_{L^2(\Omega\times(R,\infty),x_{n+1}^{1-2s})} \leq \Vert f \Vert_{\mathcal{H}^{-s}(\Omega)} e^{-\sqrt{\lambda_1} R},
\end{align*}
which finishes the proof.
\end{proof}

We now recall an eigenvalue characterization for the Caffarelli-Silvestre extension. The proposition follows exactly like Proposition 2.1 in \cite{FF20} with the obvious modifications (compare also Proposition 4.1 in \cite{BCR24}).

\begin{prop}\label{prop:EigenfunctionForCSExtension}
Let $s\in(0,1)$, $R>0$ and $a \in C^1(\Omega, \R^{n \times n}_{sym})$ be as above. Let $D \subset \R^n$ be open, bounded and Lipschitz. Define for $(x',x_{n+1}) \in \R^{n+1}_+ $
\begin{align*}
\tilde{e}_{l,m}(x',x_{n+1}) := \gamma_m x_{n+1}^s J_{-s} \left( \frac{j_{-s,m}}{R}x_{n+1} \right) e_l(x') \quad \text{for any } l,m \in \N
\end{align*}
and
\begin{align*}
\lambda_{l,m} := \mu_l + \frac{j_{-s,m}^2}{R^2} \quad \text{for any } l,m \in \N,
\end{align*}
where $J_{-s}$ denotes the Bessel function of the first kind with order $-s$, $0<j_{-s,1}<j_{-s,2}<\dots<j_{-s,m}<\cdots$ are the zeros of $J_{-s}$, $\gamma_m := \left(\int_0^R x_{n+1} \left[J_{-s}(\frac{j_{-s,m}}{R}x_{n+1})\right]^2 dx_{n+1} \right)^{-\frac{1}{2}}$ is a normalizing constant, $(e_l)_{l\geq1}$ denotes a complete orthonormal system of eigenfunctions of $\left( -\nabla'\cdot a \nabla' \right)$ in $D$ with homogeneous Dirichlet boundary data and $\mu_1<\mu_2\leq\dots\leq\mu_l\leq\cdots$ the corresponding eigenvalues.\\
Then for any $l,m \in \N$, $\tilde e_{l,m}$ is an eigenfunction of
\begin{equation}\label{EqDefiningEigenfunction}
\begin{cases}
\begin{alignedat}{2}
-\nabla\cdot x_{n+1}^{1-2s} \tilde{a} \nabla \tilde{e}_{l,m} &= x_{n+1}^{1-2s} \lambda_{l,m} \tilde{e}_{l,m} \quad &&\text{in } Q_{D,R},\\
\tilde{e}_{l,m} &= 0 \quad &&\text{on } (\partial D \times (0,R)) \cup (D \times \{R\}),\\
\lim_{x_{n+1}\to0} x_{n+1}^{1-2s} \partial_{n+1} \tilde{e}_{l,m} &= 0 \quad &&\text{on } D \times \{0\},
\end{alignedat}
\end{cases}
\end{equation}
with corresponding eigenvalues $\lambda_{l,m}$ (interpreted in a weak sense). Moreover, the set $\{\tilde{e}_{l,m}: l,m \in \N\}$ is a complete orthonormal system for $L^2(Q_{D,R},x_{n+1}^{1-2s})$.
\end{prop}

As in Proposition 4.2 in \cite{BCR24} or Proposition 2.7 in \cite{KRS21}, we seek to identify the weighted Sobolev spaces with corresponding weighted sequence spaces.

To this end, we relabel the eigenvalues and eigenfunctions from above and define these as $(\lambda_k, \tilde{\varphi}_k) \subset \R_+ \times L^2(Q_{D,R}, x_{n+1}^{1-2s})$ to denote the ordered pairs of eigenvalues and eigenfunctions $(\lambda_{l,m}, \tilde{e}_{l,m})$ from Proposition \ref{prop:EigenfunctionForCSExtension} such that $0 < \lambda_1 \leq \lambda_2 \leq \dots \leq \lambda_k \leq \cdots$. We have that $\tilde{\varphi}_k \in H_0^1(\Gamma_{D,R}, x_{n+1}^{1-2s})$. They are orthonormal in $L^2(Q_{D,R}, x_{n+1}^{1-2s})$, and by their defining equation they are also orthogonal in $H^1(Q_{D,R}, x_{n+1}^{1-2s})$. In what follows below, we will work with this basis.

The following result coincides with Proposition 4.2 from \cite{BCR24}. We are just more precise about the dependence of the asymptotic behaviour of the eigenvalues $\lambda_k$ on the cylinder height $R$.

\begin{prop}
Let $D\subset\R^n$ be open, bounded, Lipschitz and $R>0$. Let $(\lambda_k, \tilde{\varphi}_k) \subset \R_+ \times H_0^1(\Gamma_{D,R}, x_{n+1}^{1-2s})$ be as above. Then for any $\tilde{u} \in H_0^1(\Gamma_{D,R}, x_{n+1}^{1-2s})$ we have the norm equivalence
\begin{align*}
\Vert \tilde{u} \Vert_{H^1(\Gamma_{D,R},x_{n+1}^{1-2s})}^2 \sim \sum_{k=1}^\infty R^{-\frac{2}{n+1}} k^{\frac{2}{n+1}} \vert (\tilde{u}, \varphi_k)_{L^2(Q_{D,R},x_{n+1}^{1-2s})} \vert^2.
\end{align*}
Moreover,
\begin{align*}
H_0^1(\Gamma_{D,R}, x_{n+1}^{1-2s}) = \left\{ \tilde{u} \in L^2(Q_{D,R},x_{n+1}^{1-2s}) : \ \sum_{k=1}^\infty R^{-\frac{2}{n+1}} k^{\frac{2}{n+1}} \vert (\tilde{u}, \tilde{\varphi}_k)_{L^2(Q_{D,R},x_{n+1}^{1-2s})} \vert^2 < \infty \right\}.
\end{align*}
In particular, there exists an isomorphism between $H_0^1(\Gamma_{D,R}, x_{n+1}^{1-2s})$ and the sequence space
\begin{align*}
h^1 := \left\{ (a_k)_{k\in\N} \in \ell^2 : \ \sum_{k=1}^\infty R^{-\frac{2}{n+1}} k^{\frac{2}{n+1}} \vert a_k \vert^2 < \infty \right\},
\end{align*}
given by $\tilde{u} \mapsto ( (\tilde{u}, \tilde{\varphi}_k)_{L^2(Q_{D,R},x_{n+1}^{1-2s})} )_{k\in\N}$.
\end{prop}

The proof is almost identical to the one of Proposition 4.2 in \cite{BCR24}. We provide it for the convenience of the reader.

\begin{proof}
\textit{Step 1: Weyl type asymptotic of the eigenvalues.} We first deduce the asymptotic behaviour of the eigenvalues $\lambda_k$. We recall that by Proposition \ref{prop:EigenfunctionForCSExtension} the eigenvalues $\lambda_{l,m}$ were given by
\begin{align*}
\lambda_{l,m} = \mu_l + \frac{j_{-s,m}^2}{R^2}.
\end{align*}
By Weyl asymptotics (cf. Chapter 8 Corollary 3.5 in \cite{Taylor2010}) and McMahon's asymptotic expansion (cf. Section 10.21 (vi) in \cite{Olver10}) we have that $\mu_l \sim l^{\frac{2}{n}}$ and $j_{-s,m}^2 \sim m^2$, respectively, which implies
\begin{align*}
\lambda_{l,m} \sim l^{\frac{2}{n}} + \frac{m^2}{R^2}.
\end{align*}
Now, fix $N\in\N$ sufficiently large. For fixed $l \leq N^{\frac{n}{2}}$, there are $R (N-l^{2/n})^{1/2}$ elements such that $l^{2/n} + \frac{m^2}{R^2} \leq N$. On the one hand, summing over $l$ we find that the number of pairs $(l,m)$ with $l^{2/n} + \frac{m^2}{R^2} \leq N$ is bounded from above by
\begin{align*}
\sum_{l=1}^{N^{n/2}} R (N-l^{2/n})^{1/2} \leq R \sum_{l=1}^{N^{n/2}} N^{1/2} \leq R N^{\frac{n+1}{2}},
\end{align*}
and, on the other hand, it is bounded from below by
\begin{align*}
\sum_{l=1}^{N^{n/2}} R (N-l^{2/n})^{1/2} \geq R \sum_{l=1}^{\frac{1}{2}N^{n/2}} (N - \frac{1}{2^{2/n}}N)^{1/2} \geq C R N^{\frac{n+1}{2}}. 
\end{align*}
Thus, there are approximately $RN^{\frac{n+1}{2}}$ pairs $(l,m)$ such that $l^{2/n} + \frac{m^2}{R^2} \leq N$, which in turn implies that
\begin{align*}
\lambda_k \sim R^{-\frac{2}{n+1}}k^{\frac{2}{n+1}}.
\end{align*}
Recall that $\lambda_k$ were the reordered eigenvalues.

\textit{Step 2: Characterization as a sequence space.} This step follows exactly as in \cite[Proposition 4.2]{BCR24}. We just briefly summarize it. Based on Step 1, we first deduce the norm equivalence. Then we show the identification of spaces using Step 1 and the norm equivalence.
\end{proof}

As in Section 4.1 in \cite{BCR24}, we next derive the following singular value bounds for the embedding $i: H^1(Q_{D,R}, x_{n+1}^{1-2s}) \to L^2(Q_{D,R}, x_{n+1}^{1-2s})$. Compare this statement to Proposition 4.4 in \cite{BCR24}.

\begin{prop}\label{prop:instability_singular_value_bounds}
Let $D \subset \R^n$ be open, bounded, Lipschitz and $R>0$. The embedding $i: H^1(Q_{D,R}, x_{n+1}^{1-2s}) \to L^2(Q_{D,R}, x_{n+1}^{1-2s})$ satisfies the singular value bounds
\begin{align*}
\sigma_k(i) \leq CR^{\frac{1}{n+1}} k^{-\frac{1}{n+1}}
\end{align*}
for some constant $C>0$.
\end{prop}

\begin{proof}
Since the proof is almost identical to the one of Proposition 4.3 and 4.4 in \cite{BCR24}, we rather just give a short outline of how to prove this. In a first step we show that the embedding $i': H_0^1(\Gamma_{D,R},x_{n+1}^{1-2s}) \to L^2(Q_{D,R},x_{n+1}^{1-2s})$ satisfies the singular value bounds $\sigma_k(i') \sim R^{\frac{1}{n+1}} k^{-\frac{1}{n+1}}$. We do this by using the isomorphism of $H_0^1(\Gamma_{D,R},x_{n+1}^{1-2s})$ to $h^1$ from the previous proposition and the isomorphism of $L^2(Q_{D,R},x_{n+1}^{1-2s})$ to $\ell^2$ and by using the diagonal structure of their identification. In a second step we then transfer the singular value bound for $i'$ to a singular value bound for $i$ by an extension and restriction argument.
\end{proof}

\subsection{Compression estimate for an auxiliary operator}

Before we state and prove the instability result for the fractional source-to-solution Calderón problem, we first derive a compression estimate for an auxiliary comparison operator. Our strategy will be to use a similar cut-off argument as in the stability reduction result. We first consider bounded domains and then use the exponential decay from Lemma \ref{lem:instability_estimate_upper_integral} to obtain an estimate for unbounded domains.

\begin{prop}\label{prop:instability_compression_bounded_domain}
Let $\Omega$ be open, bounded, Lipschitz. Let $O, O_+ \Subset \Omega$ be open, Lipschitz such that $O \Subset O_+$ and let $R>0$ be sufficiently large. Let $\widetilde{T}: \widetilde{\mathcal{H}}^{-s}(O) \to H^1((\Omega \setminus \overline{O_+}) \times (0,R), x_{n+1}^{1-2s})$ be defined by
\begin{align*}
\widetilde{T}(f) = \restr{\tilde{u}_0}{(\Omega \setminus \overline{O_+})\times(0,R)}
\end{align*}
where $\tilde{u}_0 \in H^1(\Omega \times \R_+, x_{n+1}^{1-2s})$ is the solution to
\begin{equation*}
\begin{cases}
\begin{alignedat}{2}
-\nabla\cdot x_{n+1}^{1-2s} \nabla \tilde{u}_0 &= 0 \quad &&\text{in } \Omega \times \R_+,\\
\tilde{u}_0 &= 0 \quad &&\text{on } \partial\Omega \times \R_+,\\
-\bar{c}_s \lim_{x_{n+1}\to0} x_{n+1}^{1-2s} \partial_{n+1} \tilde{u}_0 &= f \quad &&\text{on } \Omega \times \{0\}.
\end{alignedat}
\end{cases}
\end{equation*}
There exist constants $c,C>0$ such that
\begin{align*}
\sigma_k(\widetilde{T}) \leq C \exp \left( -c \left(\frac{k}{R}\right)^{\frac{1}{n+2}} \right).
\end{align*}
Moreover, there exists $\overline{c},\overline{C}>0$ such that
\begin{align*}
e_k(\widetilde{T}) \leq \overline{C}  \exp \left( - \overline{c} R^{-\frac{1}{n+3}} k^{\frac{1}{n+3}} \right).
\end{align*}
\end{prop}

\begin{proof}
The strategy here is along the same lines as the argument in the proof of Theorem 1.2 in \cite{BCR24}.

We set $d := \dist(O, \partial O_+)$. For $N\in\N$ suitably determined below we consider a sequence of nested sets $O_j$ such that $O=O_0 \subset O_1 \subset \dots \subset O_{N} =O_+$ and $\dist(\overline{O_j}, \partial O_{j+1}) \geq \frac{d}{2N}$ and a sequence of $R_j$ such that $R_j = R+\frac{d}{N}(N-j)$ for $j\in\{0,1,\dots,N\}$. Define $D_j := \Omega \setminus \overline{O_j}$. In this way we have $Q_{D_{j+1},R_{j+1}} \subset Q_{D_j,R_j}$ with $\dist(Q_{D_{j+1},R_{j+1}}, \partial Q_{D_j,R_j}) \geq \frac{d}{2N}$.

We then set
\begin{align*}
\widetilde{T} = T_N \circ T_{N-1} \circ \dots \circ T_1 \circ T_{\text{in}},
\end{align*}
where the operators in the factorization are defined as follows:
\begin{itemize}
\item The initial operator $T_{\text{in}}$ is given by
\begin{align*}
T_{\text{in}}: \widetilde{\mathcal{H}}^{-s}(O) \to H^1(Q_{D_0,R_0}, x_{n+1}^{1-2s}), \qquad f \mapsto \restr{\tilde{u}_0}{Q_{D_0,R_0}}.
\end{align*}
\item The iteration operators $T_j$ for $j\in\{1,\dots,N\}$ are defined by
\begin{align*}
T_j: H^1(Q&_{D_{j-1},R_{j-1}}, x_{n+1}^{1-2s}) \to L^2(Q_{D_{j-1},R_{j-1}}, x_{n+1}^{1-2s}) \to H^1(Q_{D_j,R_j}, x_{n+1}^{1-2s}),\\
&\restr{\tilde{u}_0}{Q_{D_{j-1},R_{j-1}}} \mapsto \restr{\tilde{u}_0}{Q_{D_{j-1},R_{j-1}}} \mapsto \restr{\tilde{u}_0}{Q_{D_j,R_j}}.
\end{align*}
For this we applied Caccioppoli's inequality (see Proposition 4.5 in \cite{BCR24} with $q=0$) which yields that
\begin{align*}
\Vert \tilde{u}_0 \Vert_{H^1(Q_{D_j,R_j}, x_{n+1}^{1-2s})} \leq C \frac{2N}{d} \Vert \tilde{u}_0 \Vert_{L^2(Q_{D_{j-1},R_{j-1}}, x_{n+1}^{1-2s})}.
\end{align*}
Relying on the singular value bounds from Proposition \ref{prop:instability_singular_value_bounds} we infer
\begin{align*}
\sigma_k(T_j) \leq C \frac{2N}{d} R^{\frac{1}{n+1}} k^{-\frac{1}{n+1}}.
\end{align*}
Here we used that $R_j \sim R$ for all $j \in \{0,1,\dots,N\}$.
\end{itemize}
Knowing that the singular values for $T_{\text{in}}$ are bounded by a constant, we use the multiplicative property of the singular values (recall Lemma \ref{lem:entropy_numbers_properties} (i)) to obtain the overall entropy estimate for the operator $\widetilde{T}$
\begin{align*}
\sigma_k(\widetilde{T}) &\leq \sigma_1(T_{\text{in}}) \prod_{j=1}^N \sigma_{k/N}(T_j) \leq C \left( CR^{\frac{1}{n+1}} \left(\frac{2N}{d}\right) \left( \frac{k}{N} \right)^{-\frac{1}{n+1}} \right)^N \leq C \left( \frac{2C}{d} R^{\frac{1}{n+1}} \frac{N^{\frac{n+2}{n+1}}}{k^{\frac{1}{n+1}}} \right)^N.
\end{align*}
Optimizing $N$ in terms of $k$ and $R$, we choose for some $\rho>0$ small enough
\begin{align*}
N = \rho R^{-\frac{1}{n+2}} k^{\frac{1}{n+2}}.
\end{align*}
This indeed yields that
\begin{align*}
\sigma_k(\widetilde{T}) \leq C \left( C\rho^{\frac{n+2}{n+1}} \right)^{\rho R^{-\frac{1}{n+2}}k^{\frac{1}{n+2}}} \leq C \exp\left( -c \left( \frac{k}{R} \right)^{\frac{1}{n+2}} \right)
\end{align*}
for some appropriately chosen constants $C,c>0$.

The statement about the entropy numbers $e_k(\widetilde{T})$ then follows by an application of the second implication of Lemma \ref{lem:entropy_numbers_and_singular_values} with $L=R^{-\frac{1}{n+2}}$ and $\mu=\frac{1}{n+2}$ noting that the operator $\widetilde{T}$ is compact.
\end{proof}

With the above results for the upper and lower integral bounds in hand, we now combine these by optimizing in the parameter $R>0$.

\begin{prop}\label{prop:instability_compression_unbounded_domain}
Let $\Omega$, $O$, $O_+$ and $\tilde{u}_0$ be as in Proposition \ref{prop:instability_compression_bounded_domain}. Define $T: \widetilde{\mathcal{H}}^{-s}(O) \to H^1((\Omega \setminus \overline{O_+}) \times \R_+, x_{n+1}^{1-2s})$ by
\begin{align*}
T(f) = \restr{\tilde{u}_0}{(\Omega \setminus \overline{O_+})\times\R_+}.
\end{align*}
There exist constants $c,C>0$ such that
\begin{align}\label{eq:instability_singular_value_bound}
\sigma_k(T) \leq C \exp \left( -c k^{\frac{1}{n+3}} \right).
\end{align}
\end{prop}

\begin{proof}
Let $R>0$ to be determined later. We know by Lemma \ref{lem:instability_estimate_upper_integral} that
\begin{align}\label{eq:instabilty_estimate_upper_integral}
\Vert \tilde{u}_0 \Vert_{H^1(\Omega\times(R,\infty), x_{n+1}^{1-2s})} \leq C \Vert f \Vert_{\mathcal{H}^{-s}(\Omega)} e^{-cR}.
\end{align}
In particular, by \cite[Lemma 3.8(i)]{KRS21} this entails that the operator $\bar{T}:\widetilde{\mathcal{H}}^{-s}(O) \to H^1((\Omega \setminus \overline{O_+})\times(R,\infty), x_{n+1}^{1-2s})$ satisfies the entropy bound
\begin{align*}
e_k(\bar{T}) \leq e_1(\bar{T})  \leq \|\bar{T}\|_{Op} \leq C e^{-c R},
\end{align*}
where $\|\bar{T}\|_{Op} := \|\bar{T}\|_{\widetilde{\mathcal{H}}^{-s}(O) \to H^1((\Omega \setminus \overline{O_+})\times(R,\infty), x_{n+1}^{1-2s})}$.

Next, let $\widetilde{T}: \widetilde{\mathcal{H}}^{-s}(O) \to H^1((\Omega \setminus \overline{O_+})\times(0,R), x_{n+1}^{1-2s})$ be defined as in the previous proposition. It holds that
\begin{align*}
e_k(\widetilde{T}) \leq \overline{C} \exp \left(-\overline{c}R^{-\frac{1}{n+3}}k^{\frac{1}{n+3}}\right) =: \delta_{R,k}.
\end{align*}

By the additive properties of entropy numbers, see e.g. \cite[Lemma 3.8(iii)]{KRS21}, we obtain that for $T= \widetilde{T} + \bar{T}$ it holds
\begin{align*}
e_k(T) = e_k(\widetilde{T} + \bar{T}) \leq e_{[k/2]}(\widetilde{T}) + e_{[k/2]}(\bar{T}) \leq \delta_{R,[\frac{k}{2}]} + C e^{-c R} ,
\end{align*}
where $[\cdot]$ denotes the floor operation.

We optimize $R$ in terms of $k$ by choosing $R = k^{\frac{1}{n+4}}$ (observe that for $k$ large, $R$ will also be large). Thus, in total we get
\begin{align*}
e_k(T) \leq C \exp \left( -c k^{\frac{1}{n+4}} \right).
\end{align*}
Since $e_k(T) \to 0$ as $k\to\infty$, by Lemma \ref{lem:entropy_numbers_properties}, $T$ is compact and, thus, by Lemma \ref{lem:entropy_numbers_and_singular_values}, we are able to infer
\begin{align*}
\sigma_k(T) \leq C \exp \left( -c k^{\frac{1}{n+3}} \right).
\end{align*}
This finishes the proof.
\end{proof}

\subsection{Instability result for spectral fractional source-to-solution Calderón problem}

With the preliminary work in our hands, we derive the following instability estimate.

\begin{prop}
Let $\Omega, A, O \subset \R^n$ be smooth, open, bounded, connected sets such that $A \Subset \Omega$, $\Omega \setminus \overline{A}$ is connected and $O \Subset \Omega \setminus \overline{A}$. Let $r_0 >0$ and $\theta \in (0,1)$ and define $K:= \{ a \in C^1(\Omega, \R^{n \times n}): \Vert a - \Id_{n \times n} \Vert_{C^1(\Omega)} \leq r_0, \ \theta \leq a \leq \theta^{-1}, \ a = \Id_{n \times n} \text{ in } \Omega \setminus A\}$. Let $\tilde{L}_{s,O}^a$ be the source-to-solution maps as in \eqref{eq:def_nonlocal_StoSOperator_spectral}.

There exists $c>0$ such that for any $\varepsilon \in (0,r_0)$ small enough, there exists $a_1, a_2 \in K$ with
\begin{align*}
\Vert \tilde{L}_{s,O}^{a_1} - \tilde{L}_{s,O}^{a_2} \Vert_{\widetilde{\mathcal{H}}^{-s}(O) \to \mathcal{H}^s(O)} &\leq \exp\left( -c\varepsilon^{-\frac{1}{2+\frac{7}{n}}} \right),\\
\Vert a_1 - a_2 \Vert_{L^\infty(\Omega)} &\geq \varepsilon.
\end{align*}
In particular, if one has the stability property
\begin{align*}
\Vert a_1 - a_2 \Vert_{L^\infty(\Omega)} \leq \omega( \Vert \tilde{L}_{s,O}^{a_1} - \tilde{L}_{s,O}^{a_2} \Vert_{\widetilde{\mathcal{H}}^{-s}(O) \to \mathcal{H}^s(O)} ),
\end{align*}
then necessarily $\omega(t) \gtrsim \vert \log t \vert^{-(2+\frac{7}{n})}$ for $t$ small.
\end{prop}

The proof relies on the by now standard comparison technique already used in, e.g., \cite{M01, KRS21, BCR24}.

\begin{proof}
We seek to apply Theorem \ref{thm:[KRS21]Thm5.14}. It is known that the embedding $i: C^1(\Omega, \R^{n \times n}) \to L^\infty(\Omega, \R^{n \times n})$ satisfies $e_k(i) \gtrsim k^{-1/n}$ (see for example Lemma 2 in \cite{M01} with slight adaptations).

We define the operator $\Gamma(a) := \tilde{L}_{s,O}^a - \tilde{L}_{s,O}^{\Id_{n \times n}}: \widetilde{\mathcal{H}}^{-s}(O) \to \mathcal{H}^s(O)$. Let $\tilde{u}^f$ be a solution to
\begin{equation*}
\begin{cases}
\begin{alignedat}{2}
-\nabla\cdot x_{n+1}^{1-2s} \tilde{a} \nabla \tilde{u}^f &= 0 \quad &&\text{in } \Omega \times \R_+,\\
\tilde{u}^f &= 0 \quad &&\text{on } \partial\Omega\times\R_+,\\
-\bar{c}_s \lim_{x_{n+1}\to0} x_{n+1}^{1-2s} \partial_{n+1} \tilde{u}^f &= f \quad &&\text{on } \Omega \times \{0\}.
\end{alignedat}
\end{cases}
\end{equation*}
Recall that $\tilde{a} = \begin{pmatrix} a & 0 \\ 0 & 1 \end{pmatrix}$. We let $\tilde{u}^f_0$ be the solution with $a=\Id_{n \times n}$.  Consequently, $\tilde{w}:= \tilde{u}^f-\tilde{u}_0^f$ solves
\begin{equation*}
\begin{cases}
\begin{alignedat}{2}
-\nabla\cdot x_{n+1}^{1-2s} \tilde{a} \nabla \tilde{w} &= - \nabla \cdot x_{n+1}^{1-2s} (\Id_{(n+1) \times (n+1)} - \tilde{a}) \nabla \tilde{u}_0^f \quad &&\text{in } \Omega \times \R_+,\\
\tilde{w} &= 0 \quad &&\text{on } \partial\Omega\times\R_+,\\
\lim_{x_{n+1}\to0} x_{n+1}^{1-2s} \partial_{n+1} \tilde{w} &= 0 \quad &&\text{on } \Omega \times \{0\}.
\end{alignedat}
\end{cases}
\end{equation*}
We note that $\Gamma(a)f := (\tilde{L}_{s,O}^a - \tilde{L}_{s,O}^{\Id_{n \times n}})f = \restr{\left(\tilde{u}^f(\cdot,0) - \tilde{u}_0^f(\cdot,0)\right)}{O} = \restr{\tilde{w}(\cdot,0)}{O}$. Let $O_+$ be open, Lipschitz such that $O \Subset O_+ \Subset \Omega \setminus A$. Then we have the following sequence of inequalities
\begin{equation}\label{eq:instability_comparison_estimate}
\begin{aligned}
\Vert \tilde{w} (\cdot,0) &\Vert_{\mathcal{H}^s(O)} \leq \Vert \tilde{w}(\cdot,0) \Vert_{\mathcal{H}^s(\Omega)} \leq C \Vert \tilde{w} \Vert_{H^1(\Omega\times\R_+, x_{n+1}^{1-2s})}\\
&\leq C \Vert \nabla \cdot x_{n+1}^{1-2s} (\Id_{(n+1) \times (n+1)} - \tilde{a}) \nabla \tilde{u}_0^f \Vert_{H^{-1}(\Omega\times\R_+,x_{n+1}^{1-2s})}\\
&\leq C \Vert (\Id_{(n+1) \times (n+1)} - \tilde{a}) \nabla \tilde{u}_0^f \Vert_{L^2(\Omega\times\R_+, x_{n+1}^{1-2s})}\\
&\leq C \Vert (\Id_{(n+1) \times (n+1)} - \tilde{a}) \nabla \tilde{u}_0^f \Vert_{L^2((\Omega \setminus \overline{O_+}) \times \R_+,x_{n+1}^{1-2s})}\\
&\leq C \Vert \tilde{a}-\Id_{(n+1) \times (n+1)} \Vert_{L^\infty((\Omega \setminus \overline{O_+}) \times \R_+)} \Vert \tilde{u}_0^f \Vert_{H^1((\Omega \setminus \overline{O_+}) \times \R_+, x_{n+1}^{1-2s})}\\
&\leq C \Vert \tilde{u}_0^f \Vert_{H^1((\Omega \setminus \overline{O_+}) \times \R_+, x_{n+1}^{1-2s})},
\end{aligned}
\end{equation}
where we first have used trace estimates (see Theorem 2.5 in \cite{ST10}), the well-posedness estimate for $\tilde{w}$ and then the support condition on $a-\Id_{n \times n}$. Here the constant $C$ is uniform over $K$.

Let $T: \widetilde{\mathcal{H}}^{-s}(O) \to H^1((\Omega \setminus \overline{O_+})\times\R_+, x_{n+1}^{1-2s})$, $f \mapsto \restr{\tilde{u}_0^f}{(\Omega \setminus \overline{O_+})\times\R_+}$ be the operator as in Proposition \ref{prop:instability_compression_unbounded_domain} and let $(\phi_k)_{k \in \N}$ be the singular value basis for $T$, of which the corresponding eigenvalues satisfy the estimate \eqref{eq:instability_singular_value_bound}. By the comparison argument from \eqref{eq:instability_comparison_estimate} and by Proposition \ref{prop:instability_compression_unbounded_domain} we have
\begin{align*}
\Vert \Gamma(a) \phi_k \Vert_{\mathcal{H}^s(O)} \leq C \Vert \tilde{u}_0^{\phi_k} \Vert_{H^1((\Omega \setminus \overline{O_+})\times\R_+, x_{n+1}^{1-2s})} \leq C \exp\left(-\overline{c}k^{\frac{1}{n+3}}\right).
\end{align*}
Since $\Gamma(a)$ is self-adjoint, the same bound holds for the adjoint operator $\Gamma(a)^*$. Then, applying Theorem \ref{thm:[KRS21]Thm5.14} with $F = \Gamma$, $K$ as above, $m = \frac{1}{n}$, $\rho = \overline{c}$ and $\mu = \frac{1}{n+3}$ yields the existence of some $c>0$ such that for any $0 < \varepsilon <r_0$ small enough there are metrics $a_1, a_2 \in K$ with
\begin{align*}
\Vert L_{s,O}^{a_1} - L_{s,O}^{a_2} \Vert_{\widetilde{\mathcal{H}}^{-s}(O) \to \mathcal{H}^s(O)} &\leq \exp\left( -c\varepsilon^{-\frac{1}{2+\frac{7}{n}}} \right),\\
\Vert a_1 - a_2 \Vert_{L^\infty(\Omega)} &\geq \varepsilon.
\end{align*}
This concludes the proof.
\end{proof}

\begin{rmk}
We expect that with the same strategy one can also prove a corresponding instability result for the source-to-solution Calderón problem in the closed manifold setting but, due to the length of the article, we do not discuss this further at this point.
\end{rmk}

\section*{Acknowledgements}
Both authors gratefully acknowledge funding by the Deutsche Forschungsgemeinschaft (DFG, German Research Foundation) under Germany's Excellence Strategy -- EXC-2047/1 -- 390685813 and through the CRC 1720.

\bibliographystyle{alpha}
\bibliography{bibliography}

\end{document}